\documentclass{amsart}

\usepackage{graphicx}
\usepackage{amssymb}
\usepackage{amsmath}
\usepackage{amsfonts}
\usepackage{array}
\usepackage{multirow}
\usepackage{makecell}
\usepackage{hhline}
\usepackage{multirow,bigdelim}
\usepackage{mathtools}
\usepackage{tabu}
\usepackage[table]{xcolor}
\usepackage{arydshln}
\usepackage[ruled,vlined,linesnumbered]{algorithm2e}
\usepackage{hyperref}
\usepackage{float}
\usepackage{upquote}
\usepackage{comment}

\usepackage{diagbox}
\usepackage[table]{xcolor}

\newtheorem{theorem}{Theorem}[section]
\newtheorem{corollary}{Corollary}[theorem]
\newtheorem{lemma}[theorem]{Lemma}
\theoremstyle{definition}
\newtheorem{definition}[theorem]{Definition}

\theoremstyle{remark}
\newtheorem{remark}[theorem]{Remark}

\newcommand{\RN}[1]{%
  \textup{\uppercase\expandafter{\romannumeral#1}}%
}

\newcommand{\gl}[1]{\Gl(#1)}
\newcommand{\ortho}[1]{\Ortho(#1)}
\newcommand{\un}[1]{\Un(#1)}
\newcommand{\symp}[1]{\Symp(#1)}
\newcommand{\osp}[1]{\Osp(#1)}
\newcommand{\usp}[1]{\Usp(#1)}
\newcommand{\twotwo}[4]{\begin{bmatrix}#1 & #2\\#3 & #4\end{bmatrix}}
\newcommand{\twoone}[2]{\begin{bmatrix}#1 \\ #2\end{bmatrix}}
\newcommand{\threediag}[5]{\begin{bmatrix} #1&&#2\\&#3&\\#4&&#5\end{bmatrix}}
\newcommand{\sthreediag}[5]{\begin{bsmallmatrix}#1 & & #2 \\ &#3&\\#4&&#5\end{bsmallmatrix}}
\newcommand{\stwotwo}[4]{\begin{bsmallmatrix}#1 & #2 \\ #3 & #4\end{bsmallmatrix}}
\newcommand{\twotworr}[4]{
\begin{bmatrix}\begin{array}{rr}
#1 & #2\\#3 & #4
\end{array}\end{bmatrix}}
\newcommand{\mis}{\,\hspace{-0.1cm}\scalebox{0.5}[0.7]{\( - \)}}

\newcommand{\realify}[1]{[#1]_\mathbb{R}}
\newcommand{\complexify}[1]{[#1]_\mathbb{C}}
\newcommand{\glmat}[2]{G_{#1}^\mathbb{#2}}
\newcommand{\unmat}[2]{U_{#1}^\mathbb{#2}}
\newcommand{\sympmat}[2]{Sp_{#1}^\mathbb{#2}}
\newcommand{\orthomat}[2]{O_{#1}^\mathbb{#2}}
\newcommand{\fact}[1]{\mathcal{F}_{#1}}

\newcommand{\myscale}[2]{\text{\scalebox{#1}[#1]{$#2$}}}

\hypersetup{
    colorlinks=false,
    linkcolor=black,
    filecolor=magenta,      
    urlcolor=cyan,
}

\DeclareMathOperator{\Ortho}{O}
\DeclareMathOperator{\Symp}{Sp}
\DeclareMathOperator{\Un}{U}
\DeclareMathOperator{\Gl}{GL}
\DeclareMathOperator{\diag}{diag}
\DeclareMathOperator{\Osp}{OSp}
\DeclareMathOperator{\Usp}{USp}
\newcommand{\kak}{\text{K}_1\text{AK}_2}
\numberwithin{equation}{section}

\let\oldnl\nl
\newcommand{\nonl}{\renewcommand{\nl}{\let\nl\oldnl}}

\makeatletter
\renewcommand*\env@matrix[1][\arraystretch]{%
  \edef\arraystretch{#1}%
  \hskip -\arraycolsep
  \let\@ifnextchar\new@ifnextchar
  \array{*\c@MaxMatrixCols c}}
\makeatother

\SetKwInOut{Inputone}{Input 1}
\SetKwInOut{Inputtwo}{Input 2}

\usepackage[obeyFinal,textsize=footnotesize]{todonotes}

\begin{document}

\title[New Matrix Factorizations]{Fifty Three Matrix Factorizations: \\
A systematic approach}

\author{Alan Edelman}
\address{Department of Mathematics and Computer Science \& AI Laboratory, Massachusetts Institute of Technology, Cambridge, Massachusetts, 02139}
\email{edelman@mit.edu}

\author{Sungwoo Jeong}
\address{Department of Mathematics, Massachusetts Institute of Technology, Cambridge, Massachusetts, 02139}
\email{sw2030@mit.edu}


\begin{abstract}
The success of matrix factorizations such as the singular value decomposition (SVD) has motivated the search for even more factorizations. We catalog 53 matrix factorizations, most of which we believe to be new. 

Our systematic approach, inspired by the generalized Cartan decomposition of Lie theory, also encompasses known factorizations such as the SVD, the symmetric eigendecomposition, the CS decomposition, the hyperbolic SVD, structured SVDs, the Takagi factorization, and others thereby covering familiar matrix factorizations as well as ones that were waiting to be discovered.

We suggest that Lie theory has one way or another been lurking hidden in the foundations of the very successful field of matrix computations with applications routinely used in so many areas of computation. In this paper, we investigate consequences of the Cartan decomposition and the little known generalized Cartan decomposition for matrix factorizations.

We believe that these factorizations once properly identified can lead to further work on algorithmic computations and applications.
\end{abstract}

\maketitle
{\small
\tableofcontents}
\addtocontents{toc}{\protect\setcounter{tocdepth}{1}}

\vspace{-0.5cm}
\section{Introduction}
\subsection{Matrix factorizations}
An insightful exercise might be to ask what is the most important idea in linear algebra. Our answer would be ``matrix factorizations." It is as straightforward as it is unlikely to be the one we think you might hear, such as: ``linearity," ``eigenvalues" or ``singular values." Matrix factorizations have shown up relatively recently in linear algebra classes, when before they may have taken a backseat to the study of structure. However, they are critical to applications in both pure/applied mathematics, as well as science and engineering.\footnote{Indeed, Townsend and Trefethen \cite{townsend2015continuous} (\href{https://people.maths.ox.ac.uk/trefethen/publication/PDF/2014_151.pdf}{\color{blue}online version only}) highlight the critical role of matrix factorizations stating that ``one might regard this as the central dogma of classical numerical linear algebra: matrix algorithms correspond to matrix factorizations."}

\par The search for new matrix factorizations has motivated numerical linear algebraists for decades. The search has often been driven by a specific demand, a particular problem or sometimes a generalization of a known factorization. While we enjoy learning specialized new factorizations, we were motivated to understand a unifying structure. 

\begin{table}
{\tabulinesep=0.3mm
\setlength{\tabcolsep}{2pt}
\title{\large Table of 53 Matrix Factorizations \vspace{0.05in} } 

\tiny
\centering
\begin{tabu}{|c|c|c|c|c|}
\cline{1-5}
& & \small $\mathbb{R}$ & \small $\mathbb{C}$ & \small $\mathbb{H}$ \\
\noalign{\global\arrayrulewidth=1pt}\hhline{-|-|-|-|-|}\noalign{\global\arrayrulewidth=0.4pt}
    
\multirow{10}{*}{\rotatebox[origin=c]{90}{$\Un_\beta(n)$ (Compact)\hspace{0.8cm}}} &
\cellcolor{brown!20} & & \cellcolor{gray!25} $U_n = O_n D {O_n}'$ & \cellcolor{gray!25}$\unmat{n}{H} = U_n D {U_n}'$ \\

& \multirow{-2}{*}{\cellcolor{brown!20}$\fact{1}$} & & \cellcolor{gray!25}{\color{red} ODO (2013, 2018)} \cite{vilenkin2013representation,fuehr2018note} & \cellcolor{gray!25}{\color{red} Bloch, Messiah (1962)} \cite{bloch1962canonical} \\
\hhline{~|-|-|-|-|}

& \cellcolor{brown!20}$\fact{2}$ & $O_{2n} = \realify{\unmat{n}{C}}\stwotwo{R}{}{}{R^{-1}}{\realify{\unmat{n}{C}}}'$ & $U_{2n}=\complexify{\unmat{n}{H}}\stwotwo{D}{}{}{D} {\complexify{\unmat{n}{H}}}'$ & \\ \cline{2-5}
    
& $\fact{3}$ & & $U_{2n} = O_{2n}\stwotwo{D}{}{}{D}\complexify{\unmat{n}{H}}$ & \\ \hhline{~|-|-|-|-|}
    
& & \cellcolor{gray!25}$O_n = $ & \cellcolor{gray!25}$U_n = $ & $\unmat{n}{H} = $ \\

& & \cellcolor{gray!25}$\stwotwo{O_p}{}{}{O_q}\stwotwo{C}{S}{-S}{C}\stwotwo{O_r}{}{}{O_s}$ & \cellcolor{gray!25}$\stwotwo{U_p}{}{}{U_q}\stwotwo{C}{S}{-S}{C}\stwotwo{U_r}{}{}{U_s}$ & \multirow[t]{2}{*}{$\stwotwo{\unmat{p}{H}}{}{}{\unmat{q}{H}}\stwotwo{C}{S}{-S}{C}\stwotwo{\unmat{r}{H}}{}{}{\unmat{s}{H}}$} \\ 

& \cellcolor{brown!20} & \multicolumn{2}{c|}{\cellcolor{gray!25}{\color{red}CSD $(p,q)=(r,s)$ (1968) Wigner} \cite{wigner1968generalization}, {\color{red}Davis, Kahan} \cite{davis1970rotation}} &  \\

& \multirow{-4}{*}{$\fact{4}$\vspace{0.1cm}} & \multicolumn{2}{c|}{\cellcolor{gray!25}{\color{red}CSD general Davis, Kahan (1969)} \cite{davis1969some}, {\color{red}GSVD (1981) }\cite{paige1981towards}} & \\ 
\cline{2-5}

& $\fact{5}$ & & $U_n =O_n\stwotwo{C}{iS}{iS}{C}\stwotwo{U_p}{}{}{U_q}$ & $\unmat{n}{H}=U_n\stwotwo{C}{jS}{jS}{C}\stwotwo{\unmat{p}{H}}{}{}{\unmat{q}{H}}$ \\ \cline{2-5}

& $\fact{6}$ & $O_{2n} = \realify{\unmat{n}{C}}\Theta\stwotwo{O_{2p}}{}{}{O_{2q}}$ & $U_{2n} = \complexify{\unmat{n}{H}}\Theta\stwotwo{U_{2p}}{}{}{U_{2q}}$ & \\ 
\noalign{\global\arrayrulewidth=1pt}\hhline{-|-|-|-|-|}\noalign{\global\arrayrulewidth=0.4pt}

\multirow{10}{*}{\rotatebox[origin=c]{90}{$\Gl_\beta(n)$\hspace{0.3cm}}} &
\cellcolor{brown!20} & \cellcolor{gray!25}$\glmat{n}{R} = O_n\Sigma O_n'$ &
\cellcolor{gray!25} $\glmat{n}{C} = U_n\Sigma U_n'$ & 
\cellcolor{gray!25} $\glmat{n}{H} = \unmat{n}{H}\Sigma {\unmat{n}{H}}'$ \\ 

& \multirow{-2}{*}{\cellcolor{brown!20}{$\fact{7}$}} & \multicolumn{3}{c|}{\cellcolor{gray!25}{\color{red}Singular value decomposition}} \\
\cline{2-5}

& $\fact{8}$ & $\glmat{n}{R} = O_n\stwotwo{Ch}{Sh}{Sh}{Ch}\stwotwo{\glmat{p}{R}}{}{}{\glmat{q}{R}}$ & $\glmat{n}{C} = U_n\stwotwo{Ch}{Sh}{Sh}{Ch}\stwotwo{\glmat{p}{C}}{}{}{\glmat{q}{C}}$ & $\glmat{n}{H} = \unmat{n}{H}\stwotwo{Ch}{Sh}{Sh}{Ch}\stwotwo{\glmat{p}{H}}{}{}{\glmat{q}{H}}$\\ \hhline{~|-|-|-|-|}

& \multirow{2}{*}{$\fact{9}$} & \cellcolor{gray!25}$\glmat{n}{R} = O_n\Sigma \unmat{p,q}{R}$ &
\cellcolor{gray!25}$\glmat{n}{C} = U_n\Sigma \unmat{p,q}{C}$ &
\multirow{2}{*}{$\glmat{n}{H} = \unmat{n}{H}\Sigma \unmat{p,q}{H}$} \\

& & \multicolumn{2}{c|}{\cellcolor{gray!25}{\color{red}Hyperbolic SVD (1989) Onn, Steinhardt, Bojanczyk} \cite{onn1989hyperbolic}} & \\
\hhline{~|-|-|-|-|}

& \multirow{2}{*}{$\fact{10}$} & 
\cellcolor{gray!25} $\glmat{2n}{R} = O_{2n}\stwotwo{\Sigma}{}{}{\Sigma}\sympmat{2n}{R}$ &
\multirow{2}{*}{$\glmat{2n}{C} = U_{2n}\stwotwo{\Sigma}{}{}{\Sigma}\sympmat{2n}{C}$} & \\

& & \cellcolor{gray!25}{\color{red}SVD-like decomp. (2003) Xu} \cite{xu2003svd} & & \\
\cline{2-5}

& $\fact{11}$ & $\glmat{2n}{R} = O_{2n}\stwotwo{\Sigma}{}{}{\Sigma^{-1}} \realify{\glmat{n}{C}}$ & $\glmat{2n}{C} = U_{2n}\stwotwo{\Sigma}{}{}{\Sigma^{-1}} \complexify{\glmat{n}{H}}$ & \\ \cline{2-5}

& $\fact{12}$ & & $\glmat{n}{C} = U_n B \glmat{n}{R}$ & $\glmat{n}{H} = \unmat{n}{H} B \glmat{n}{C}$ \\ \cline{2-5}

& $\fact{13}$ & & $\glmat{n}{C} = U_n\Sigma \orthomat{n}{C}$ & $\glmat{n}{H} = \unmat{n}{H}\Sigma \orthomat{n}{H}$ \\ 
\noalign{\global\arrayrulewidth=1pt}\hhline{-|-|-|-|-|}\noalign{\global\arrayrulewidth=0.4pt}

\multirow{5}{*}{\rotatebox[origin=c]{90}{$\Symp_\beta(2n)$\hspace{1cm}}} & \cellcolor{brown!20} & 
\cellcolor{gray!25} $\sympmat{2n}{R} = \realify{\unmat{n}{C}}\stwotwo{\Sigma}{}{}{\Sigma^{-1}}{\realify{\unmat{n}{C}}}'$ & 
\cellcolor{gray!25} $\sympmat{2n}{C} = \complexify{\unmat{n}{H}}\stwotwo{\Sigma}{}{}{\Sigma^{-1}}\complexify{\unmat{n}{H}}'$ & \\

& \multirow{-2}{*}{\cellcolor{brown!20}$\fact{14}$} & \cellcolor{gray!25}{\color{red}Balian (1965)} \cite{balian1965forme}, {\color{red}Xu (2003)} \cite{xu2003svd}
& \cellcolor{gray!25}{\color{red}Fa{\ss}bender, Ikramov (2005)} \cite{fassbender2005several} & \\ \cline{2-5}

& $\fact{15}$ & \makecell{$\sympmat{2n}{R} = $ \\ $\realify{\unmat{n}{C}}\stwotwo{Ch}{Sh}{Sh}{Ch}\stwotwo{\glmat{n}{R}}{}{}{(\glmat{n}{R})^{-T}}$} & \makecell{$\sympmat{2n}{C} = $ \\ $\complexify{\unmat{n}{H}}\stwotwo{Ch}{Sh}{Sh}{Ch}\stwotwo{\glmat{n}{C}}{}{}{(\glmat{n}{C})^{-T}}$} & \\ \cline{2-5}

& $\fact{16}$ & \makecell{$\sympmat{2n}{R} = $ \\ $\realify{\unmat{n}{C}}\stwotwo{H}{}{}{H^{-1}}\stwotwo{\sympmat{2p}{R}}{}{}{\sympmat{2q}{R}}$} & \makecell{$\sympmat{2n}{C} = $ \\ $\complexify{\unmat{n}{H}}\stwotwo{H}{}{}{H^{-1}}\stwotwo{\sympmat{2p}{C}}{}{}{\sympmat{2q}{C}}$} & \\ \cline{2-5}

& $\fact{17}$ & & $\sympmat{2n}{C} = \complexify{\unmat{n}{H}}\stwotwo{Ch}{iSh}{-iSh}{Ch}\sympmat{2n}{R}$ & \\
\noalign{\global\arrayrulewidth=1pt}\hhline{-|-|-|-|-|}\noalign{\global\arrayrulewidth=0.4pt}

\multirow{7}{*}{\rotatebox[origin=c]{90}{$\Un_\beta(p, q)$\hspace{2.5cm}}} & \cellcolor{brown!20} & \cellcolor{gray!25}$\unmat{p,q}{R} = $ & \cellcolor{gray!25}$\unmat{p,q}{C} = $ & $\unmat{p,q}{H}=$ \\

& \cellcolor{brown!20} & \cellcolor{gray!25}$\stwotwo{O_p}{}{}{O_q}\stwotwo{Ch}{Sh}{Sh}{Ch}\stwotwo{O_p'}{}{}{O_q'}$ & \cellcolor{gray!25}$\stwotwo{U_p}{}{}{U_q}\stwotwo{Ch}{Sh}{Sh}{Ch}\stwotwo{U_p'}{}{}{U_q'}$ & \multirow[t]{2}{*}{$\stwotwo{\unmat{p}{H}}{}{}{\unmat{q}{H}}\stwotwo{Ch}{Sh}{Sh}{Ch}\stwotwo{{\unmat{p}{H}}'}{}{}{{\unmat{q}{H}}'}$} \\

& \multirow{-3}{*}{\cellcolor{brown!20}$\fact{18}$\vspace{0.4cm}} & \multicolumn{2}{c|}{\cellcolor{gray!25}{\color{red}Hyperbolic CS decomposition }\cite{grimme1996model,vilenkin2013representation,wigner1968generalization}} & \\ 
\cline{2-5} 

& $\fact{19}$ & \makecell{$\unmat{p,q}{R} =$ \\ $\stwotwo{O_p}{}{}{O_q}\stwotwo{H}{}{}{H'}\stwotwo{\unmat{a, b}{R}}{}{}{\unmat{c,d}{R}}$} & \makecell{$\unmat{p,q}{C} =$ \\ $\stwotwo{U_p}{}{}{U_q}\stwotwo{H}{}{}{H'}\stwotwo{\unmat{a,b}{C}}{}{}{\unmat{c,d}{C}}$} & \makecell{$\unmat{p,q}{H} =$ \\ $\stwotwo{\unmat{p}{H}}{}{}{\unmat{q}{H}}\stwotwo{H}{}{}{H'}\stwotwo{\unmat{a,b}{H}}{}{}{\unmat{c, d}{H}}$} \\ \cline{2-5}

& $\fact{20}$ & \makecell{$\unmat{2p,2q}{R} = $ \\ $\stwotwo{O_{2p}}{}{}{O_{2q}}\stwotwo{I}{}{}{H}\realify{\unmat{p,q}{C}}$} & \makecell{$\unmat{2p,2q}{C} = $ \\ $\stwotwo{U_{2p}}{}{}{U_{2q}}\stwotwo{I}{}{}{H}\complexify{\unmat{p,q}{H}}$} & \\ \cline{2-5}

& $\fact{21}$& \makecell{$\unmat{n,n}{R} =$ \\ $\stwotwo{O_n}{}{}{O_n'}\stwotwo{Ch}{Sh}{Sh}{Ch}\realify{\orthomat{n}{C}}$} & \makecell{$\unmat{n,n}{C} =$ \\ $\stwotwo{U_n}{}{}{U_n'}\stwotwo{Ch}{Sh}{Sh}{Ch}\complexify{\orthomat{n}{H}}$}  & \\ \cline{2-5}

& $\fact{22}$ & & \makecell{$\unmat{p,q}{C} = $ \\ $\stwotwo{U_p}{}{}{U_q}\stwotwo{Ch}{iSh}{-iSh}{Ch}\unmat{p,q}{R}$} & \makecell{$\unmat{p,q}{H} = $ \\ $\stwotwo{\unmat{p}{H}}{}{}{\unmat{q}{H}}\stwotwo{Ch}{jSh}{-jSh}{Ch}\unmat{p,q}{C}$} \\ 
\noalign{\global\arrayrulewidth=1pt}\hhline{-|-|-|-|-|}\noalign{\global\arrayrulewidth=0.4pt}

\multirow{3}{*}{\rotatebox[origin=c]{90}{$\Ortho_\beta(n)$\hspace{0.3cm}}} & \cellcolor{brown!20}$\fact{23}$ & & $\orthomat{n}{C} = O_n B O_n'$ & $\orthomat{n}{H} = U_n B U_n'$ \\ \cline{2-5}

& $\fact{24}$ & & $\orthomat{n}{C} = O_n\stwotwo{Ch}{iSh}{-iSh}{Ch}\stwotwo{\orthomat{p}{C}}{}{}{\orthomat{q}{C}}$ & $\orthomat{n}{H} = U_n \stwotwo{Ch}{jSh}{-jSh}{Ch}\stwotwo{\orthomat{p}{H}}{}{}{\orthomat{q}{H}}$ \\ \cline{2-5}

& $\fact{25}$ & & $\orthomat{2n}{C} = O_{2n} \stwotwo{Ch}{iSh}{-iSh}{Ch}\complexify{\orthomat{n}{H}}$ & \\
\cline{1-5}
\end{tabu}}
\caption{\small The list of matrix factorizations obtained in our work. For matrix notations, see Table \ref{tab:translation} (last column) and Table \ref{tab:auxmatrix}. All 53 cases can be found as theorems in this work. \colorbox{brown!20}{Brown} denotes KAK decompositions,  \colorbox{gray!25}{gray} denotes factorizations where we have found references. Matrix block structures may be abbreviated to save space.}
\label{tab:completelist}
\end{table}

\par \textbf{Our main result}, Table \ref{tab:completelist} is the list of 53 systematically derived matrix factorizations arising from the \textit{generalized Cartan ($\kak$) decomposition} \cite{flensted1978spherical}. Table \ref{tab:completelist} may be thought of as a $25\times 3$ array with entries consisting of matrix factorizations. The rows of Table \ref{tab:completelist} are labeled $\fact{1}$ through $\fact{25}$. 

\par To assist the reader in getting a quick sense of how to read Table \ref{tab:completelist} using Tables \ref{tab:translation} and \ref{tab:auxmatrix} to look up notation, we will use the (real orthogonal) $\times$ (unitary diagonal) $\times$ (another real orthogonal) $\fact{1}$ factorization of any unitary matrix as an example:
\begin{figure}[h]
    \title{\vspace{-0.4cm}\textbf{\Large KEY TO FACTORIZATION TABLES} \\ e.g.  $\fact{1}$ : Unitary = (real orth.) $\times$ (unitary diagonal) $\times$ (another orth.) } 
    \centering
    \frame{\includegraphics[width=4.8in]{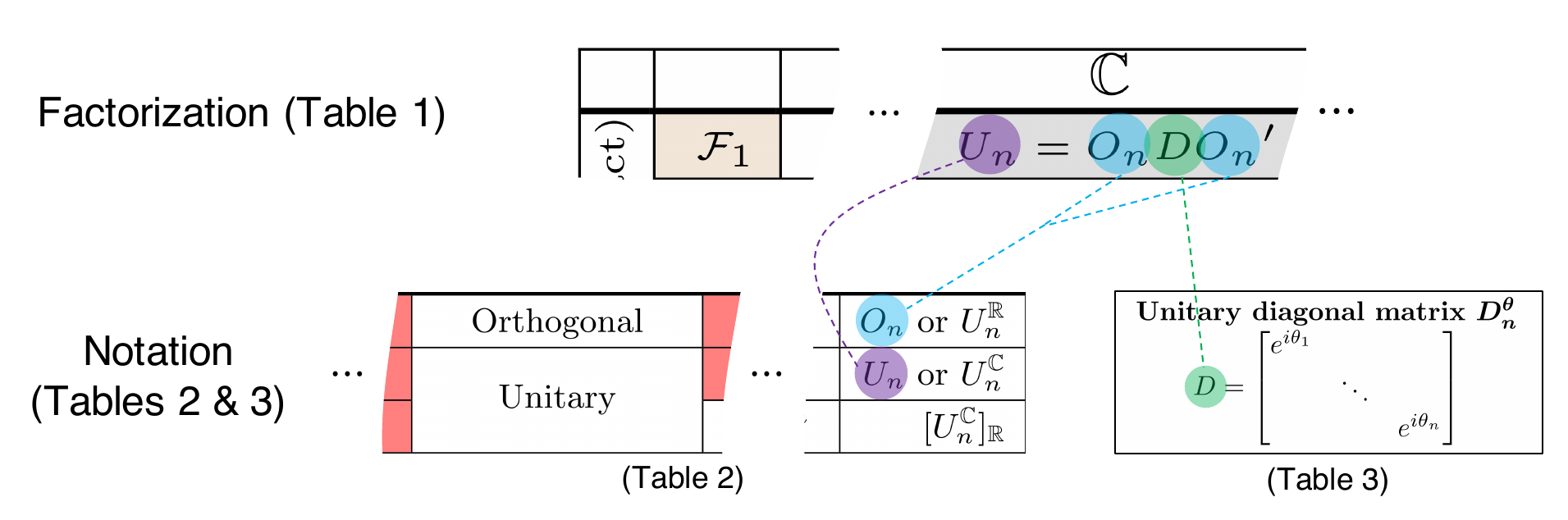}}
    \label{fig:keyfacts}
\end{figure}

\par As mathematical background to the factorizations in Table \ref{tab:completelist}, we mention the KAK and $\kak$ decompositions. The KAK decomposition marked with \colorbox{brown!25}{brown} includes the SVD (labeled $\fact{7}$) and the square partition CS decomposition (CSD). The $\kak$ decomposition includes the general CSD (labeled $\fact{4}$) and many others. Cells colored in \colorbox{gray!25}{gray} denote factorizations where we have found references in such communities as numerical linear algebra and physics.

\subsection{The KAK and $\kak$ decompositions}
Like a first edition Superman comic, gathering dust unrecognized in an attic, the KAK and $\kak$ decompositions are somewhat unknown. The reasons for this provide important lessons from mathematical history. (See Section \ref{sec:lesson} for details with a timeline.)

\par The KAK decomposition, may have first appeared in the work of Harish-Chandra \cite{harish1956representations}, described as a refinement or sharpening of the Cartan decomposition by Helgason \cite{Helgason1962}. However it remains unclear if Cartan himself knew about the KAK \cite{helgasonprivatecomm}. Nonetheless, perhaps out of respect, Helgason refers to the KAK as the Cartan decomposition in his 1978 classic \cite[p.402]{Helgason1978}.

\par We point out that neither the explicit name ``KAK decomposition" nor the importance of the decomposition is stated in the works above by Harish-Chandra or Helgason. Rather the decomposition is used somewhat in passing as a tool, and hence it was unlikely to be noticed by all but a few. 

\par We found a fascinating parallel in the discovery and the proclamation of the CSD amongst  20th century applied mathematicians, also at first used only as a tool, but later given a name and recognition as a matrix factorization by Stewart as described beautifully by Paige and Wei \cite[p.308]{paige1994history}:\footnote{Reference numbers are altered from the original quote to match the references in this paper for the convenience of the reader.}

\begin{quote}
\textit{Unaware of \cite{davis1969some}, Stewart in an appendix to \cite{stewart1977perturbation} gave a proof of the CSD ... [Stewart's] contribution was extremely important, not so much because it appears to be the first complete proof given for a CSD, but because it simply stated the CSD as a general decomposition of unitary matrices, rather than just a tool for analyzing angles between subspaces - this was something \cite{davis1969some} had not quite done. This elegant and unequivocal statement brought the CSD to the notice of a wider audience, as well as emphasizing its broader usefulness. Stewart widely advocated the use of the CSD, and came up with this appropriate name.}
\end{quote}

\par Just as the KAK remained obscure for many years, it is probably safe to say at this time that the generalization to $\kak$ remains very little known. The $\kak$ decomposition\footnote{We are grateful to Pavel Etingof, pointing us to \cite{Kobayashi} which lead us down the path of discovery.} appears first in Flensted-Jensen's work of 1978 \cite{flensted1978spherical}, where he names his result the ``generalized Cartan decomposition." As history repeats itself, it seems $\kak$ was created as a tool to study the structure of semisimple Lie algebras for specialists in the pure mathematical areas of harmonic analysis, Lie theory, and representation theory. For such studies, there seems to be little interest in individual factorizations so important in applied mathematics, engineering, the sciences, and probably, in the end, also useful to pure mathematics.

\subsection{Unified approaches to matrix factorizations} 
Given the importance of  matrix factorizations for applications, it would be of no surprise that various attempts towards  unification have been considered. We begin with a brief survey on the applied side. Mackey, Mackey and Tisseur \cite{mackey2003structured,mackey2005structured} provide an overview of structural tools for factorizations in the context of automorphism groups of scalar products. We were inspired by their work, especially their open questions suggested in Section 5.2 and 8.3 of \cite{mackey2005structured}, whose answer mostly lies in the global Cartan decompositions and KAK decompositions discussed in this paper and their isomorphic forms as described in Remark \ref{rem:isomorphismpreserveF}, e.g., the perplectic SVD in Remark \ref{rem:realperplecticsvd}.

\par Notably, Kleinsteuber provides a unified algorithmic approach on the decomposition that Cartan for sure has in his work: $\mathfrak{p} = \bigcup_{k\in K}\text{Ad}(k)\mathfrak{a}=\bigcup_{k\in K}k\mathfrak{a}k^{-1}$ (see Figure \ref{fig:cartankak} in Section \ref{sec:lesson}), in a framework of a generalization of the Givens algorithm. The structure preserving Jacobi algorithms proposed in  Kleinsteuber's thesis (2005) \cite{kleinsteuber2005jacobi} treats the SVD, Hermitian eigendecomposition and more in a unified scheme. A broad chart of structured eigenproblems with references is also nicely listed in \cite{bunse1992chart}. Bhatia (1994) \cite{bhatia1994matrix} considers a few cases where the tangent spaces are decompositions of matrix spaces.

\par The abstraction approach is great for the soul as it lets us understand so much, and can lead to new discoveries, but often the discoverers are unaware or uninterested in the details of what might seem to them as tedious examples. The applied approach is critical for applications and impact. We applaud both the pure and applied styles as valuable intellectual achievements. We credit Cartan for discovering an abstraction for the aforementioned class of matrix factorizations: $\mathfrak{p} = \bigcup_{k\in K}k\mathfrak{a}k^{-1}$  \cite[p.359]{cartan1927certaines}, and credit many 20th century mathematicians for forthcoming abstractions.
Cartan, Iwasawa\footnote{Who incidentally taught Gil Strang linear algebra at MIT.}, Kostant and Bruhat represent a line of mathematicians that skip past any one matrix factorization and do what is admirable in pure mathematics, create an abstraction, one might even say a blueprint for many factorizations. We will argue, in contrast, that specialists in matrix factorizations, some examples in this context include Golub, Kahan, Paige, Stewart, Van Loan and many others did what is admirable in applied mathematics. They discovered, developed, named, found applications, properties, and/or algorithms\footnote{Especially in the context of floating point arithmetic.} for specific matrix factorizations.


\subsection{Organization of the paper} 
Section \ref{sec:dictionary} provides an overview of the matrices in classical Lie groups (summarized in Table \ref{tab:translation}) from a linear algebra point of view intended for a modern audience. Section \ref{sec:background} discusses the basics of the KAK and $\kak$ decomposition, with easy-to-follow examples in Sections \ref{sec:gencartanex1} and \ref{sec:gencartanex2}. Sections \ref{sec:compact} to \ref{sec:orthofact} are the main results of this paper (summarized in Table \ref{tab:completelist}), the matrix factorizations of the following matrices:
\begin{itemize}
    \item Section \ref{sec:compact} : Orthogonal and unitary matrices
    \item Section \ref{sec:glfact} : Invertible matrices
    \item Section \ref{sec:sympfact} : Symplectic matrices
    \item Section \ref{sec:jorthogonalfact} : Indefinite unitary matrices
    \item Section \ref{sec:orthofact} : Complex and quaternionic orthogonal matrices
\end{itemize}


Section \ref{sec:lesson} points out some important lessons of mathematical history, together with a discussion from {\'E}lie Cartan's dicoveries to the developments of relatively newer matrix factorizations, in a chronological order. 

\addtocontents{toc}{\protect\setcounter{tocdepth}{2}}
\section{Matrices from the classical Lie groups}\label{sec:dictionary}
It has long been established that orthogonal matrices enjoy very special numerical and analytical properties. The \textit{classical Lie groups} named by Weyl \cite{weyl1946classical} consist of invertible matrices and generalizations of orthogonal matrices (Table \ref{tab:translation}). These generalizations are very simple and are the building blocks of our 53 factorizations.

\subsection{Preliminaries}\label{sec:RCHbasic}
\hfill \\
\noindent \textbf{Real, complex and quaternionic matrices:} Let $\mathbb{R}, \mathbb{C}, \mathbb{H}$ denote the reals, complexes and (real) quaternions. The symbol $\mathbb{F}$ will be used in all three cases, though quaternions are not a field. The $\beta$-symbol with $\beta=1, 2, 4$ is often used, representing each of the three cases respectively. $m\times n$ matrices are denoted by $\mathbb{F}^{m\times n}$.
\begin{definition}
    The matrices $J_n$ and $I_{p, q}$ are defined as follows. ($p, q, n \in\mathbb{N}$)
    \begin{equation}\label{eq:JnIpqdefinition}
        J_n := \twotwo{\,\,0}{I_n}{-I_n}{0}, \hspace{0.5cm} I_{p, q} := \twotwo{I_p}{0}{\,\,0}{-I_q}.
    \end{equation}
\end{definition}

\noindent \textbf{Realify and complexify:} Sometimes, we may represent a complex matrix as a real matrix\footnote{When programming languages did not support imaginary numbers, the same representation had been used by programmers, using a $2m\times 2n$ real matrix to encode an $m\times n$ complex matrix.} (which we call ``realify"), and a quaternion matrix as a complex matrix (``complexify"). The \textit{realify map} $\realify{\,\,\cdot\,\,}:\mathbb{C}^{m\times n}\to\mathbb{R}^{2m\times 2n}$ is defined as
\begin{equation}\label{eq:realify}
    \text{Realify : } X \mapsto \realify{X} \equiv \twotworr{\Re(X)}{\Im(X)}{-\Im(X)}{\Re(X)} .
\end{equation}
The \textit{complexify map} $\complexify{\,\,\cdot\,\,}:\mathbb{H}^{m\times n}\to\mathbb{C}^{2m\times 2n}$ is defined as 
\begin{equation}\label{eq:complexify}
    \text{Complexify : } Y \mapsto \complexify{Y} \equiv \twotworr{\Re(Y)+i\Im_i(Y)}{\Im_j(Y)+i\Im_k(Y)}{-\Im_j(Y)+i\Im_k(Y)}{\Re(Y)-i\Im_i(Y)},
\end{equation}
where $Y = \Re(Y) + i\Im_i(Y) + j\Im_j(Y) + k\Im_k(Y)$. 

\vspace{.2cm}
\noindent \textbf{The Transposes ($M^T,M^H,M^D,M^{D_i},M^{D_j},M^{D_k}$):}
The following defines variations on transpose for real, complex, and quaternion matrices:
\begin{alignat}{4}
(M^T)_{ab} &= M_{ba}, \,\,(\text{for } \mathbb{R}, \mathbb{C}, \mathbb{H}), \hspace{0.5cm} &&(M^H)_{ab} &&= \overline{M_{ba}}, &&\,\,(\text{for }\mathbb{C}) \\
(M^D)_{ab} &= \overline{M_{ba}},\,\,(\text{for }\mathbb{H}) && (M^{D_\eta})_{ab} &&= -\eta\overline{M_{ba}}\eta, &&\,\,(\text{for }\mathbb{H}, \eta\in\{i,j,k\})
\end{alignat}
We define $M^{D_i}$,$M^{D_j}$,$M^{D_k}$ for quaternion matrices as the transpose that only ``conjugates" the $i$, $j$, or $k$ term respectively. (See Figure \ref{fig:transposes} for an example of $M^T,M^D$ and $M^{D_j}$ for a quaternionic matrix $M$.) Formally we state the following \cite{rodman2014topics}.  
\begin{definition}
For a quaternionic matrix $M\in\mathbb{H}^{n\times n}$, the \textit{$\eta$-conjugate\footnote{We follow the $\eta$ notation in Horn and Zhang \cite{horn2012generalization}, for quaternionic imaginary units $i,j,k$.} transpose} $M^{D_\eta}$ is defined as ($\eta\in\{i, j, k\}$),
\begin{equation}\label{eq:etaDdefinition}
    M^{D_\eta}:=\eta^{-1}M^D\eta = -\eta M^D\eta.
\end{equation}
Moreover, if $A = A^{D_\eta}$ we say $A$ is \textit{$\eta$-Hermitian}.
\end{definition}

\begin{figure}[h]
    \centering
    {\includegraphics[width=4in]{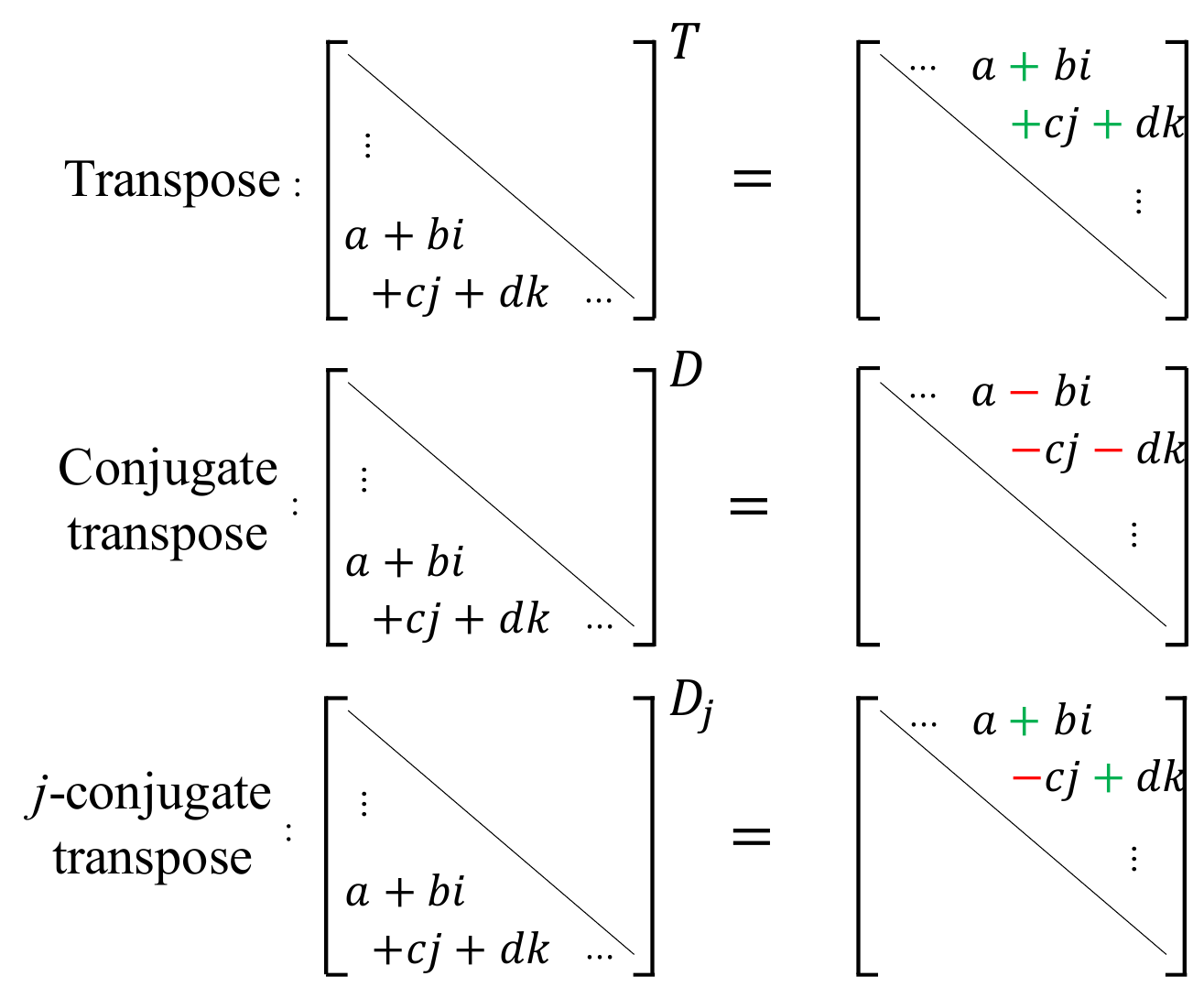}}
    \caption{Transpose $T$, conjugate transpose $D$, and $j$-conjugate transpose $D_j$ of a quaternionic matrix. $i$-conjugate transpose and $k$-conjugate transpose are defined similarly.}
    \label{fig:transposes}
\end{figure}

\par The symbols $T$, $H$, and $D$ refer to the \textit{transpose}, \textit{Hermitian transpose}, and \textit{dual} respectively. $H$ and $D$ both serve as \textit{conjugate transposes}. When it is unambiguous, we use the common notation $M^\dagger$ to denote the complex/quaternionic conjugate transpose ($M^H$ or $M^D$).

\par Some identities are useful. Suppose $C\in\mathbb{C}^{n\times n}$ and $Q\in\mathbb{H}^{n\times n}$. Then we have
\begin{align}
    \realify{C^H} &= \realify{C}^T, \\
    \realify{C^T} &= I_{n, n}\realify{C}^T I_{n, n}, \\
    \complexify{Q^D} &= \complexify{Q}^H, \\
    \complexify{Q^{D_j}} &= [Q]_\mathbb{C}^T = -J_n\complexify{Q}^HJ_n. \label{eq:DjcomplexifyT}
\end{align}
Additionally, $-J_n\realify{C}J_n = \realify{C}$ and $-J_n\overline{\complexify{Q}}J_n = \complexify{Q}$ holds. 

\subsection{Solutions $G$ to $G^*JG = J$ create Lie groups}\label{sec:autogroup}
The classical Lie groups can be derived by solving the quadratic matrix equation $G^*JG = J$ for $G$ (invertible solutions), given a fixed $J$ (top of Table \ref{tab:translation}). One can select $*$ to be $\{T, H, D\}$ or $\{T, T, T\}$ for $\mathbb{R}, \mathbb{C}, \mathbb{H}$ respectively. (For $\mathbb{H}$, $T$ is applied through the complexified matrix as $T$ is not an involution on $\mathbb{H}^{n\times n}$. See Remark \ref{rem:quaternioninvolution}.) For a choice of $J$, one can select among $0_n$, $I_n$, $I_{p, q}$ and $J_n$.\footnote{While these matrices at first may seem very arbitrary, they represent large equivalence classes through the change of basis. For example $I_{p, q}$ represents any real symmetric or complex Hermitian matrix as the ``signature" of that matrix \cite[p.240]{artin2011algebra} and $J_n$ represents any skew-symmetric or skew-Hermitian matrix.} For example, choosing $J = I_n$ and $* = \{T, H, D\}$ gives three groups, $\{G:G^TG = I_n\}$, $\{G:G^HG = I_n\}$, $\{G:G^DG = I_n\}$ corresponding to $\beta = 1, 2, 4$. 

\par Out of 24 combinations\footnote{2 for the choice of $*$, 3 for the choice of $\mathbb{F}$, 4 for the choice of $J$.}, we omit the overlapping (i.e., isomorphic, so double counting) 11. (Remark \ref{rem:conjsymp} talks about one such isomorphism.) The remaining 13 Lie groups are presented in the \textit{upper part} of Table \ref{tab:translation} denoted by 5 simplified $\beta$-notated symbol: $\color{gray}\Gl_\beta(n)$, $\color{red}\Un_\beta(n)$, $\color{blue}\Un_\beta(p, q)$, $\color{brown}\Symp_\beta(2n)$ and $\color{teal}\Ortho_\beta(n)$. The \textit{lower table} drills down into the complete list of 13 classical Lie groups including  the commonly used  Lie group symbols (column 3), the names of matrices (column 4) and  definitions (column 5). The last column is our symbol for individual matrices which will be used throughout this paper. 

\begin{table}
\begin{center}
{\tabulinesep=0.0mm
\small
\setlength{\tabcolsep}{0pt}
\hspace*{-.3in}
\begin{tabu}{|c|c|ccc|}
\multicolumn{5}{c}{\large Classical Lie groups, $\{\text{Invertible } G:G^*JG=J\}$} \\
\multicolumn{5}{c}{\footnotesize A transpose $*$ and a matrix $J$ can produce a classical Lie group} \\
\multicolumn{5}{c}{($\beta=1,2,4$  represents $\mathbb{R},\mathbb{C},\mathbb{H}$)} \\
\hline

\diagbox[width=2.5cm, height=1cm]{\footnotesize Transpose $*$ }{\makecell{\,\\$J$\,\,}} & \makecell{$0_n$} & $I_n$ & $I_{p, q}=\stwotwo{I_p}{0}{0}{-I_q}$ & $J_n=\stwotwo{0}{I_n}{-I_n}{0}$ \\
\hline

& &\multicolumn{3}{c|}{\cellcolor{white}\color{black}{{The below are all generalizations of orthogonal matrices}}} \\

\makebox(75,30){\makecell{Conjugate \\ transposes \\ $M^T,M^H,M^D$}} &
\colorbox{gray!25}{\makebox(60,30){\makecell{\large$\Gl_\beta(n)$\\``Invertibles"\\$\beta=1,2,4$}}} &
\colorbox{red!50}{\makebox(80,30){\makecell{\large$\Un_\beta(n)$\\``Unitaries"\\$\beta=1,2,4$}}} &
\colorbox{blue!30}{\makebox(80,30){\makecell{\large$\Un_\beta(p,q)$\\``Indefinite Unitaries"\\$\beta=1,2,4$}}} & 
\makebox(80,30){\makecell{\,\\-\\\,}} \\

\makebox(75,30){\makecell{Transpose\\ (only) $M^T$}} & 
\makebox(60,30){\makecell{\,\\-\\\,}} &  
\colorbox{teal!20}{\makebox(80,30){\makecell{\large$\Ortho_\beta(n)$\\``$\mathbb{C}, \mathbb{H}$ Orthogonals"\\$\beta=2,4$}}} & 
\makebox(80,30){\makecell{\,\\-\\\,}} & 
\colorbox{brown!50}{\makebox(80,30){\makecell{\large$\Symp_\beta(2n)$\\``Symplectics"\\$\beta=1,2$}}} \\
\hline
\end{tabu}}


{\tabulinesep=0.6mm
\small
\centering
\begin{tabu}{|c|c|c|c|c|c|r|c|}
\hline
{\rotatebox[origin=c]{90}{Our Symbol}\hspace*{-.03in}} & $\beta$ & \makecell{Lie group\\ Symbol} & \makecell{Matrix \\ Name}& Definition & \makecell{Ambient\\ Space}& \makecell{Our \\ Matrix\\Symbol}
\\
\Xhline{3\arrayrulewidth}

\multirow{4}{*}{\hspace*{-.1in}\rotatebox[origin=c]{90}{\hspace{1cm}{ \colorbox{gray!25}{\makebox[2.6cm]{ $\Gl_\beta(n)$} }}} \hspace*{-.2in}} & 1 & \cellcolor{gray!25}{} $\gl{n, \mathbb{R}}$ & \makecell{Real\\Invertible} &
\cellcolor{gray!25} $\{G:\det G\ne 0\}$ & $ \mathbb{R}^{n\times n}$ & $\glmat{n}{R}$\\
\hhline{~|-|-|-|-|-|-|}
& 2 & \cellcolor{gray!25}{}$\gl{n, \mathbb{C}}$ & \makecell{Complex\\Invertible} & \cellcolor{gray!25}$\{G:\det G\ne 0\}$   & $\mathbb{C}^{n\times n}$ &  $\glmat{n}{C}$ \\
\hhline{~|-|-|-|-|-|-|}
& \multirow{2}{*}{4} &\cellcolor{gray!25}{} $\gl{n, \mathbb{H}}$ & \multirow{2}{*}{\makecell{Quaternion\\Invertible}} & \cellcolor{gray!25}$\{G:\det G\ne 0\}$ & $ \mathbb{H}^{n\times n}$ & $\glmat{n}{H}$ \\
\hhline{~|~|-|~|-|-|-|}
& &\cellcolor{gray!25}{} $\Un^*(2n)$ &  & Complexify($\uparrow$) & $\mathbb{C}^{2n\times 2n}$ & $\complexify{\glmat{n}{H}}$ \\
\Xhline{2.5\arrayrulewidth}

\multirow{5}{*}{\hspace*{-.11in}\rotatebox[origin=c]{90}{\hspace{1cm}{\colorbox{red!50}{\makebox[2.3cm]{ $\Un_\beta(n)$} }}} \hspace*{-.2in}} & 1 & \cellcolor{red!50}{}$\ortho{n}$ &  Orthogonal & \cellcolor{red!50}$\{O:O^TO=I_n\} $ & $\mathbb{R}^{n\times n}$ & $ O_n \mbox{ or } \unmat{n}{R}$\\
\hhline{~|-|-|-|-|-|-|}
& \multirow{2}{*}{2} & \cellcolor{red!50}{} $\un{n}$ &\multirow{2}{*}{Unitary} & \cellcolor{red!50}$\{U:U^HU=I_n\}$ & $\mathbb{C}^{n\times n}$ & $U_n \mbox{ or } \unmat{n}{C}$\\
\hhline{~|~|-|~|-|-|-|}
& & \cellcolor{red!50}{}$\osp{2n}$ &  & Realify($\uparrow$) & $\mathbb{R}^{2n\times 2n}$ & $\realify{\unmat{n}{C}}$\\
\hhline{~|-|-|-|-|-|-|}
& \multirow{2}{*}{4} & \cellcolor{red!50}{}$\un{n, \mathbb{H}}$ & \multirow{2}{*}{\makecell{Quaternionic\\Unitary}} & 
\cellcolor{red!50}$\{U:U^DU=I_n\}$ & $\mathbb{H}^{n\times n}$ & $\unmat{n}{H}$\\
\hhline{~|~|-|~|-|-|-|}
& & \cellcolor{red!50}{}$\usp{2n}$ &  & Complexify($\uparrow$) & $\mathbb{C}^{2n\times 2n}$ & $\complexify{\unmat{n}{H}}$\\
\Xhline{2.5\arrayrulewidth}

\multirow{4}{*}{\hspace*{-.11in}\rotatebox[origin=c]{90}{ \hspace{1cm}{\colorbox{blue!30}{\makebox[1.8cm]{ $\Un_\beta(p,q)$} }}} \hspace*{-.2in}} & 1 &  \cellcolor{blue!30}{}$\ortho{p, q}$ & Indef orthogonal &\cellcolor{blue!30}{} $\{O:O^TI_{p, q}O = I_{p, q}\}$ & $\mathbb{R}^{n\times n}$ & $\unmat{p, q}{R}$\\
\hhline{~|-|-|-|-|-|-|}
& 2 & \cellcolor{blue!30}{}$\un{p, q}$ & Indef unitary &\cellcolor{blue!30}{} $\{U:U^HI_{p, q}U = I_{p, q}\}$ & $\mathbb{C}^{n\times n}$ & $\unmat{p,q}{C}$\\
\hhline{~|-|-|-|-|-|-|}
& \multirow{2}{*}{4} &\cellcolor{blue!30}{}  $\un{p, q, \mathbb{H}}$ & \multirow{2}{*}{\makecell{Quaternionic\\Indef unitary}} &\cellcolor{blue!30}{}  $\{U:U^DI_{p, q}U = I_{p, q}\}$ & $\mathbb{H}^{n\times n}$ & $\unmat{p,q}{H}$\\
\hhline{~|~|-|~|-|-|-|}
& & \cellcolor{blue!30}{} $\symp{p, q}$ &  &
Complexify($\uparrow$) & $\mathbb{C}^{2n\times 2n}$ & $\complexify{\unmat{p,q}{H}}$ \\
\Xhline{2.5\arrayrulewidth}

\multirow{3}{*}{\hspace*{-.11in}\rotatebox[origin=c]{90}{ \hspace{1cm}{\colorbox{teal!20}{\makebox[2.2cm]{$\Ortho_\beta(n)$} }}} \hspace*{-.2in}} & 2 & \cellcolor{teal!20}{} $\ortho{n, \mathbb{C}}$ & \makecell{Complex\\Orthogonal} & \cellcolor{teal!20}{} $\{O:O^TO=I_n\}$ & $\mathbb{C}^{n\times n}$ & $\orthomat{n}{C}$\\
\hhline{~|-|-|-|-|-|-|}
& \multirow{2}{*}{4} &\cellcolor{teal!20}{}  $\ortho{n, \mathbb{H}}$ & \multirow{2}{*}{\makecell{\,\\Quaternionic\\Orthogonal}} & \cellcolor{teal!20}$\{O:O^{D_j}O=I_n\}$ & $\mathbb{H}^{n\times n}$ & $\orthomat{n}{H}$\\
\hhline{~|~|-|~|-|-|-|}
& & \cellcolor{teal!20}{} $\Ortho^*(2n)$ &  & \makecell{Complexify($\uparrow$) \\ \tiny (If complexified, $O^TO = I_{2n}$)} & $\mathbb{C}^{2n\times 2n}$ & $\complexify{\orthomat{n}{H}}$ \\
\Xhline{2.5\arrayrulewidth}

\multirow{2}{*}{\hspace*{-.11in}\rotatebox[origin=c]{90}{{\colorbox{brown!50}{\makebox[1.6cm]{$\Symp_\beta(2n)$}}}} \hspace*{-.2in}}
 & 1 & \cellcolor{brown!50}{} $\symp{2n, \mathbb{R}}$ & \makecell{Real\\Symplectic} &
\cellcolor{brown!50}
$\{S:S^TJ_nS=J_n\}$ & $\mathbb{R}^{2n\times 2n}$ & $\sympmat{2n}{R}$ \\
\hhline{~|-|-|-|-|-|-|}
& 2 &\cellcolor{brown!50}{}  $\symp{2n, \mathbb{C}}$ & \makecell{Complex\\Symplectic} & \cellcolor{brown!50}
$\{S:S^TJ_nS=J_n\}$ & $\mathbb{C}^{2n\times 2n}$ & $\sympmat{2n}{C}$ \\
\Xhline{2.5\arrayrulewidth}
\end{tabu}}
\end{center}
\caption{Not satisfied with the organization of tables of classical Lie groups in the literature, we propose the following simplified organization of the 13 classical Lie groups.}
\label{tab:translation}
\end{table}

\begin{remark}[isomorphic Lie groups that change $\mathbb{F}$]\label{rem:isomorphisms}
The realify and complexify maps create isomorphic Lie groups. In the Lie theory literature, we rarely find a distinction between isomorphic representations of the same Lie group as the algebraic structure is considered more important than the specific representation. On the other hand, we have chosen to be more explicit: \textbf{we denote the isomorphic representation with a distinct matrix/Lie group notation.\footnote{There are a number of reasons that it is a good idea to use explicit representations. For example when factoring $2n\times 2n$ matrices in Section \ref{sec:F11}, it would appear to be nonsensical to multiply through an $n\times n$ complex matrix. (In other literature, it would be often left to the reader to realize that the isomorphism $\realify{\,\,\cdot\,\,}$ is at play.)}} 

\par Consider an example where the isomorphism $\complexify{\,\,\cdot\,\,}$ (the complexify map) is at play: the $n\times n$ quaternionic general linear group $\gl{n, \mathbb{H}}$ has an isomorphic copy $\Un^{*}(2n)$, which can be obtained by the complexify map. Matrices from each group will be distinctively denoted by $\glmat{n}{H}$ and $\complexify{\glmat{n}{H}}$, respectively. (See the last column of Table \ref{tab:translation} for the complete list of such distinct matrix symbols.) 

\end{remark}

\begin{remark}[isomorphic Lie groups that change $J$ but preserve $\mathbb{F}$]\label{rem:isomorphismpreserveF}
On first glance the explicit $J$'s that are commonly used such as $I_n,I_{p,q},J_n$ (and the less common $J=0_n$) seem arbitrary but they are not. If we may make an analogy with the Jordan form of a matrix, these $J$'s represent equivalence classes which may be built up in to a theory that can cover any $J$ \cite[p.92]{rossmann2006lie}. 

\par The perplectic group and perplectic matrices \cite{mackey2003structured} defined in \eqref{def:realperplectic}, \eqref{def:complexperplectic} provide interesting examples (see Remarks \ref{rem:realperplectic} and \ref{rem:complexperplectic} for real and complex cases, respectively). Another example is conjugate symplectic matrices in Remark \ref{rem:conjsymp}.
\end{remark}

\begin{remark}[Classical Lie groups as automorphism groups]\label{rem:automorphismgroups}
For a given $J$-scalar product $\langle x,y\rangle_J:=x^*Jy$, it is easy to see that $G^*JG=J$ if and only if  $\langle x,y\rangle_J=\langle Gx,Gy\rangle_J$ for all $x,y$.\footnote{Bilinearity can break down for $\mathbb{H}$ since $(x \alpha)^Ty= (x^T \alpha y) \ne \alpha x^T y$ and $(Mx)^T \ne x^T M^T$.} The set of such $G$ form a group, this group is known as the automorphism group of the scalar product, studied nicely in \cite{mackey2003structured}. Effectively, the abstract group is independent of basis, which is another way of seeing the equivalence class of possible $J$'s as mentioned in Remark \ref{rem:isomorphisms}.
\end{remark}

\subsection{$\Gl_\beta(n)$, general linear groups}
The groups $\gl{n, \mathbb{R}}$, $\gl{n, \mathbb{C}}$, $\gl{n, \mathbb{H}}$ are the groups of $n\times n$ invertible matrices of the corresponding $\mathbb{F}$. We adopt the simplified symbol $\Gl_\beta(n)$ to express the three groups at once, for $\beta = 1, 2, 4$. Note that the determinant zero condition for $\gl{n, \mathbb{H}}$ is defined through its complexified counterpart $\Un^*(2n)$, since there is no natural determinant for quaternionic matrices.

\subsection{$\Un_\beta(n)$, unitary groups}
Defined with $U^\dagger U = I_n$ for conjugate transpose $\dagger$, the name ``unitary group" (usually complex) can be extended to real and quaternionic cases. The real unitary group is indeed the orthogonal group $\ortho{n}$ of $n\times n$ orthogonal matrices, and the quaternionic unitary group $\un{n, \mathbb{H}}$ is the set of all quaternionic matrices satisfying $U^DU = I_n$. We denote these three groups at once by $\Un_\beta(n)$, $\beta = 1, 2, 4$. 

\subsection{$\Un_\beta(p, q)$, indefinite unitary groups}\label{sec:jorthogonal}
Let $n = p+q$. A real matrix $G\in\mathbb{R}^{n\times n}$ is called \textit{indefinite orthogonal} if it satisfies
\begin{equation}\label{def:jorthogonalmatrix}
    G^TI_{p,q}G = I_{p,q}.
\end{equation}
Similarly we define complex/quaternionic \textit{Indefinite unitary matrices} with 
\begin{equation}\label{def:junitarymatrix}
    G^\dagger I_{p,q}G = I_{p,q}.
\end{equation}

\par Indefinite orthogonal/unitary matrices have many other names. A matrix $G$ satisfying \eqref{def:junitarymatrix} is called a ``hypernormal matrix" in \cite{onn1989hyperbolic}, a``pseudo-unitary matrix" in \cite{mackey2003structured, mackey2005structured}, and the most commonly used name is ``$J$-unitary matrix" in \cite{higham2003j} and by many other authors. (Which is confusing since the term ``$J$-orthogonal" is sometimes used to describe a real symplectic matrix \cite{benner1998symplectic, xu2003svd} and as we do in this paper, $J$ is used more generally.) To avoid ambiguities, we use the term ``indefinite unitary matrix" (and ``indefinite orthogonal matrix") for the matrices satisfying \eqref{def:junitarymatrix} (and \eqref{def:jorthogonalmatrix}).

\par In the literature, the groups themselves go by various names with more traditional symbols. Our $\Un_1(p,q)$ goes by the symbol $\ortho{p, q}$ and is called the indefinite orthogonal group, pseudo-orthogonal group, hyperbolic orthogonal group, the generalized orthogonal group, etc. Our $\Un_2(p,q)$ is denoted simply $\un{p,q}$, the indefinite unitary group, etc. The quaternionic case is often denoted by the symbol $\symp{p, q}$ which we will restrict our use to the complex representation of $\Un_4(p,q)$. The quaternionic representation will be denoted by $\un{p, q, \mathbb{H}}$. 

\par In Remark \ref{rem:csdhcsdcompare}, we will point out the role played by these indefinite unitary groups in explaining a connection between the well-known CS decomposition and the perhaps not as well-known hyperbolic CS decomposition. A key mathematical point is the hyperbolic versions are not only analogous, but are indeed related through a Lie theoretic concept of the compact/noncompact dual KAK decomposition. 

\begin{remark}[the real perplectic group]\label{rem:realperplectic}
We remark that the interesting \textit{real perplectic group}, studied in \cite{mackey2003structured}, is isomorphic to the indefinite orthogonal group with $p=\lceil\frac{n}{2}\rceil$ and $q=\lfloor\frac{n}{2}\rfloor$ as has already been pointed out in \cite{mackey2005perplectic}. 

\par Take as our $J$ the (backward identity) matrix $E_n$ with 1's on its antidiagonal.
\begin{equation}\label{eq:Edefinition}
    E_n:=\threediag{0}{1}{\reflectbox{$\ddots$}}{1}{0}.
\end{equation}
The eigenvalues of $E_n$ are clearly $p=\lceil\frac{n}{2}\rceil$ positive ones and $q=\lfloor\frac{n}{2}\rfloor$ negative ones hence the signature is $(\lceil \frac{n}{2}\rceil,\lfloor\frac{n}{2}\rfloor).$ It follows that the real perplectic group 
\begin{equation}\label{def:realperplectic}
    \{G\in\mathbb{R}^{n\times n}:G^TE_nG = E_n\},
\end{equation}
is isomorphic to corresponding indefinite orthogonals. See Remark \ref{rem:realperplecticsvd} for isomorphism and the real perplectic SVD. Interestingly, the complex perplectic group is isomorphic to the complex orthogonal group. (See Remark \ref{rem:complexperplectic}.) 
\end{remark}

\subsection{$\Symp_\beta(2n)$, symplectic groups}\label{sec:symplectic}
The symplectic groups correspond to $J$'s that are real or complex skew-symmetric (as opposed to skew-Hermitian), represented by $J_n$ (see Remark \ref{rem:isomorphismpreserveF}). For $\mathbb{F} = \mathbb{R}, \mathbb{C}$ (but not $\mathbb{H}$)\footnote{There is no quaternion symplectic group due to the breakdown of linearity of the quaternion form $\langle x,y\rangle_{J_n} = x^TJ_ny$ as mentioned in the footnote in Remark \ref{rem:automorphismgroups}.}, a \textit{symplectic matrix} $G\in\mathbb{F}^{2n\times 2n}$ is defined with the property (the \textit{symplectic group} is denoted\footnote{In many literature $2n\times 2n$ symplectic groups are denoted by $\symp{n, \mathbb{F}}$.} by $\symp{2n, \mathbb{F}}$)
\begin{equation}\label{eq:sympdefinition}
G^T J_n G = J_n,
\end{equation}
where the transpose is the regular transpose without the complex conjugation. $\symp{2n, \mathbb{R}}$ and $\symp{2n, \mathbb{C}}$ are the real and the complex symplectic group, sometimes denoted by $\Symp_\beta(2n)$ with $\beta = 1, 2$.

\begin{remark}[using $H$ instead of $T$ is not complex symplectic]\label{rem:conjsymp}
A natural but ultimately inconsistent and problematic definition for complex symplectic matrices would take as a definition $G^HJ_nG=J_n$. This ``$H$" transposing definition may originate in 1979 \cite[p.4]{paige1979hamiltonian} and 1981 \cite[p.14]{paige1981schur} (which references and proposes a generalization of \cite[Paragraph 4.19.9]{marcus1992survey}), and sometimes called by the name \textit{conjugate symplectic matrices}\footnote{Tracking back, some older references call conjugate symplectic group the ``Hermitian symplectic group," and even older references use the name ``Hermitian modular group."} \cite{mackey2003structured}. (This confusion has been pointed out several times e.g., Remark 1.2 of \cite{mehrmann1988symplectic}.) The $2n\times 2n$ \textit{conjugate symplectic group} is denoted by $\Symp^*(2n, \mathbb{C})$ and it is isomorphic to the indefinite unitary group $\un{n, n}$. See Remark \ref{rem:conjsympandunn} for the isomorphism and factorizations of conjugate symplectic matrices. 
\end{remark}

\begin{remark}[$\beta$ doubling for $\beta$-unitaries]\label{rem:betadoubling}
We describe a charming alchemy that turns (at the group level) orthogonals into unitaries and similarly complex unitaries into quaternionic unitaries by the presence of the symplectic structure. Specifically our notation allows us to state both $\beta=1,2$ cases at once (respectively, $\mathbb{F} = \mathbb{R}, \mathbb{C}$):
\begin{equation}\label{eq:betadoubling}
    [\Un_{2\beta}(n)]_\mathbb{F} = \Un_\beta(2n)\cap\Symp_\beta(2n).
\end{equation}
Nonetheless it is worth spelling out what this says. The real case says that if you have a $2n\times 2n$ real orthogonal matrix that is also real symplectic then it is a realified unitary matrix. In addition, if you have a $2n\times 2n$ complex unitary matrix that is also complex symplectic then it is a complexified quaternionic unitary matrix. 

\par The group of $2n\times 2n$ unitary symplectic matrices is often denoted by $\symp{n}$ or $\text{USp}(2n)$, called the \textit{unitary (or compact) symplectic group}. Similarly the group of $2n\times 2n$ orthogonal symplectics is called the \textit{orthogonal symplectic group} and denoted by $\osp{2n}$. The following is an explicit version of \eqref{eq:betadoubling}. 
\begin{alignat}{2}
    \realify{\un{n}} &= \osp{2n} &&:= \ortho{2n}\cap\symp{2n, \mathbb{R}}, \\
    \complexify{\un{n, \mathbb{H}}} &= \text{USp}(2n) &&:= \un{2n}\cap\symp{2n, \mathbb{C}}.
\end{alignat}
\end{remark}

\subsection{$\Ortho_\beta(n)$, orthogonal groups: complex and quaternion}\label{sec:orthogonal}
Although the complex/quaternionic analogues of the orthogonal group (defined with $M^\dagger$) are often referred to as the complex and quaternionic unitary group, there exist two alternative Lie groups $\Ortho_\beta(n)$, $\beta = 2, 4$ which go by the name complex and quaternionic orthogonal groups. The \textit{complex orthogonal group}\footnote{$\text{SL}(n, \mathbb{C})$, $\text{SO}(n, \mathbb{C})$, $\symp{2n, \mathbb{C}}$ are the three complex Lie groups (complex manifolds) with semisimple complex Lie algebras.} is defined in a way that is reminiscent of the real orthogonal group using $M^T$ for $M\in\mathbb{C}^{n\times n}$, 
\begin{equation}
    \ortho{n, \mathbb{C}} := \{G\in\mathbb{C}^{n\times n}| G^TG = I_n\}.
\end{equation}
The \textit{quaternionic orthogonal group} $\ortho{n, \mathbb{H}}$ is a Lie group defined as,
\begin{equation}\label{def:quaternionorthogonal}
    \ortho{n, \mathbb{H}} := \{G\in\mathbb{H}^{n\times n}| G^{D_j}G = I_n\}.
\end{equation}

\par The choice of the $j$-conjugate transpose $M^{D_j}$ instead of $M^T$ to define $\Ortho_\beta(n)$, $\beta = 4$ seems somewhat unnatural but it becomes clear if one consider the identity \eqref{eq:DjcomplexifyT} of the map $\complexify{\,\,\cdot\,\,}$. The complexified $M^{D_j}$ is the regular transpose $\complexify{M}^T$. Thus, $\Ortho_\beta(n)$ all have the property $O^TO=I$ in their complex representations. 

A commonly found symbol in the literature is $\Ortho^*(2n)$, the complexified $\Ortho(n, \mathbb{H})$. 
\begin{equation}\label{eq:orthostar}
    \text{O}^{*}(2n) := \{G\in\Un^*(2n): G^TG= I_{2n}\}=\complexify{\ortho{n, \mathbb{H}}}.
\end{equation}
The standard choice in \eqref{def:quaternionorthogonal} uses $D_j$ but $D_i$ or $D_k$ can be also used instead of $D_j$. Sometimes we explicitly emphasize the choice of $i, j, k$, denoting the group by $\Ortho_i(n, \mathbb{H})$, $\Ortho_j(n, \mathbb{H})$, $\Ortho_k(n, \mathbb{H})$. The group structure is isomorphic for all three cases.

\begin{remark}[the quaternionic matrices satisfying $G^TG=I_n$]\label{rem:quaternioninvolution}
A reader might be curious about the straightforward quaternionic extension of orthogonality by the transpose, a quaternion matrix $G\in\mathbb{H}^{n\times n}$ satisfying
\begin{equation}
    G^TG = I_n.
\end{equation}
However the set of such matrices is not a group, since the relationship $(AB)^T = B^TA^T$ does not hold for quaternion matrices. (In other words, the transpose is not an involution for quaternion matrices.)
\end{remark}

\begin{remark}[the complex perplectic group]\label{rem:complexperplectic}
As mentioned in Remark \ref{rem:realperplectic}, the \textit{complex perplectic group} is another interesting group of matrices defined as:
\begin{equation}\label{def:complexperplectic}
    \{G\in\mathbb{C}^{n\times n}: G^TE_nG = E_n\}.
\end{equation}
Whenever $*=T$, $\mathbb{F} = \mathbb{C}$ are given for the equation $G^*JG=J$, the group determined by any invertible complex symmetric $J$ is isomorphic\footnote{This is the automorphism group of the complex symmetric bilinear form.} to the complex orthogonal group $\ortho{n, \mathbb{C}}$. This is due to the fact that any complex symmetric matrix $J_s$ can be decomposed as $J_s = V^TV$ for some $V$.\footnote{See Lemma \ref{lem:compsymsqrt}. One can also use the Takagi factorization to see this.} An isomorphic mapping $G\mapsto V^{-1}GV$ sends $\{G: G^TJ_sG=J_s\}$ to $\ortho{n, \mathbb{C}}$. See Remark \ref{rem:complexperplecticsvd} for details on the isomorphism for the complex perplectic group and the corresponding structure preserving SVD.  
\end{remark}

\subsection{Lie groups and Lie algebras}
Until this point, we have discussed exclusively the classical Lie groups. Generally, a \textit{Lie group} is defined as a differentiable manifold that possesses a group structure. Note that some Lie groups may not have a matrix representation. A \textit{Lie algebra} can be thought as the tangent space of the Lie group at the identity. For example, the Lie group $G = \ortho{n}$ has its tangent space at the identity $\mathfrak{g}=\mathfrak{o}(n)$, the set of all skew-symmetric matrices. This is often denoted by $\mathfrak{g} = \text{Lie}(G)$. The Lie algebras of classical Lie groups are listed in Appendix \ref{app:liealgebras}.

\subsection{Some auxiliary matrices}\label{sec:auxmatrix}
Let us define some useful auxiliary matrices which appears frequently in our 53 matrix factorizations. Let $\theta=(\theta_1, \dots, \theta_n)$ be a real vector in $\mathbb{R}^n$. We define nine auxiliary matrices $C_n^\theta, S_n^\theta$, $D_n^\theta, R_n^\theta, Ch_n^\theta, Sh_n^\theta, \Sigma_n^\theta, B_n^{\eta, \theta}$ and $H_{m, n}^\theta$ as in Table \ref{tab:auxmatrix}. For simplicity if $n, \theta$ is clear from the context we just denote them by $C, S, D, R, Ch, Sh, \Sigma, B, H$. 

\begin{table}[h]
{\tabulinesep=0.4mm
\title{\Large Auxiliary matrices for a real vector $\theta= (\theta_1, \dots, \theta_n)\in\mathbb{R}^n$ \vspace{0.1cm}}

\begin{tabu}{|c|}
\hline 

\makecell{
\textbf{Cosine and Sine diagonal matrices $\boldsymbol{C_n^\theta, S_n^\theta}\in\mathbb{R}^{n\times n}$} \\ 
$C= \myscale{0.8}{\threediag{\cos\theta_1}{}{\ddots}{}{\cos\theta_n}}, \hspace{0.2cm} S= \myscale{0.8}{\threediag{\sin\theta_1}{}{\ddots}{}{\sin\theta_n}}$
} \\ \hline

\makecell{
\textbf{Unitary diagonal matrix $\boldsymbol{D_n^\theta}\in\mathbb{C}^{n\times n}$} \\
$D = \threediag{e^{i\theta_1}}{}{\ddots}{}{e^{i\theta_n}}$ 
} \\ \hline

\makecell{
\textbf{Rotation block diagonal matrix $\boldsymbol{R_n^\theta}\in\mathbb{R}^{n\times n}$} \\
$R=\threediag{R_1}{}{\ddots}{}{R_{\frac{n}{2}}}, \text{ ($n$ even)}\hspace{0.5cm} R=\myscale{0.8}{\begin{bmatrix}1 & & & \\ & R_1 & & \\ & & \ddots & \\ & & & R_{\lfloor\frac{n}{2}\rfloor}\end{bmatrix}}, \text{ ($n$ odd)}$
\\ 
$\text{ where } R_l=\stwotwo{\,\,\,\cos\theta_l}{\sin\theta_l}{-\sin\theta_l}{\cos\theta_l}$
} \\ \hline

\makecell{
\textbf{Hyperbolic Cosine and Sine diagonal matrices $\boldsymbol{Ch_n^\theta, Sh_n^\theta}\in\mathbb{R}^{n\times n}$} \\
$Ch = \myscale{0.8}{\threediag{\cosh\theta_1}{}{\ddots}{}{\cosh\theta_n}},\hspace{0.5cm} Sh = \myscale{0.8}{\threediag{\sinh\theta_1}{}{\ddots}{}{\sinh\theta_n}}$
} \\ \hline

\makecell{
\textbf{Positive real diagonal matrix $\boldsymbol{\Sigma_n^\theta}\in\mathbb{R}^{n\times n}$} \\ 
$\Sigma = \threediag{\sigma_1}{}{\ddots}{}{\sigma_n} := \threediag{e^{\theta_1}}{}{\ddots}{}{e^{\theta_n}}$
} \\ \hline

\makecell{
\textbf{Imaginary (hyperbolic) rotation block diagonal matrix $\boldsymbol{B_n^{\eta, \theta}}\in\mathbb{H}^{n\times n}$} \\
$B= \threediag{B_1^\eta}{}{\ddots}{}{B_{\frac{n}{2}}^\eta},\text{ ($n$ even)}\hspace{0.5cm} B=\myscale{0.8}{ \begin{bmatrix} 1 & & & \\ & B_1^\eta & & \\ & & \ddots & \\ & & & B_{\lfloor\frac{n}{2}\rfloor}^\eta \end{bmatrix}}, \text{ ($n$ odd)}$ \\ 
$\text{ where } B^\eta_l=\stwotwo{\,\,\,\cosh\theta_l}{\eta\sinh\theta_l}{-\eta\sinh\theta_l}{\,\,\,\cosh\theta_l}$, $\eta\in\{i, j, k\}$
} \\ \hline

\makecell{
\textbf{Hyperbolic block rotation matrix $\boldsymbol{H_{m, n}^\theta}\in\mathbb{R}^{(n+m)\times(n+m)}$} \\
$H = \threediag{Ch^\theta_n}{Sh_n^\theta}{I_{m-n}}{Sh_n^\theta}{Ch_n^\theta}$
}\\ \hline
\end{tabu}}
\vspace{0.1cm}
\caption{List of auxiliary matrices frequently used in this work.}
\label{tab:auxmatrix}
\end{table}

Figure \ref{fig:rotations} describes the hyperbolic rotation matrix $H_{1, 1}$ against the usual $2\times 2$ rotation matrix. The matrix $H_{m, n}^\theta$ is a multi-dimensional extension of the hyperbolic rotation matrix.
\begin{figure}[h]
    \title{\Large Rotation vs. Hyperbolic Rotation}
    \centering
    \frame{\includegraphics[width=4.5in]{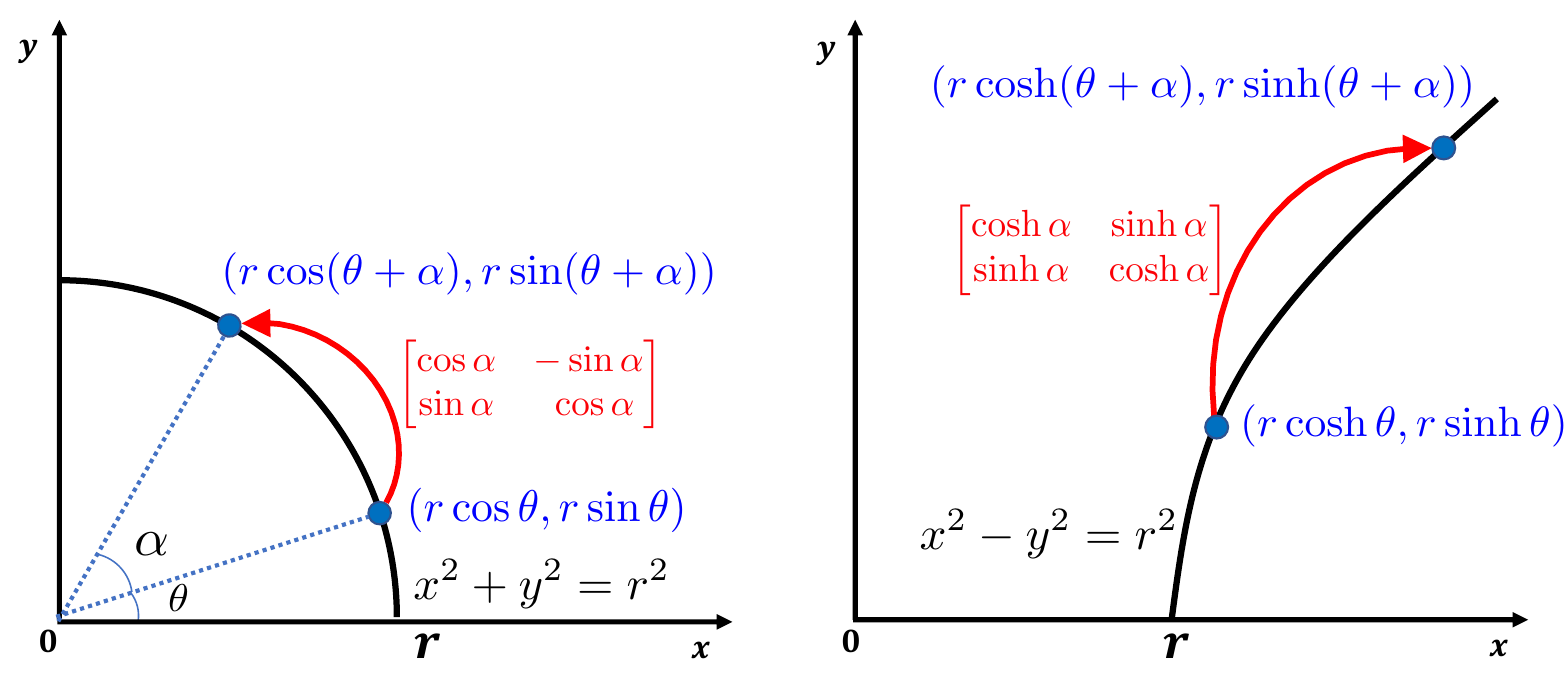}}
    \caption{The (usual) rotation $(R_2^{\alpha})^T$ and the hyperbolic rotation $H_{1, 1}^{\alpha}$.}
    \label{fig:rotations}
\end{figure}




\section{Background}\label{sec:background}

An important result in {\'E}lie Cartan's work on the Riemannian symmetric space \cite{Cartan1,cartan1927certaines,Cartan2} is the \textit{Cartan decomposition} $\mathfrak{g} = \mathfrak{k}+ \mathfrak{p}$. (Some readers may be unfamiliar with $\mathfrak{k}$ pronounced ``k" or ``fraktur k.")
The Cartan decomposition leads to the KAK decomposition, $G = KAK$ (and also extends to compact cases). The term ``decomposition" is used both for $\mathfrak{g} = \mathfrak{k}+\mathfrak{p}$ (additive) and $G= KAK$ (multiplicative). See \eqref{eq:cartandecompositionexample}, \eqref{eq:kakexample} for helpful examples.

\par The decomposition $G = KAK$ itself may have first appeared in Harish-Chandra's work \cite[p.590]{harish1956representations} in 1956. In the 1960s  \cite{Helgason1962,takahashi1963representations,wolf1964self} and also Helgason's 1978 textbook \cite[p.402]{Helgason1978}, the KAK decomposition shows up under the name\footnote{As has been pointed out by Helgason \cite{helgasonprivatecomm}, however, it is unclear that if Cartan himself knew directly about the KAK decomposition as the link between Lie algebras and Lie groups were not yet completed at the time.} ``Cartan decomposition." The name ``KAK decomposition" appears in  the late 1970s \cite{barbasch1979fourier} but it is unclear who first started using the name. Refer to Section \ref{sec:lesson} for more details on the history of the Cartan and the KAK decomposition.

\par The KAK decomposition has been extended to the \textit{generalized Cartan decomposition} by Flensted-Jensen in his 1978 paper \cite{flensted1978spherical}. Originally his work \cite{flensted1978spherical,flensted1980discrete} considered the noncompact cases only but the compact cases were subsequently studied in \cite{hoogenboom1983generalized,hoogenboomthesis} by Hoogenboom. We will call both of them the ``$\kak$ decomposition". In the 1990s, Matsuki extensively studied specific examples   \cite{matsuki1995double,matsuki1997double}, and the idea of visible actions on symmetric spaces \cite{kobayashi2007visible} led to more case studies \cite{Kobayashi,sasaki2010generalized}. Furthermore, the root systems of the compact $\kak$ decompositions have been computed and classified by Matsuki \cite{matsuki2002classification}.

\subsection{Symmetric spaces and the KAK decomposition}\label{sec:symspace}
Let us begin with an elementary example. For the Lie group $G = \gl{n, \mathbb{R}}$, its tangent space at the identity (Lie algebra) is $\mathfrak{g} = \{\text{All $n\times n$ matrices}\}$. Let $\sigma$ be the involution\footnote{Any automorphism $\theta$ on $\mathfrak{g}$ that satisfies $\theta^2 = \text{Id}$ is called an involution.} $X\mapsto -X^T$. Obviously the eigenvalues of $\sigma$ are $\pm 1$ and the following direct sum decomposition $\mathfrak{g} = \mathfrak{k}+ \mathfrak{p}$ holds for the $\pm1$ eigenspaces $\mathfrak{k}, \mathfrak{p}\subset \mathfrak{g}$ of $\sigma$.
\begin{equation}\label{eq:cartandecompositionexample}
    \arraycolsep=1pt
    \hspace{-0.2cm}\begin{array}{ccccc}
        \mathfrak{g} & = & \mathfrak{k} & + & \mathfrak{p} \\
        \{\text{All $n\times n$ matrices}\} & = & \{\text{Skew-symmetric matrices}\}& + & \{\text{Symmetric matrices}\}.
    \end{array}\hspace{-0.5cm}\,
\end{equation}
Recall that $e^{a+b} = e^ae^b$ holds for numbers $a, b$. Certainly this is not true for matrices. However on the set level the following decomposition
\begin{equation}\label{eq:globalcartanexample}
    \arraycolsep=1pt
    \hspace{-0.2cm}\begin{array}{ccccc}
        e^\mathfrak{g} & = & e^\mathfrak{k} & \cdot & e^\mathfrak{p} \\
        \gl{n, \mathbb{R}} & = & \{\text{Orthogonal matrices}\}& \cdot & \{\text{Pos. def. Symmetric matrices}\},
    \end{array}\hspace{-0.5cm}\,
\end{equation}
holds\footnote{Some readers may have noticed that we are abusing notations in \eqref{eq:globalcartanexample}. The sets $e^{\mathfrak{g}}$ and $e^{\mathfrak{k}}$ are in fact $\text{SL}(n, \mathbb{R})$ and $\text{SO}(n)$. We multiplied $\{+1, -1\}$ for both Lie groups so that the decomposition \eqref{eq:globalcartanexample} becomes the usual polar decomposition.} as it is the well-known \textit{polar decomposition}. Let us denote $K=\ortho{n}$ and $P=e^{\mathfrak{p}} = \{\text{Pos. def. Symmetric matrices}\}$ so that \eqref{eq:globalcartanexample} is $G = K\cdot P$. 

\par The (additive) decomposition \eqref{eq:cartandecompositionexample} is the famous \textit{Cartan decomposition}. Moreover it is not a coincidence that \eqref{eq:globalcartanexample} holds, as in Lie theory it is proved that \eqref{eq:globalcartanexample} is true given \eqref{eq:globalcartanexample} is from the Cartan decomposition of any given $G$. The (multiplicative) decomposition \eqref{eq:globalcartanexample} is sometimes called the \textit{global Cartan decomposition}.

\par The involution $\sigma$ is a key ingredient that connects these decompositions. In fact, the involution $\sigma(X)=-X^T$ on $\mathfrak{g}$ is not an arbitrary involution. It is a unique involution called the \textit{Cartan involution}. To be a Cartan involution, the subgroup $K\subset G$ must be a maximal compact subgroup. For example, if we try to add even a single nonorthogonal matrix to $\ortho{n}$, the generated group is no longer compact. Therefore, the Cartan decomposition is also unique.

\par Cartan noticed that the maximal abelian subgroup (largest commuting subgroup) of $P$ plays an important role. Let $A$ be a maximal abelian subgroup of $P$, which is the set of all positive diagonal matrices for our example. (Obviously, positive diagonal matrices form a group, as they are closed under multiplication, are invertible, include the identity, etc., even though the positive definite matrices are not a group.) Also let $\mathfrak{a}\subset\mathfrak{p}$ be the Lie algebra of $A$ (all diagonal matrices). Then, the following decompositions\footnote{The decomposition of $\mathfrak{p}$ (left) is what Cartan wrote in \cite{cartan1927certaines}, $\mathfrak{p}=\bigcup_{k\in K}\text{Ad}(k)\mathfrak{a}$, also mentioned in our introduction. Notationally, we write $\mathfrak{p} = \bigcup_{k\in K}k\mathfrak{a}k^{-1}$ rather than the wrong set notation $\mathfrak{p} = K\mathfrak{a} K^{-1}$ as the latter does not denote the constraint that the $k$ on the left is the inverse of the $k$ on the right, hence the union symbol.} hold:
\begin{equation}\label{eq:symeigdecompositionexample}
    \mathfrak{p} = \bigcup_{k\in K}k\mathfrak{a}k^{-1} {\hspace{0.5cm}\xRightarrow{\exp}\hspace{0.5cm}} P = \bigcup_{k\in K}kAk^{-1}.
\end{equation}

In linear algebra the left side of the arrow in \eqref{eq:symeigdecompositionexample} is known as the symmetric eigendecomposition (and on the right, the positive definite case). Finally combining \eqref{eq:globalcartanexample} and \eqref{eq:symeigdecompositionexample}, we obtain the \textit{KAK decomposition} $G = KAK$,
\begin{equation}\label{eq:kakexample}
    \arraycolsep=1pt
    \hspace{-0.2cm}\begin{array}{ccccccc}
        G & = & K &\cdot & A &\cdot & K \\
        \gl{n, \mathbb{R}} & = & \{\text{Orthogonal}\}& \cdot & \{\text{Positive Diagonal}\} & \cdot & \{\text{Orthogonal}\}.
    \end{array}\hspace{-0.5cm}\,
\end{equation}
In the current example \eqref{eq:kakexample} is the SVD of square invertible matrices. 

\par We point out that $G$ is required to be a noncompact Lie group to have the Cartan decomposition. However, decompositions \eqref{eq:symeigdecompositionexample}, \eqref{eq:kakexample} are also available for compact Lie groups. By the following simple technique \cite[Proposition 7.15]{knapp2013lie} (Weyl's unitary trick) we can obtain a compact Lie group $U$ from a noncompact $G$. Let us define the Lie algebra $\mathfrak{u}$ from the Cartan decomposition $\mathfrak{g} = \mathfrak{k}+\mathfrak{p}$ by multiplying the imaginary unit $i$ on the $\mathfrak{p}$ part:
\begin{equation*}
    \mathfrak{u} := \mathfrak{k} + i\mathfrak{p}.
\end{equation*}
Continuing with the example $G = \gl{n, \mathbb{R}}$ the Lie algebra $\mathfrak{u}$ is $\{\text{Skew-symmetric}\}+ i\{\text{Symmetric}\}$, the set of all skew-Hermitian matrices. Letting $U:=e^{\mathfrak{u}}$ we obtain the compact Lie group $U = \un{n}$, the unitary group.\footnote{As we discussed in the footnote after \eqref{eq:globalcartanexample}, $e^{\mathfrak{u}}$ is in fact the Lie group $\text{SU}(n)$, not $\un{n}$. The group $U$ can be precisely defined as the group such that $K\subset U$ and $\text{Lie}(U) = \mathfrak{u}$.} For $A' = \exp(i\mathfrak{a})$, the set of all unit diagonal matrices, we have the compact counterpart of \eqref{eq:kakexample}, $U = KA'K$.

\par We will be concerned with the triples $(G,U,K)$ with the properties thus described. These triples arose historically in connection with Cartan's theory of \textit{Riemannian symmetric spaces}. Both the noncompact and the compact quotient spaces $G/K$ and $U/K$ are symmetric spaces.

\par Cartan classified all possible irreducible Riemannian globally symmetric spaces \cite{cartan1927certaines}. Table \ref{tab:cartanss} is the complete list of triples $(G, U, K)$ for infinite families of Cartan's Riemannian symmetric spaces. Note that for a fixed $U$, a subgroup $K$ which makes $U/K$ a Riemannian symmetric space might not be unique.

\begin{table}[h]
    \centering
    {Cartan's 20 symmetric spaces  \\ $G/K$ (noncompact), $U/K$ (compact)}
    \begin{tabular}{|c|c|c|c|}
    \hline
    \makecell{\small Cartan's\\\small Type} & $G$ (noncompact) & $U$ (compact) & Subgroup $K$ \\ \hline
    A & $\gl{n, \mathbb{C}}$ & $\un{n}\times \un{n}$ & $\un{n}$\\
    A$\RN{1}$ & $\gl{n, \mathbb{R}}$ & $\un{n}$ & $\ortho{n}$ \\  
    A$\RN{2}$ & $\gl{n, \mathbb{H}}$ & $\un{2n}$ & $\un{n, \mathbb{H}}$ (or $\usp{2n}$) \\
    A$\RN{3}$ & $\un{p,q}$ & $\un{n}$ & $\un{p}\times\un{q}$ \\ \hdashline
    BD & $\ortho{n, \mathbb{C}}$ & $\ortho{n}\times \ortho{n}$ & $\ortho{n}$\\
    BD$\RN{1}$ & $\ortho{p,q}$ & $\ortho{n}$ & $\ortho{p}\times\ortho{q}$ \\
    D$\RN{3}$ & $\Ortho^*(2n)$ & $\ortho{2n}$ & $\osp{2n}$\\ \hdashline
    C & $\symp{2n, \mathbb{C}}$ & $\usp{2n}\times \usp{2n}$ & $\usp{2n}$\\
    C$\RN{1}$ & $\symp{2n, \mathbb{R}}$ & $\usp{2n}$ & $\osp{2n}$\\
    C$\RN{2}$ & $\un{p, q, \mathbb{H}}$ & $\un{n, \mathbb{H}}$ & $\un{p, \mathbb{H}}\times \un{q, \mathbb{H}}$\\ \hline
    \end{tabular}
    \vspace{0.1cm}
    \caption{The full list of triples $(G, U, K)$. Cartan's irreducible Riemannian symmetric spaces are $G/K$ (noncompact) and $U/K$ (compact). Refer to Table \ref{tab:translation} for Lie group notation. (We do not consider exceptional types here.) For types A, BD, C which does not seem to have proper $U/K$ representations, refer to \cite{Helgason1978} for details.}
    \label{tab:cartanss}
\end{table}

\par Regarding the theory of symmetric spaces, there are a number of classic textbooks: Helgason \cite{Helgason1978,Helgason1984}, Knapp \cite{knapp2013lie}, Gilmore \cite{gilmore2012lie} and many more. The authors also describe some key ideas with modern linear algebra in Section 3 of \cite{edelman2020generalized}. 

\par Let us formally state what we have discussed so far in this section. Let $G$ be a noncompact semisimple Lie group with the Lie algebra $\mathfrak{g}$. There exists a unique (up to isomorphism) Cartan involution $\sigma$ on $\mathfrak{g}$ with $\pm 1$ eigenspaces $\mathfrak{k}, \mathfrak{p}\subset\mathfrak{g}$. Then, the Cartan decomposition $\mathfrak{g} = \mathfrak{k} +\mathfrak{p}$ holds. Let $K$ be the analytic subgroup of $G$ with $\text{Lie}(K)=\mathfrak{k}$, and let $P=\exp(\mathfrak{p})$. The product map $K\times P \to G$ is a diffeomorphism. The quotient $G/K$ is a noncompact Riemannian manifold, thus it is a (noncompact) Riemannian symmetric space. Define $\mathfrak{u} := \mathfrak{k} + i\mathfrak{p}$ and let $U$ be the Lie group such that $K\subset U$ and $\text{Lie}(U) = \mathfrak{u}$. The group $U$ is compact and the quotient $U/K$ is a compact Riemannian symmetric space.

\par Let $\mathfrak{a}$ be a maximal abelian subalgebra\footnote{To compute $\mathfrak{a}$ one has to consider small examples and then expand to larger cases. In textbooks such as Helgason the list of $\mathfrak{a}$ is given as there are only a finite number of $\mathfrak{p}$'s for Cartan's symmetric spaces. The authors are not aware of the general procedure of computing a maximal abelian subalgebra $\mathfrak{a}$ of $\mathfrak{p}$ (or $\mathfrak{p}_\tau\cap\mathfrak{p}_\sigma$, which appears in the following section). However, we propose a conjecture that the odd powers of any given generic element $p\in\mathfrak{p}$ construct a basis of a maximal abelian subalgebra.} of $\mathfrak{p}$ and $A := \exp(\mathfrak{a})$. Then, $\mathfrak{p} = \cup_{k\in K}\text{Ad}(k)\!\cdot\!\mathfrak{a} = \cup_{k\in K}k \mathfrak{a} k^{-1}$ holds and on the group level $P = \cup_{k\in K}\text{Ad}(k)\!\cdot\! A$ holds. Finally, we obtain the KAK decomposition. 
\begin{theorem}[KAK decomposition, {\cite[Theorem 6.7, p.249]{Helgason1978}}]
For $(G, U, K)$ discussed above, let $\mathfrak{a}$, $\mathfrak{a}'$ be maximal abelian subalgebras of $\mathfrak{p}$ and $i\mathfrak{p}$, respectively. Then for $A = \exp(\mathfrak{a})$ and $A' = \exp(\mathfrak{a}')$ we have,
\begin{equation*}
    G = KAK \hspace{0.5cm}\text{and}\hspace{0.5cm} U = KA'K.
\end{equation*}
\end{theorem}

\par For a Lie group $G$ there also exist non-Cartan involutions, specifically involutions where the subgroup $K$ (whose Lie algebra is the $+1$ eigenspace) is noncompact. For example, the map $\tau:X\mapsto I_{p,q}XI_{p,q}$ is another involution on $\mathfrak{gl}(n, \mathbb{R})$ for $n=p+q$. The $\pm1$ eigenspace decomposition $\mathfrak{g} = \mathfrak{k}_\tau + \mathfrak{p}_\tau$ is the following:
\begin{equation*}
    \arraycolsep=1pt
    \begin{array}{ccccc}
        \mathfrak{g} & = & \mathfrak{k}_\tau & + & \mathfrak{p}_\tau \\
        \mathfrak{gl}(n, \mathbb{R}) & = & \bigg\{\begin{bmatrix} \,\,a\,\, & \,\,0\,\,\\0 & d \end{bmatrix} \bigg|\,\,a\in\mathbb{R}^{p\times p}, d\in\mathbb{R}^{q\times q}\bigg\} 
        & + & \bigg\{\begin{bmatrix}
    0 & \,\,b\,\,\\c^T & 0
    \end{bmatrix}\bigg|\,\,b, c\in\mathbb{R}^{p\times q}\bigg\}.
    \end{array}
\end{equation*}

\par The subgroup $K_\tau\subset G$ is $\gl{p, \mathbb{R}} \oplus \gl{q, \mathbb{R}}$ (block diagonal sum), which is clearly not compact. For a non-Cartan involution $\tau$ the quotient $G/K_\tau$ is a pseudo-Riemannian differentiable manifold (pseudo-Riemannian means the metric is indefinite), thus $G/K_\tau$ is a \textit{pseudo-Riemannian symmetric space}. Pseudo-Riemannian symmetric spaces play an important role in the upcoming section. 

\begin{remark}[The use of $\gl{n, \mathbb{R}}$ rather than $\text{SL}(n, \mathbb{R})$]
The noncompact Lie group $G$ as one can find in many Lie group textbooks, is usually assumed to be semisimple (or more precisely, has a semisimple Lie algebra). However the Lie groups introduced in Table \ref{tab:translation} are mostly \textit{reductive} Lie groups which can be reduced to semisimple cases by scalar multiplication. A good example would be the \textit{special linear group} $\text{SL}(n, \mathbb{R})$ ($n\times n$ real matrices with determinant one) which is semisimple and connected. The corresponding reductive Lie group we use is the Lie group $\gl{n, \mathbb{R}}$ which is neither semisimple nor connected. We multiply a factor $\mathbb{R}$, abusing the theory of semisimple Lie algebra and applying them on reductive groups. For example the Cartan decomposition for $\text{SL}(n, \mathbb{R})$ would be 
\begin{equation*}
    \mathfrak{sl}(n, \mathbb{R}) = \mathfrak{so}(n)+\{\text{All traceless symmetric matrices}\},
\end{equation*}
whereas we use the reductive version \eqref{eq:cartandecompositionexample}. 
\end{remark}

\subsection{The generalized Cartan decomposition}\label{sec:gencartantheory}
To construct a generalized Cartan decomposition, we need two ingredients. The first ingredient is a Riemannian noncompact symmetric space $G/K_\sigma$. Again, ``Riemannian" implies the subgroup $K_\sigma$ (usually denoted by $K$) is a maximal compact subgroup of $G$, and $\sigma$ is the (unique) Cartan involution. The second ingredient is a pseudo-Riemannian symmetric space $G/K_\tau$ arising from a (Cartan or non-Cartan) involution $\tau$ on $\mathfrak{g}$ (the Lie algebra of $G$), which satisfies $\tau\sigma=\sigma\tau$. Now we are ready.

\par The Cartan decomposition arising from $\sigma$ is the following:
\begin{equation*}
    \mathfrak{g}=\mathfrak{k}_\sigma+\mathfrak{p}_\sigma.
\end{equation*}
Since we have another involution $\tau$, we have another decomposition of $\mathfrak{g}$, 
\begin{equation*}
    \mathfrak{g} = \mathfrak{k}_\tau+\mathfrak{p}_\tau, 
\end{equation*}
where $\mathfrak{k}_\tau$ and $\mathfrak{p}_\tau$ are $\pm 1$ eigenspaces of $\tau$, respectively. The subgroup $K_\tau$ of $G$ is the subgroup which has $\mathfrak{k}_\tau$ as its tangent space. We call this the \textit{generalized Cartan triple}.
\begin{equation}\label{def:gencartantriple}
    (G, K_\sigma, K_\tau)
\end{equation}

\par The last step is to take the intersection $\mathfrak{p}_\sigma\cap\mathfrak{p}_\tau$ and compute a maximal abelian subalgebra $\mathfrak{a}$ inside $\mathfrak{p}_\sigma\cap\mathfrak{p}_\tau$. The choice of $\mathfrak{a}$ may not be unique, as we will see at the end of Section \ref{sec:gencartanex2}. (This is also true for the KAK decomposition.) Once $\mathfrak{a}$ is selected we quotient out the symmetry to make $a\in \exp(\mathfrak{a})$ unique for the factorization, by fixing a Weyl chamber\footnote{For example if $\exp(\mathfrak{a})$ is the set of all positive diagonal matrices, fixing a Weyl chamber to obtain $\mathfrak{a}^+$ is equivalent to setting the diagonal entries of $a\in\exp(\mathfrak{a}^+)$ in a decreasing order.} and computing $\mathfrak{a}^+$. Finally we get the subgroup $A=\exp(\mathfrak{a}^+)$. The generalized Cartan decomposition \cite{flensted1978spherical} follows. 

\begin{theorem}[generalized Cartan ($\kak$) decomposition {\cite[Theorem 4.1]{flensted1978spherical}}]\label{thm:gencartan}
Let $\mathfrak{g}$ be a noncompact semisimple Lie algebra over $\mathbb{R}$. Suppose $G$ is a connected Lie group with Lie algebra $\mathfrak{g}$. Let $\sigma$ be its Cartan involution on $\mathfrak{g}$ and $\tau$ be any involution on $\mathfrak{g}$ such that $\tau\sigma = \sigma\tau$. Define $K_\sigma$, $K_\tau$ and $A$ as above. Then we have the group decomposition
\begin{equation*}
    G = K_\sigma A K_\tau.
\end{equation*}
More precisely, for any $g\in G$, there exists a unique $a\in A$ such that $g = k_\sigma a k_\tau$ for some $k_\sigma\in K_\sigma$ and $k_\tau\in K_\tau$. 
\end{theorem}

\par We can write the procedure to obtain a matrix factorization by Theorem \ref{thm:gencartan}, as Algorithm \ref{alg:K1AK2noncompact}. The following sections then will present step-by-step examples. The list of Lie algebras $\mathfrak{g}$ of the classical Lie groups can be found in Appendix \ref{app:liealgebras}. The involutions $\tau$ for our 53 matrix factorizations are listed in Appendix \ref{app:involutions}. 
\begin{algorithm}\label{alg:K1AK2noncompact}
\caption{Create a new factorization $G = K_\sigma A K_\tau$ (noncompact)}
\Inputone{$G$: Choose a classical Lie group $G$ from Table \ref{tab:translation}}
\nonl \hspace{1.4cm} $K_\sigma$: Automatically given as the corresponding $K$ in Table \ref{tab:cartanss}\\
\Inputtwo{$\tau$: Choose an involution on $\mathfrak{g}$, the tangent space of $G$ at $I$.}
Compute decomposition $\mathfrak{g}=  \mathfrak{k}_\tau+\mathfrak{p}_\tau$ \\ 
\nonl Then, $K_\tau= \exp(\mathfrak{k}_\tau)^\star$ \\
Compute the Cartan decomposition $\mathfrak{g}= \mathfrak{k}_\sigma+\mathfrak{p}_\sigma$ \\
\nonl Choose $\mathfrak{a}$ = \text{Maximal abelian subalgebra of $\mathfrak{p}_\sigma\cap\mathfrak{p}_\tau$}\\
Let $A = \exp(\mathfrak{a})$\\
\KwOut{Factorization $G = K_\sigma A K_\tau$} 
\nonl \hfill\tiny{$\star$ More precisely, $K_\tau\subset G$ is the analytic subgroup with Lie algebra $\mathfrak{k}_\tau$}
\end{algorithm}

\subsection{A first example: factorization $\fact{13}$, complex (Theorem \ref{thm:F13complex})}\label{sec:gencartanex1}
\begin{align*}
    &\text{Input 1. Choose }G = \gl{n, \mathbb{C}}. \text{ From Table \ref{tab:cartanss} (first row), }K_\sigma=\un{n}. \\
    &\text{Input 2. Involution $\tau(X)$ : }X\mapsto -X^T \text{ on } \mathfrak{g} = \mathfrak{gl}(n, \mathbb{C}).
\end{align*}

\noindent\textit{\underline{Step 1. Compute $K_\tau$, $\mathfrak{p}_\tau$.}}
\par $\mathfrak{g}$ is the set of all complex $n\times n$ matrices. The decomposition $\mathfrak{g} = \mathfrak{k}_\tau + \mathfrak{p}_\tau$ is,
\begin{equation*}
    \mathfrak{g} = \{\text{complex skew-symmetric matrices}\} + \{\text{complex symmetric matrices}\},
\end{equation*}
where we obtain $\mathfrak{p}_\tau$ immediately, and the group $K_\tau$ is obtained as $\exp(\mathfrak{k}_\tau)$,\footnote{As noted in $\star$ of Algorithm \ref{alg:K1AK2noncompact}, directly exponentiating $\mathfrak{k}_\tau$ gives the subgroup of $K_\tau$, elements with positive determinants. One need to multiply the factor $\pm 1$ to match the group $\ortho{n, \mathbb{R}}$. Such factors can be offsetted by quotienting out the symmetries in $\mathfrak{a}$.} 
\begin{equation*}
    K_\tau = \{\text{All complex orthogonal matrices}\} = \ortho{n, \mathbb{C}}.
\end{equation*}

\noindent\textit{\underline{Step 2. Compute $\mathfrak{a}$.}}
\par The Cartan decomposition $\mathfrak{g} = \mathfrak{k}_\sigma + \mathfrak{p}_\sigma$ is
\begin{equation*}
    \mathfrak{g} = \{\text{All skew-Hermitian matrices}\} + \{\text{All Hermitian matrices}\}.
\end{equation*}
The intersection $\mathfrak{p}_\sigma \cap \mathfrak{p}_\tau$ is the intersection of all Hermitian matrices with all complex symmetric matrices, 
\begin{equation*}
    \mathfrak{p}_\sigma \cap \mathfrak{p}_\tau = \{\text{All real symmetric matrices}\}.
\end{equation*}
An obvious choice of the maximal abelian subalgebra is the set of all diagonal matrices $\mathfrak{a}$.

\noindent\textit{\underline{Step 3. Complete $G = K_\sigma A K_\tau$.}}
\par Taking the exponential of $\mathfrak{a}$ we obtain $A$, the set of all positive diagonal matrices. The group $K_\sigma$ is given as $\un{n}$ and $K_\tau$ given in step 1, is $\ortho{n, \mathbb{C}}$. The matrix factorization we obtain as $G= K_\sigma A K_\tau$ is the following matrix factorization,
\begin{equation*}
    \text{For any }\glmat{n}{C}\in\gl{n, \mathbb{C}}, \hspace{0.5cm} \glmat{n}{C} = U_n \Sigma O_n^\mathbb{C},
\end{equation*}
where $U_n$ is an $n\times n$ unitary matrix and $O_n^\mathbb{C}$ is an $n\times n$ complex orthogonal matrix. $\Sigma$ is a positive diagonal matrix unique up to the diagonal permutation. This is exactly the factorization $\fact{13}$ complex, Theorem \ref{thm:F13complex} in Section \ref{sec:F13}. 

\subsection{Another example: factorization $\fact{8}$, real (Theorem \ref{eq:F8real})}\label{sec:gencartanex2}
\begin{align*}
    &\text{Input 1. Choose }G = \gl{n, \mathbb{R}}. \text{ From Table \ref{tab:cartanss} (second row), }K_\sigma = \ortho{n}. \\
    &\text{Input 2. Involution $\tau(X)$ : }X\mapsto I_{p, q}XI_{p, q} \text{ on }\mathfrak{g} = \mathfrak{gl}(n, \mathbb{R}).
\end{align*}

\noindent\textit{\underline{Step 1. Compute $K_\tau$, $\mathfrak{p}_\tau$.}}
\par $\mathfrak{g}$ is the set of all real $n\times n$ matrices. Compute the decomposition $\mathfrak{g}=\mathfrak{k}_\tau+\mathfrak{p}_\tau$,
\begin{equation*}
    \mathfrak{g} =  
    \bigg\{\begin{bmatrix}
    a & 0\\0 & d
    \end{bmatrix}\bigg|a\in\mathbb{R}^{p\times p}, d\in\mathbb{R}^{q\times q}\bigg\} + 
    \bigg\{\begin{bmatrix}
    0 & b\\c^T & 0
    \end{bmatrix}\bigg|b, c\in\mathbb{R}^{p\times q}\bigg\}.
\end{equation*}
We obtain $\mathfrak{p}_\tau$ immediately, and the group $K_\tau$ is obtained from $\exp(\mathfrak{k}_\tau)$, 
\begin{equation}\label{eq:firstexKtau}
    K_\tau = \bigg\{\begin{bmatrix}
    A & 0\\0 & D
    \end{bmatrix}\bigg|A\in\gl{p, \mathbb{R}}, D\in\gl{q, \mathbb{R}}\bigg\}.
\end{equation}

\noindent\textit{\underline{Step 2. Compute $\mathfrak{a}$.}}
\par The Cartan decomposition $\mathfrak{g} = \mathfrak{k}_\sigma+\mathfrak{p}_\sigma$ is 
\begin{equation*}
    \mathfrak{g} = \{\text{All skew-symmetric matrices}\} + \{\text{All symmetric matrices}\}.
\end{equation*}
The intersection $\mathfrak{p}_\tau\cap\mathfrak{p}_\sigma$ becomes
\begin{equation}\label{eq:example2intersection}
    \mathfrak{p}_\tau\cap\mathfrak{p}_\sigma = \bigg\{\begin{bmatrix}
    0 & b\\b^T & 0
    \end{bmatrix}\bigg|b\in\mathbb{R}^{p\times q}\bigg\},
\end{equation}
and a choice of the maximal abelian subalgebra $\mathfrak{a}\subset\mathfrak{p}_\tau\cap\mathfrak{p}_\sigma$ is the set of all matrices of the form
\begin{gather}
    \nonumber
    \hspace{0.6cm}{\overbrace{\hspace{1.8cm}}^{p}\overbrace{\hspace{1cm}}^{q}}
    \\[-4pt]
    h = 
    \begin{bmatrix}
    \begin{array}{c:c}
    \makebox(33, 38){}& 
    \makebox(17, 38){$\begin{smallmatrix}\theta_1\hspace{0.6cm}\\ \diagdown \\ \hspace{0.6cm}\theta_q  \\ \rule{0cm}{0.5cm}\end{smallmatrix}$}\\ \hdashline
    \makebox(38,17){$\begin{smallmatrix}\theta_1\hspace{1.4cm}\\ \diagdown\hspace{0.8cm}\\ \theta_q\hspace{0.2cm}\end{smallmatrix}$}
    & \makebox(17, 20){}
    \end{array}\label{eq:example2maxabelian}
    \end{bmatrix}.
\end{gather}

\noindent\textit{\underline{Step 3. Complete $G = K_\sigma A K_\tau$.}}
\par Taking the exponential of a matrix $h$ above (with the positive $\theta_l$'s), we obtain the subgroup $A$, consisting of $a$ such that (in fact, this is $H_{p, q}^\theta$ of Table \ref{tab:auxmatrix}):
\begin{equation*}
    a = \begin{bmatrix}
    \begin{array}{c:c}
    \makebox(33, 38){$\begin{smallmatrix}\diagdown\hspace{1.3cm}\\ \cosh\theta_l \hspace{0.7cm}\\ \diagdown \hspace{0.1cm} \\ \hspace{0.7cm}\rule{0cm}{0.4cm}\text{\large{$I_{p-q}$}}\end{smallmatrix}$}& 
    \makebox(17, 38){$\begin{smallmatrix}\diagdown\hspace{0.6cm}\\ \sinh\theta_l \\ \hspace{0.6cm}\diagdown  \\ \rule{0cm}{0.5cm}\end{smallmatrix}$}\\ \hdashline
    \makebox(38,17){$\begin{smallmatrix}\diagdown\hspace{1.2cm}\\ \sinh\theta_l \hspace{0.7cm}\\ \diagdown\hspace{0.1cm}\end{smallmatrix}$}
    & \makebox(17, 20){$\begin{smallmatrix}\diagdown\hspace{0.7cm}\\ \cosh\theta_l\\ \hspace{0.5cm}\diagdown\end{smallmatrix}$}
    \end{array}
    \end{bmatrix}.
\end{equation*}
The group $K_\sigma$ is $\ortho{n}$ and $K_\tau$ is given in \eqref{eq:firstexKtau}. 

\par The computed generalized Cartan decomposition $G = K_\sigma AK_\tau$ is the following matrix factorization. (Theorem \ref{thm:F8real})
\begin{equation*}
    \text{For }\glmat{n}{R}\in\gl{n, \mathbb{R}}, \hspace{0.5cm} \glmat{n}{R} = O_n \begin{bmatrix}\begin{array}{cc:c}
        Ch & & Sh \\
         & I_{p-q} & \\ \hdashline
         Sh & & Ch
    \end{array}\end{bmatrix}
    \begin{bmatrix}
    \begin{array}{c:c}
    \glmat{p}{R} & \\ \hdashline & \glmat{q}{R}\end{array}
    \end{bmatrix},
\end{equation*}
where $Ch, Sh\in\mathbb{R}^{q\times q}$ are diagonal matrices with $\cosh, \sinh$ values of $\theta_l$'s, $\glmat{p}{R}\in\gl{p, \mathbb{R}}$, $\glmat{q}{R}\in\gl{q, \mathbb{R}}$ and $O_n$ is an $n\times n$ orthogonal matrix. This is $\fact{8}$ (real) matrix factorization which will appear in Section \ref{sec:F8}.

\hspace{0.1cm}

\noindent\textit{\underline{The choice of a maximal abelian subalgebra $\mathfrak{a}$ changes the matrix factorization.}}
As mentioned in Section \ref{sec:symspace}, the choice of $\mathfrak{a}$ inside $\mathfrak{p}_\sigma\cap\mathfrak{p}_\tau$ is not unique. The choice \eqref{eq:example2maxabelian} inside \eqref{eq:example2intersection} is the standard choice for obtaining the tridiagonal matrix $a$. Nonetheless, other choices of $\mathfrak{a}$ are isomorphic to the standard choice but may have different matrix representations. For example, another choice of $\mathfrak{a}$ is the collection of the following $h$ matrices:

\begin{equation*}
    h = \begin{bmatrix}
    \begin{array}{c:c}
    \makebox(33, 38){}& 
    \makebox(17, 38){$\begin{smallmatrix}\hspace{0.6cm}\theta_1\\ \text{\reflectbox{$\diagdown$}}\hspace{0.1cm} \\ \theta_q\hspace{0.6cm}  \\ \rule{0cm}{0.5cm}\end{smallmatrix}$}\\ \hdashline
    \makebox(38,17){$\begin{smallmatrix}\theta_1\hspace{0.2cm}\\ \text{\reflectbox{$\diagdown$}}\hspace{0.9cm}\\ \theta_q\hspace{1.4cm}\end{smallmatrix}$}
    & \makebox(17, 20){}
    \end{array}
    \end{bmatrix}.
\end{equation*}
In this case the subgroup $A$ becomes the collection of the following $a$ matrices:
\begin{equation}\label{eq:alternativeforms}
    a = \begin{bmatrix}
    \begin{array}{c:c}
    \makebox(33, 38){$\begin{smallmatrix}\diagdown\hspace{1.3cm}\\ \cosh\theta_l \hspace{0.7cm}\\ \diagdown \hspace{0.1cm} \\ \hspace{0.8cm}\rule{0cm}{0.4cm}\text{\large{$I_{p-q}$}}\end{smallmatrix}$}& 
    \makebox(17, 38){$\begin{smallmatrix}\hspace{0.6cm}\text{\reflectbox{$\diagdown$}}\\ \sinh\theta_l \\ \text{\reflectbox{$\diagdown$}}\hspace{0.6cm}  \\ \rule{0cm}{0.5cm}\end{smallmatrix}$}\\ \hdashline
    \makebox(38,17){$\begin{smallmatrix} \text{\reflectbox{$\diagdown$}}\hspace{0.1cm}\\ \sinh\theta_l \hspace{0.7cm}\\ \text{\reflectbox{$\diagdown$}}\hspace{1.2cm}\end{smallmatrix}$}
    & \makebox(17, 20){$\begin{smallmatrix}\diagdown\hspace{0.7cm}\\ \cosh\theta_l\\ \hspace{0.5cm}\diagdown\end{smallmatrix}$}
    \end{array}
    \end{bmatrix}.
\end{equation}
The $\kak$ decomposition also holds for such alternative choices of $A$.

\subsection{Compact case and $\kak$ decomposition}
If the group $G$ is compact, the situation is slightly different. As we discussed in Section \ref{sec:symspace}, a compact Lie group $G$ might have multiple choices of $K$ to make $G/K$ a Riemannian manifold. Thus, the two ingredients are $G/K_\sigma$ and $G/K_\tau$, both being Riemannian symmetric spaces. Then the later steps are nearly identical, as one needs to compute the maximal abelian subalgebra $\mathfrak{a}\subset\mathfrak{p}_\sigma\cap\mathfrak{p}_\tau$. (In fact, the subgroup $A$ is a torus.) See Theorem \ref{thm:gencartancompact} and Algorithm \ref{alg:K1AK2compact} for the detail.

\begin{theorem}[generalized Cartan ($\kak$) decomposition, compact case {\cite[Theorem 3.6]{hoogenboom1983generalized}}]\label{thm:gencartancompact}
Let $G$ be a compact semisimple Lie group with $\text{Lie}(G) = \mathfrak{g}$. Let $\sigma$, $\tau$ be two commuting involutions such that $G/K_\sigma$, $G/K_\tau$ are Riemannian symmetric spaces. Let $\mathfrak{g} = \mathfrak{k}_\sigma + \mathfrak{p}_\sigma$ and $\mathfrak{g} = \mathfrak{k}_\tau + \mathfrak{p}_\tau$ be the $\pm 1$ eigenspace decompositions of $\mathfrak{g}$. Suppose $\mathfrak{a}\subset(\mathfrak{p}_\sigma \cap\mathfrak{p}_\tau)$ be a maximal abelian subalgebra. Fix a Weyl chamber to get $\mathfrak{a}^+$ and let $A=\exp(\mathfrak{a}^+)$. Then, the generalized Cartan decomposition of $G$ follows:
\begin{equation*}
    G = K_\sigma A K_\tau.
\end{equation*}
\end{theorem}

\begin{algorithm}\label{alg:K1AK2compact}
\caption{Create a new factorization $G = K_\sigma A K_\tau$ (compact)}
    \Inputone{Compact Lie group $G$ with Lie algebra $\mathfrak{g}$}
    \Inputtwo{Choose two involutions $\tau, \sigma$ on $\mathfrak{g}$}
    $\mathfrak{p}_\sigma = \{g\in\mathfrak{g}|\sigma(g)=-g\}$, $K_\sigma = \exp(\{g\in\mathfrak{g}|\sigma(g) = g\})$ \\ 
    \nonl $\mathfrak{p}_\tau = \{g\in\mathfrak{g}|\tau(g)=-g\}$, $K_\tau = \exp(\{g\in\mathfrak{g}|\tau(g) = g\})$ \\ 
    $\mathfrak{a}$ = \text{Maximal abelian subalgebra of $\mathfrak{p}_\sigma\cap\mathfrak{p}_\tau$}\\
    $A = \exp(\mathfrak{a})$\\
    \KwOut{Factorization $G = K_\sigma A K_\tau$}
\end{algorithm}

\subsection{Factorization folding : generalizing the link between the SVD and symmetric eigendecomposition}\label{sec:folding}
The (square) SVD and the symmetric eigendecomposition have an important relationship which can be illustrated as follows.
\begin{equation}\label{eq:foldingrelation}\boxed{
    \hspace{0.5cm}\makecell{A = U\Sigma V^T\\(\text{SVD})} \hspace{0.3cm}{\LARGE{\Longrightarrow}}\hspace{0.3cm} \makecell{AA^T = U\Sigma^2U^T\\(\text{Eigendecomp.})}
}\end{equation}

The SVD is a classical example of the KAK decomposition. Let $G = \gl{n, \mathbb{R}}$, $K = \ortho{n}$ and $P = \{A\in\mathbb{R}^{n\times n}|A\text{ : Symmetric positive definite}\}$. Then the SVD is the KAK decomposition of $G$ and the symmetric positive definite eigendecomposition is a decomposition of $P$. 

\par Let $\tau:X\mapsto X^{-T}$ be the group level involution (which satisfies $\tau(k) = k$, $\tau(p) = p^{-1}$). Then since 
\begin{equation*}
    g\cdot\tau(g)^{-1} = kak'\tau(k')^{-1}\tau(a)^{-1}\tau(k)^{-1} = kak'(k')^{-1}ak^{-1} = ka^2k^{-1},
\end{equation*}
the relationship \eqref{eq:foldingrelation} can be rewritten with $p:=g\cdot\tau(g)^{-1}$ as the following. 
\begin{equation*}\boxed{
    \makecell{g = ka k' \\ \text{(SVD, KAK decomp.)}} 
    \hspace{0.3cm}{{\xRightarrow{\text{Fold}}}}\hspace{0.3cm}
    \makecell{p = k a^2 k^{-1} \\ \text{(Eigendecomp.)}}
}\end{equation*}
We call this the \textit{folding} of the SVD into the eigendecomposition, as we collapse $k'$ using the involution.\footnote{This is equivalent to Kleinsteuber's table of ``normal form methods" in his thesis \cite{kleinsteuber2005jacobi}.} In fact, the folding can be done in both directions (collapsing $k$ or $k'$), and we call the above example as the \textit{right folding}. The \textit{left folding} is a factorization of $p' := \tau(g)^{-1}\cdot g$, as follows.
\begin{equation*}\boxed{
    \makecell{p' = k' a^2 k'^{-1} \\ \text{(Eigendecomp.)}}
    \hspace{0.5cm}{{\xLeftarrow[\text{Left}]{\text{Fold}}}}\hspace{0.5cm}
    \makecell{g = ka k' \\ \text{(SVD)}} 
    \hspace{0.5cm}{{\xRightarrow[\text{Right}]{\text{Fold}}}}\hspace{0.5cm}
    \makecell{p = k a^2 k^{-1} \\ \text{(Eigendecomp.)}}
}\end{equation*}

\par For the $\kak$ decomposition, the same folding technique can be applied by generalizing the involution $\tau:X\mapsto X^{-T}$ on each side. Consider a $\kak$ decomposition $g = k_1ak_2$, let $\tau_1, \tau_2$ be two global (group level) involutions such that $\tau_1(k_1) = k_1$ and $\tau_2(k_2) = k_2$. We obtain the following two foldings.
\begin{equation}\label{eq:foldingrecipe}
    \arraycolsep=1.4pt
    \begin{array}{rclcrccclr}
         g & \cdot & \tau_2(g)^{-1} & = & k_1 & \cdot & a^2 & \cdot & \tau_2(k_1)^{-1} & \hspace{0.5cm}\text{(Right Folding)}\\
         \tau_1(g)^{-1} & \cdot & g & = & \tau_1(k_2)^{-1} & \cdot & a^2 & \cdot & k_2 & \hspace{0.5cm}\text{(Left Folding)}
    \end{array}
\end{equation}

A nice example is the matrix factorization $\fact{13}$ (Theorem \ref{thm:F13complex}) of an invertible square complex matrix $G$,
\begin{equation}\label{eq:orthosvdintro}
    G = U \Sigma V = \text{Unitary $\times$ Positive diagonal $\times$ Complex orthogonal.}
\end{equation}
Let $\Lambda = \Sigma^2$. The ``left folding" of the above factorization is, 
\begin{equation}
    G^HG = V^H\Sigma U^HU\Sigma V = V^H \Lambda V,
\end{equation}
which is the congruence diagonalization (by a \textit{complex orthogonal} matrix) of a Hermitian positive definite matrix. 

\par On the other hand, the ``right folding" of $\fact{13}$ is interesting in that we use the transpose and not the conjugate transpose:
\begin{equation*}
    GG^T = U\Sigma V V^T \Sigma U^T = U \Lambda U^T,
\end{equation*}
which is the factorization of a complex symmetric (not Hermitian) matrix $F=GG^T$ into $U\Lambda U^T$, where $U$ is a unitary matrix and $\Lambda$ is a positive diagonal matrix. This right folding is called the \textit{Takagi factorization} (Theorem \ref{thm:Takagicomp}). 
\begin{equation*}\boxed{
    \makecell{A = V^H \Lambda V\\ \text{(Cong. Diagonalization)}}
    \hspace{0.3cm}\xLeftarrow[\text{Left}]{\text{Fold}}\hspace{0.3cm}
    \makecell{G = U\Sigma V \\ (\kak,\text{ Thm \ref{thm:F13complex}})} 
    \hspace{0.3cm}\xRightarrow[\text{Right}]{\text{Fold}}\hspace{0.3cm}
    \makecell{F = U\Lambda U^T \\ \text{(Takagi)}}
}\end{equation*}

\par In the upcoming sections, we will discover many new matrix factorizations computed as the $\kak$ decomposition of the classical Lie groups. Each factorization naturally has the left and right foldings, as the $\kak$ decomposition always has two associated involutions (usually one Cartan and one non-Cartan). In this paper, we present many but not all of the folded factorizations. The foldings will be dealt in depth in future work. However the reader can always follow the recipe in \eqref{eq:foldingrecipe} to obtain a useful folded factorization.

\section{Matrix factorizations of orthogonal/unitary matrices}\label{sec:compact}
The classical compact Lie groups are the three $\beta$-unitary groups, $\ortho{n}$, $\un{n}$ and $\un{n, \mathbb{H}}$, denoted at once by $\Un_\beta(n)$ (often denoted by the symbol $\symp{n}$ or $\usp{2n}$). As Cartan classified in his work \cite{Cartan1,cartan1927certaines,Cartan2}, there are seven infinite families of (Riemannian) compact symmetric spaces $U/K$. Column 2 modulo column 3 ($U/K$) of Table \ref{tab:cartanss} are the seven Riemannian compact symmmetric spaces.

\par The factorizations are categorized into six categories. We denote them by $\fact{1}$ through $\fact{6}$. Note that $\fact{1}$ and $\fact{2}$ are KAK decompositions.\footnote{For the compact symmetric spaces BD$\RN{1}$, A$\RN{3}$, C$\RN{2}$ it is possible to have different $K_\sigma, K_\tau$ by changing the parameters of the involution. Thus, they are classified along the $\kak$ decomposition and will be discussed in Section \ref{sec:csdecomposition}.} Using the $\beta$-symbols (Table \ref{tab:translation}) discussed in Section \ref{sec:dictionary}, the generalized Cartan triple \eqref{def:gencartantriple}, $(G, K_\sigma, K_\tau)$ of each category is the following. 
\begin{alignat*}{2}
    \fact{1}&\text{ : }  \big(\Un_\beta(n), \Un_{\beta/2}(n), \Un_{\beta/2}(n)\big) \,\,(\text{KAK of }\text{A}\RN{1}, \text{C}\RN{1}) \hspace{3cm} &&\beta=2, 4 \\
    \fact{2}&\text{ : }  \big(\Un_\beta(2n), \Un_{2\beta}(n), \Un_{2\beta}(n)\big)\,\,(\text{KAK of }\text{D}\RN{3}, \text{A}\RN{2})  &&\beta=1,2 \\
    \fact{3}&\text{ : }  \big(\Un_\beta(2n), \Un_{\beta/2}(2n), \Un_{2\beta}(n)\big) &&\beta = 2\\
    \fact{4}&\text{ : } \big(\Un_\beta(n), \Un_\beta(p)\times\Un_\beta(q), \Un_\beta(r)\times\Un_\beta(s)\big) &&\beta = 1, 2, 4\\
    \fact{5}&\text{ : } \big(\Un_\beta(n), \Un_{\beta/2}(n), \Un_\beta(p)\times\Un_\beta(q)\big) &&\beta = 2, 4\\
    \fact{6}&\text{ : } \big(\Un_\beta(2n), \Un_{2\beta}(n), \Un_\beta(2p)\times\Un_\beta(2q)\big) &&\beta = 1, 2
\end{alignat*}

\par The theory and computations of the compact $\kak$ decomposition are studied thoroughly by Matsuki in \cite{matsuki2002classification}. This section is the reinterpretation of his work, in terms of matrix factorizations. The full list of matrix factorizations obtained in this section is listed in Table \ref{tab:compactlist}. For readers familiar with Cartan's classification (Table \ref{tab:cartanss}) we also have Cartan types for the $\kak$ decomposition (column 3) which is the combination of the two corresponding Cartan types of $G/K_\sigma$, $G/K_\tau$.

\begin{table}[h]
{\tabulinesep=0.7mm
\begin{tabu}{|c|c|c|c|c|}
\hline
$\mathcal{F}$ & Matrix factorization $g = k_\sigma a k_\tau$ & \makecell{Cartan\\type}  & Ref & Thm\\
\hline

\multirow{2}{*}{$\fact{1}$} & $U_n = O_n D O_n'$ & A$\RN{1}$-$\RN{1}$ & \cite{fuehr2018note,vilenkin2013representation} & \ref{thm:F1complex}\\
\cline{2-5}

& $\unmat{n}{H} = U_n D U_n'$ &  C$\RN{1}$-$\RN{1}$ & \cite{bloch1962canonical} & \ref{thm:F1quaternion}\\
\hline

\multirow{2}{*}{$\fact{2}$} & $O_{2n} = \realify{\unmat{n}{C}}\twotwo{R}{}{}{R^{-1}}\realify{\unmat{n}{C}}'$ & D$\RN{3}$-$\RN{3}$  & & \ref{thm:F2real}\\
\cline{2-5}

& $U_{2n} = \complexify{\unmat{n}{H}}\twotwo{D}{}{}{D}\complexify{\unmat{n}{H}}'$& A$\RN{2}$-$\RN{2}$  & \cite{edelman2020generalized} & \ref{thm:F2complex}\\
\hline

$\fact{3}$ & $U_{2n} = O_{2n}\twotwo{D}{}{}{D}\complexify{\unmat{n}{H}}$ & A$\RN{1}$-$\RN{2}$ &   & \ref{thm:F3complex} \\
\hline

\multirow{3}{*}{$\fact{4}$} & $O_n = \twotwo{O_p}{}{}{O_q}\sthreediag{C}{S}{I_{r-s}}{-S}{C}\twotwo{O_r}{}{}{O_s}$ & BD$\RN{1}$-$\RN{1}$  & \multirow{2}{*}{\cite{davis1969some,wigner1968generalization}}  & \ref{thm:F4real} \\
\cline{2-3}\cline{5-5}

& $U_n = \twotwo{U_p}{}{}{U_q}\sthreediag{C}{S}{I_{r-s}}{-S}{C}\twotwo{U_r}{}{}{U_s}$ &  A$\RN{3}$-$\RN{3}$ &  & \multirow{2}{*}{\ref{thm:F4complexquaternion}}\\
\cline{2-4}

& $\unmat{n}{H} = \twotwo{\unmat{p}{H}}{}{}{\unmat{q}{H}}\sthreediag{C}{S}{I_{r-s}}{-S}{C}\twotwo{\unmat{r}{H}}{}{}{\unmat{s}{H}}$ & C$\RN{2}$-$\RN{2}$  & & \\
\hline

\multirow{2}{*}{$\fact{5}$} & \makecell{$U_n = O_n\sthreediag{C}{iS}{I_{p-q}}{iS}{C}\twotwo{U_p}{}{}{U_q}$\\ 
$\&$ \\$V_{n, q}^\mathbb{C} = V_{n, 2q}^\mathbb{R}\twoone{C}{iS} U_q$} & A$\RN{1}$-$\RN{3}$  & & \ref{thm:F5complex}\\
\cline{2-5}

& \makecell{$\unmat{n}{H} = U_n \sthreediag{C}{jS}{I_{p-q}}{jS}{C}\twotwo{\unmat{p}{H}}{}{}{\unmat{q}{H}} $\\ $\&$ \\ $V_{n, q}^\mathbb{H} = V_{n, 2q}^\mathbb{C} \twoone{C}{jS} \unmat{q}{H}$}& C$\RN{1}$-$\RN{2}$ &  & \ref{thm:F5quaternion}\\
\hline

\multirow{2}{*}{$\fact{6}$} & $O_{2n} = \realify{\unmat{n}{C}} \begin{bsmallmatrix} I_{p-q} & & & \\ & C\otimes I_2 & & S\otimes J_1 \\ & & I_{p-q} & \\ & S\otimes J_1& & C\otimes I_2 \end{bsmallmatrix} \twotwo{O_{2p}}{}{}{O_{2q}}$ & D$\RN{1}$-$\RN{3}$ &  & \ref{thm:F6real}\\
\cline{2-5}

& $U_{2n} = \complexify{\unmat{n}{H}}\begin{bsmallmatrix} I_{p-q} & & & \\ & C\otimes I_2 & & S\otimes J_1 \\ & & I_{p-q} & \\ & S\otimes J_1& & C\otimes I_2 \end{bsmallmatrix}\twotwo{U_{2p}}{}{}{U_{2q}}$ & A$\RN{2}$-$\RN{3}$ &   & \ref{thm:F6complex} \\
\hline
\end{tabu}}
\vspace{0.1cm}
\caption{List of factorizations of unitary matrices. See the last column of Table \ref{tab:translation} and Table \ref{tab:auxmatrix} for matrix symbols. See Table \ref{tab:cartanss} for Cartan types. Again the matrix block structures are abbreviated to save space.}
\label{tab:compactlist}
\end{table}

\subsection{${\fact{1}}$ : ${(\Un_\beta(n), \Un_{\beta/2}(n), \Un_{\beta/2}(n))}$, the ODO decomposition}\label{sec:F1}
\begin{gather*}
\boxed{
\begin{bmatrix}
\rule{0cm}{0.1cm}\\ n\times n \\ \text{{ $\beta$-unitary }} \\ \rule{0cm}{0.1cm}
\end{bmatrix}
= 
\begin{bmatrix}
\rule{0cm}{0.1cm}\\ n\times n \\ \hspace{0.15cm} \text{$\frac{\beta}{2}$-unitary} \hspace{0.15cm}\, \\ \rule{0cm}{0.1cm}
\end{bmatrix}
\cdot
{\renewcommand{\arraystretch}{1.3}
\begin{bmatrix}
\LARGE{\text{$\diagdown$}} \rule{1.2cm}{0cm} \\ e^{i\theta_l} \\  \rule{1.1cm}{0cm}\LARGE{\text{$\diagdown$}}
\end{bmatrix}}
\cdot
\begin{bmatrix}
\rule{0cm}{0.1cm}\\ n\times n \\ \hspace{0.15cm} \text{$\frac{\beta}{2}$-unitary} \hspace{0.15cm}\, \\ \rule{0cm}{0.1cm}
\end{bmatrix}
}
\end{gather*}

We start out by reviewing some cases of the KAK decomposition. First we discuss the first matrix factorization $\fact{1}$. They are the KAK decompositions of the compact symmetric spaces A$\RN{1}$ and C$\RN{1}$. The compact symmetric space A$\RN{1}$, $G/K=\un{n}/\ortho{n}$ has the following KAK decomposition.
\begin{theorem}[ODO decomposition]\label{thm:F1complex}
For any $n\times n$ unitary matrix $U\in\un{n}$ there exist
\begin{itemize}
    \item Two $n\times n$ orthogonal matrices $O_1, O_2\in\ortho{n}$,
    \item An $n\times n$ unitary diagonal matrix $D = \diag(e^{i\theta_1}, \dots, e^{i\theta_n})$,
\end{itemize}
such that the following factorization holds:
\begin{equation}
    U = O_1DO_2.
\end{equation}
Here, $\theta_l\in[0, \pi)$ and $D$ is unique up to permutation of the diagonal elements.
\end{theorem}

We call this the \textit{ODO decomposition}. Since this is one of the simple KAK decompositions, it sometimes appears in Lie theory literatures (e.g., \cite{vilenkin2013representation}) as an example of the KAK decomposition. As a matrix factorization, the ODO decomposition has first appeared in \cite{fuehr2018note} and the authors also discussed it in \cite{edelman2020generalized}. The ODO decomposition says that for any unitary matrix $U$, we have the identical set of (left and right) singular vectors for the real and imaginary parts of $U$. In other words, 
\begin{equation}\label{eq:ODOseparation}
    \Re (U) = O_1CO_2, \hspace{0.8cm}\Im (U) = O_1SO_2,
\end{equation}
are the (real) SVD of the real and imaginary parts of a unitary matrix $U$. The folding of the ODO decomposition is the eigendecomposition of $U^TU$, which is often used to sample the circular orthogonal ensemble (COE). Moreover the eigenvectors can be chosen to be real. 

\par In fact this folding is the \textit{Takagi factorization} (Theorem \ref{thm:Takagicomp}) of a unitary symmetric matrix. The Takagi factorization of any unitary symmetric matrix $A$ is (almost) equivalent to the folding of the ODO decomposition, $A = ODO^T$. The Takagi factorization of $A$ is $A = (O\sqrt{D}) I (O\sqrt{D})^T$, with all Takagi values being 1.\footnote{Interestingly it has been shown in \cite{ikramov2012takagi} that the Takagi factorization of a unitary symmetric matrix can be computed in finitely many steps.} More details on the Takagi factorization will be discussed in Section \ref{sec:F13}.

\par The quaternion $\fact{1}$ is the KAK decomposition of the compact symmetric space C$\RN{1}$, $G/K=\un{n, \mathbb{H}}/\un{n}$.
\begin{theorem}[$\fact{1}$, quaternion]\label{thm:F1quaternion}
Fix $\eta\in\{j, k\}$. For any $n\times n$ quaternionic unitary matrix $Q\in\un{n, \mathbb{H}}$, there exist
\begin{itemize}
    \item Two $n\times n$ complex unitary matrices $U_1, U_2\in\un{n}$,
    \item An $n\times n$ unitary diagonal matrix $D = \diag(e^{\eta\theta_1}, \dots, e^{\eta\theta_n})$,
\end{itemize}
such that the following factorization holds:
\begin{equation}\label{eq:F1quaternion}
    Q = U_1 D U_2,
\end{equation}
Here, $\theta_l\in[0, \pi)$ and $D$ is unique up to permutation of the diagonal elements.
\end{theorem}

\par This is an extension of the ODO decomposition in the quaternionic sense. The complex representation of Theorem \ref{thm:F1quaternion} can be obtained by applying the complexify map $\complexify{\,\,\cdot\,\,}$ on both sides of \eqref{eq:F1quaternion}. An isomorphic form of the complexified \eqref{eq:F1quaternion} is introduced by Bloch and Messiah \cite{bloch1962canonical}, as a decomposition of the Bogoliubov transformation for fermions. See Remark \ref{rem:blochmessiah} for more details on the Bloch-Messiah decomposition and related works. 

\subsection{${\fact{2}}$ : ${(\Un_\beta(2n), \Un_{2\beta}(n), \Un_{2\beta}(n))}$, the QDQ decomposition}

\begin{gather*}
\boxed{
{\renewcommand{\arraystretch}{1.2}
\begin{bmatrix}
\rule{0cm}{0.1cm}\\ 2n\times 2n \\ \hspace{0.3cm}\text{$\beta$-unitary} \hspace{0.3cm}\,\\ \rule{0cm}{0.1cm}
\end{bmatrix}}
= 
\overbrace{\renewcommand{\arraystretch}{1.2}
\begin{bmatrix}
\rule{0cm}{0.1cm}\\ n\times n \\ \hspace{0.25cm}\text{$2\beta$-unitary} \hspace{0.25cm}\,\\ \rule{0cm}{0.1cm}
\end{bmatrix}}^{\text{realified/complexified}}
\cdot
{\renewcommand{\arraystretch}{0.6}
\begin{bmatrix}
    \diagdown \rule{1.7cm}{0cm} \\
     e^{i\theta_l}  \rule{1cm}{0cm}\\
     \diagdown \rule{0.3cm}{0cm}\\
     \hspace{0.4cm} \diagdown  \\
    \hspace{1.1cm} e^{i\theta_l} \\
    \hspace{1.8cm} \diagdown \\
\end{bmatrix}}
\cdot
\overbrace{\renewcommand{\arraystretch}{1.2}
\begin{bmatrix}
\rule{0cm}{0.1cm}\\ n\times n \\ \hspace{0.25cm}\text{$2\beta$-unitary} \hspace{0.25cm}\,\\ \rule{0cm}{0.1cm}
\end{bmatrix}}^{\text{realified/complexified}}
}
\end{gather*}

The real and complex $\fact{2}$ are the KAK decompositions of compact symmetric spaces D$\RN{3}$ and A$\RN{2}$. First, the KAK decomposition of the compact symmetric space D$\RN{3}$, $G/K = \ortho{2n}/\osp{2n}$ is the following factorization. 

\begin{theorem}[$\fact{2}$, real]\label{thm:F2real}
For any $2n\times 2n$ orthogonal matrix $O$, there exist 
\begin{itemize}
    \item Two realified $n\times n$ complex unitary matrices $U_1, U_2 \in\realify{\un{n}}$, 
    \item $m =\lfloor \frac{n}{2} \rfloor$ unique (up to order) angles $\theta_1, \dots, \theta_m\in[0, \frac{\pi}{2})$,
\end{itemize}
such that the following factorization holds:
\begin{equation}
    O = U_1 \twotwo{R}{}{}{R^{-1}} U_2,
\end{equation}
where $R=R_n^\theta$ is an $n\times n$ rotation block diagonal matrix defined in Table \ref{tab:auxmatrix}. 
\end{theorem}
The inverse matrix $R^{-1}$ is just a block diagonal matrix of $2\times 2$ rotation blocks $\begin{bsmallmatrix} \cos\theta_l & \mis\sin\theta_l \\ \sin\theta_l & \,\,\cos\theta_l\end{bsmallmatrix}$, where $R$ has inverse rotation blocks $\begin{bsmallmatrix} \,\cos\theta_l & \sin\theta_l \\ \mis\sin\theta_l & \cos\theta_l\end{bsmallmatrix}$ on its diagonal.

The compact symmetric space A$\RN{2}$, $\un{2n}/\usp{2n}$, has the following KAK decomposition and we call this the \textit{QDQ decomposition} \cite{edelman2020generalized}. 
\begin{theorem}[QDQ decomposition]\label{thm:F2complex}
For any $2n\times 2n$ complex unitary matrix $U$, there exist 
\begin{itemize}
    \item Two complexified $n\times n$ quaternionic unitary matrices $Q_1, Q_2\in\complexify{\un{n, \mathbb{H}}}$,
    \item An $n\times n$ unitary diagonal matrix $D = \diag(e^{i\theta_1}, \dots, e^{i\theta_n})$,
\end{itemize}
such that the following factorization holds:
\begin{equation}
    U = Q_1\twotwo{D}{}{}{D}Q_2.
\end{equation}
Here, $\theta_l\in[0, \frac{\pi}{2})$ and $D$ is unique up to permutation of the diagonal elements.
\end{theorem}

\par The QDQ decomposition of a $2n\times 2n$ unitary matrix sampled from the Haar measure of $\un{2n}$ can be used to obtain the circular symplectic ensemble (CSE).

\subsection{${\fact{3}}$ : ${(\Un_\beta(2n), \Un_{\beta/2}(2n), \Un_{2\beta}(n))}$}

\begin{gather*}
\boxed{
{\renewcommand{\arraystretch}{1.2}
\begin{bmatrix}
\rule{0cm}{0.1cm}\\ 2n\times 2n \\ \hspace{0.35cm}\text{ Unitary } \hspace{0.35cm}\,\\ \rule{0cm}{0.1cm}
\end{bmatrix}}
= 
{\renewcommand{\arraystretch}{1.2}
\begin{bmatrix}
\rule{0cm}{0.1cm}\\ 2n\times 2n \\ \hspace{0.2cm}\text{Orthogonal} \hspace{0.2cm}\,\\ \rule{0cm}{0.1cm}
\end{bmatrix}}
\cdot
{\renewcommand{\arraystretch}{0.6}
\begin{bmatrix}
    \diagdown \rule{1.7cm}{0cm} \\
    e^{i\theta_l}  \rule{1cm}{0cm}\\
     \diagdown \rule{0.3cm}{0cm}\\
     \hspace{0.4cm} \diagdown  \\
    \hspace{1.1cm} e^{i\theta_l} \\
    \hspace{1.8cm} \diagdown \\
\end{bmatrix}}
\cdot
\overbrace{\renewcommand{\arraystretch}{1.7}
\begin{bmatrix}
n\times n \\ \hspace{0.1cm}\text{Quaternionic} \hspace{0.1cm}\, \\ \text{Unitary}
\end{bmatrix}}^{\text{complexified}}
}
\end{gather*}

The factorization $\fact{3}$ is the first (non KAK) $\kak$ decomposition. In this generalized Cartan triple, we have two $\beta$-unitary groups with $\beta/2$ and $2\beta$ as the subgroups $K_\sigma$ and $K_\tau$, given $G = \Un_\beta(2n)$. Thus, we only have the choice of $\beta = 2$. Using Cartan's notation, this is compact A$\RN{1}$-$\RN{2}$. 

\begin{theorem}[$\fact{3}$, complex]\label{thm:F3complex}
For any $2n\times 2n$ complex unitary matrix $U\in\un{2n}$, there exist
\begin{itemize}
    \item A $2n\times 2n$ orthogonal matrix $O\in\ortho{2n}$, 
    \item A complexified $n\times n$ quaternionic unitary matrix $Q\in\complexify{\un{n}}$,
    \item An $n\times n$ unitary diagonal matrix $D = \diag(e^{i\theta_1}, \dots, e^{i\theta_n})$,
\end{itemize}
such that the following factorization holds:
\begin{equation}
    U = O \twotwo{D}{}{}{D} Q.
\end{equation}
Here, $\theta_l\in[0, \frac{\pi}{2})$ and $D$ is unique up to permutation of the diagonal elements.
\end{theorem}

Interestingly $\fact{3}$ can be used to sample the circular unitary ensemble (CUE), from a Haar measured $2n\times 2n$ matrix $U$ \cite{edelman2020generalized}. 

\subsection{${\fact{4}}$ : ${(\Un_\beta(n), \Un_\beta(p)\times\Un_\beta(q), \Un_\beta(r)\times\Un_\beta(s))}$, the CS decomposition}\label{sec:csdecomposition}

\begin{gather*}
\boxed{
{\renewcommand{\arraystretch}{1.2}
\begin{bmatrix}
\rule{0cm}{0.1cm}\\ n\times n \\ \hspace{0.3cm}\text{$\beta$-unitary} \hspace{0.3cm}\,\\ \rule{0cm}{0.1cm}
\end{bmatrix}}
= 
\begin{bmatrix}
\colorbox[HTML]{D6D6D6}{\makebox(26,26){\makecell{$p\times p$\\ \tiny{$\beta$-unitary}}}} & \\ 
& \hspace{-0.35cm}\colorbox[HTML]{D6D6D6}{\makebox(26,26){\makecell{$q\times q$\\ \tiny{$\beta$-unitary}}}}
\end{bmatrix}
\cdot
\begin{bmatrix}
\begin{array}{c:c}
\makebox(33, 38){$\begin{smallmatrix}\diagdown\hspace{1.1cm}\\ \cos\theta_l \hspace{0.7cm}\\ \diagdown \hspace{0.1cm} \\ \hspace{0.7cm}\rule{0cm}{0.4cm}\text{\large{$I_{r-s}$}}\end{smallmatrix}$}& 
\makebox(17, 38){$\begin{smallmatrix}\diagdown\hspace{0.6cm}\\ \sin\theta_l\\ \hspace{0.6cm}\diagdown  \\ \rule{0cm}{0.5cm}\end{smallmatrix}$}\\ \hdashline
\makebox(38,17){$\begin{smallmatrix}\diagdown\hspace{1.2cm}\\ -\sin\theta_l\hspace{0.7cm}\\ \diagdown\hspace{0.1cm}\end{smallmatrix}$}
& \makebox(17, 20){$\begin{smallmatrix}\diagdown\hspace{0.7cm}\\ \cos\theta_l\\ \hspace{0.5cm}\diagdown\end{smallmatrix}$}
\end{array}
\end{bmatrix}
\cdot
\begin{bmatrix}
\colorbox[HTML]{D6D6D6}{\makebox(35,35){\makecell{$r\times r$\\ \small{$\beta$-unitary}}}} & \\ 
& \hspace{-0.35cm}\colorbox[HTML]{D6D6D6}{\makebox(18,18){\makecell{\small{$s\times s$}\\\tiny{$\beta$-unit.}}}}
\end{bmatrix}
}
\end{gather*}

The matrix factorization $\fact{4}$ arising as a $\kak$ decomposition of $\Un_\beta(n)$ is a popular matrix factorization, the CS decomposition (CSD). Let $(p, q)$ and $(r, s)$ be two partitions of $n$, so $p, q, r, s$ are four integers such that $p+q = r+s = n$. Without losing generality, suppose $r\ge p\ge q\ge s$. We start with the real (orthogonal) CSD.

\begin{theorem}[CSD, real]\label{thm:F4real}
For any $n\times n$ orthogonal matrix $O\in\ortho{n}$, there exist four orthogonal matrices $O_r, O_p, O_q, O_s$ from $\ortho{r}$, $\ortho{p}$, $\ortho{q}$, $\ortho{s}$ and $s$ angles $\theta_1, \dots, \theta_s\in[0, \frac{\pi}{2})$ (unique up to order) such that the following holds:
\begin{equation}\label{eq:F4real}
O  = 
\begin{bmatrix}
\begin{array}{c:c}
    O_p &  \\ \hdashline
     & O_q
\end{array}
\end{bmatrix}
\begin{bmatrix}
\begin{array}{rc:c|c}
 C &  & & S \\
 & I_{p-s} & & \\ \hdashline
 & & I_{q-s} & \\ \hline
 -S & & & C
\end{array}
\end{bmatrix}
\begin{bmatrix}
\begin{array}{c|c}
    O_r &  \\ \hline
     & O_s
\end{array}
\end{bmatrix}.
\end{equation}
The matrices $C, S\in\mathbb{R}^{s\times s}$ are diagonal matrices with cosine and sine values of $\theta_1, \dots, \theta_s$ on their diagonals. These trigonometric values are called the CS values.
\end{theorem}

The real CSD is the compact $\kak$ decomposition of the type BD$\RN{1}$-$\RN{1}$. If $p=r$ and $q=s$, Theorem \ref{thm:F4real} becomes the KAK decomposition of the compact symmetric space BD$\RN{1}$. The following complex and quaternion CSD are compact $\kak$ decomposition of the types A$\RN{3}$-$\RN{3}$ and C$\RN{2}$-$\RN{2}$. 

\begin{theorem}[CSD, complex/quaternion]\label{thm:F4complexquaternion}
For any $n\times n$ complex (resp. quaternionic) unitary matrix $U$, there exist four complex (resp. quaternionic) unitary matrices $U_r, U_p, U_q, U_s$ with the sizes $r\times r$, $p\times p$, $q\times q$, $s\times s$, and $s$ angles $\theta_1, \dots, \theta_s\in[0, \frac{\pi}{2})$ (unique up to order) such that the following holds.
\begin{equation}\label{eq:F4complexquaternion}
U =  
\begin{bmatrix}
\begin{array}{c:c}
    U_p &  \\ \hdashline
     & U_q
\end{array}
\end{bmatrix}
\begin{bmatrix}
\begin{array}{rc:c|c}
 C &  & & S \\
 & I_{p-s} & & \\ \hdashline
 & & I_{q-s} & \\ \hline
 -S & & & C
\end{array}
\end{bmatrix}
\begin{bmatrix}
\begin{array}{c|c}
    U_r &  \\ \hline
     & U_s
\end{array}
\end{bmatrix}\end{equation}
The matrices $C, S\in\mathbb{R}^{s\times s}$ are defined as in Theorem \ref{thm:F4real}. 
\end{theorem}

In some ways the primitive form of the CS decomposition can be traced back to Jordan \cite{jordan1875essai}, and Hotelling \cite{hotelling1936relations}. The explicit form of the CSD with the restricted partition $p=r$, $q=s$ was first developed by Davis and Kahan in 1968 \cite{davis1970rotation} (which was published in 1970). It is in their later paper \cite{davis1969some} where they extended the CSD to cover the general partitions. At the same time, Eugene Wigner also published a result in 1968 equivalent to the CSD with $p=r$, $q=s$ which he viewed as a generalization of Euler's angles \cite[Eq.15]{wigner1968generalization}. (Remarkably, Wigner discovered the hyperbolic CS decomposition along the CSD in the same paper. See Section \ref{sec:F18} for details.) The CSD is covered in standard matrix computation texts \cite{golub2013matrix}, and it is used as a basic tool in the matrix perturbation theory and has numerous applications. Refer to \cite{paige1994history} for the history and applications of the CSD. 

\begin{remark}[the Stiefel manifold: tall skinny orthogonal/unitary matrices]
The \textit{Stiefel manifold} $V_p(\mathbb{F}^n)$ can be defined as a set of (ordered) $p$ orthonormal vectors in $\mathbb{F}^n$. A group theoretic definition of Stiefel manifold is (real case)
\begin{equation}
    V_p(\mathbb{R}^n) := \ortho{n}/\ortho{n-p}.
\end{equation}
The complex and quaternionic Stiefel manifolds are similarly defined with $\un{n}$ and $\un{n, \mathbb{H}}$, denoted by $V_p(\mathbb{C}^n)$ and $V_p(\mathbb{H}^n)$ respectively.

The Stiefel manifolds are useful when presenting the factorization of orthogonal and unitary matrices in a tall skinny (partial) format. In linear algebra, the Stiefel manifold is widely known as rectangular (or tall skinny) orthogonal and unitary matrices. A point in the Stiefel manifold $V_p(\mathbb{R}^n)$ is an $n\times p$ orthogonal matrix, as its columns are $p$ orthonormal vectors. A good example of a partial format would be the partial CS decomposition, which follows after this remark. 
\end{remark}

\par In the following paragraphs, we state the partial format of the CSD, which only involves a subset of columns in \eqref{eq:F4real} and \eqref{eq:F4complexquaternion}. In the real case it starts with a tall skinny orthogonal matrix, which can be regarded as an element in the Stiefel manifold, $V_k(\mathbb{R}^n)$. We introduce the last $s$ column version of Theorem \ref{thm:F4real}. We omit the complex and quaternion cases since they are trivially extended.

\begin{theorem}[tall skinny CSD, real]\label{thm:partialCSDreal}
For any $n\times s$ tall skinny orthogonal matrix $O\in V_s(\mathbb{R}^n)$, there exist orthogonal matrices $O_p$, $O_q$, $O_s$ and $s$ angles $\theta_1, \dots, \theta_s$ such that
\begin{equation}
    O = \twotwo{O_p}{}{}{O_q}\twoone{C'}{S'}O_s = \twoone{O_pCO_s}{O_qSO_s},
\end{equation}
where $C'\in\mathbb{R}^{p\times s}$, $S'\in\mathbb{R}^{q\times s}$ are (nonsquare) diagonal matrices with cosine and sine values of $\theta_1, \dots, \theta_s$ on their diagonals. 
\end{theorem}

\par An important algorithm related to the partial format of the CSD is the \textit{generalized singular value decomposition} (GSVD). The GSVD is an extension of the CSD to general matrices. One can imagine the GSVD equivalently as applying the \textit{QL decomposition}\footnote{Imagine the lower triangular analogue of the QR decomposition.} to any invertible matrix first and then applying CSD on the $Q$ part. The relationship between the GSVD and the CSD is well described in \cite{van1985computing}. We briefly state the real GSVD here.
\begin{theorem}[GSVD, real]
Let $A\in\mathbb{R}^{p\times s}$ and $B\in\mathbb{R}^{q\times s}$ be two matrices such that the matrix $[A; B]$ has rank $s$. Then we have the following decomposition,
\begin{equation}\label{eq:GSVD}
\twoone{A}{B} = \twotwo{O_p}{}{}{O_q}\twoone{C'}{S'}H
\end{equation}
where $H\in\mathbb{R}^{s\times s}$ is a full rank matrix. $C'\in\mathbb{R}^{p\times s}$, $S'\in\mathbb{R}^{q\times s}$ are (nonsquare) diagonal matrices with cosine and sine values of $\theta_1, \dots, \theta_s$ on their diagonals. 
\end{theorem}

\subsection{${\fact{5}}$ : ${(\Un_\beta(n), \Un_{\beta/2}(n), \Un_\beta(p)\times\Un_\beta(q))}$}

\begin{gather*}
\boxed{
{\renewcommand{\arraystretch}{1.2}
\begin{bmatrix}
\rule{0cm}{0.1cm}\\ n\times n \\ \hspace{0.3cm}\text{{$\beta$-unitary}} \hspace{0.3cm}\,\\ \rule{0cm}{0.1cm}
\end{bmatrix}}
= 
{\renewcommand{\arraystretch}{1.2}
\begin{bmatrix}
\rule{0cm}{0.1cm}\\ n\times n \\ \hspace{0.3cm}\text{{$\frac{\beta}{2}$-unitary}} \hspace{0.3cm}\,\\ \rule{0cm}{0.1cm}
\end{bmatrix}}
\cdot
\begin{bmatrix}
\begin{array}{c:c}
\makebox(33, 38){$\begin{smallmatrix}\diagdown\hspace{1.1cm}\\ \cos\theta_l \hspace{0.7cm}\\ \diagdown \hspace{0.1cm} \\ \hspace{0.7cm}\rule{0cm}{0.4cm}\text{\large{$I_{p-q}$}}\end{smallmatrix}$}& 
\makebox(17, 38){$\begin{smallmatrix}\diagdown\hspace{0.6cm}\\ i\sin\theta_l\\ \hspace{0.6cm}\diagdown  \\ \rule{0cm}{0.5cm}\end{smallmatrix}$}\\ \hdashline
\makebox(38,17){$\begin{smallmatrix}\diagdown\hspace{1.2cm}\\ i\sin\theta_l\hspace{0.7cm}\\ \diagdown\hspace{0.1cm}\end{smallmatrix}$}
& \makebox(17, 20){$\begin{smallmatrix}\diagdown\hspace{0.7cm}\\ \cos\theta_l\\ \hspace{0.5cm}\diagdown\end{smallmatrix}$}
\end{array}
\end{bmatrix}
\cdot
\begin{bmatrix}
\colorbox[HTML]{D6D6D6}{\makebox(35,35){\makecell{$p\times p$\\ \small{$\beta$-unitary}}}} & \\ 
& \hspace{-0.35cm}\colorbox[HTML]{D6D6D6}{\makebox(18,18){\makecell{\small{$q\times q$}\\\tiny{$\beta$-unit.}}}}
\end{bmatrix}
}
\end{gather*}

The factorization $\fact{5}$ is obtained by the $\kak$ decomposition with the triple $(G, K_\sigma, K_\tau) = (\Un_\beta(n), \Un_{\beta/2}(n), \Un_\beta(p)\times\Un_\beta(q))$, for $\beta = 2, 4$. With Cartan's classification, they are compact types A$\RN{1}$-$\RN{3}$ and C$\RN{1}$-$\RN{2}$. In this section we assume $n = p + q$ and $p\ge q$. The factorization $\fact{5}$ for $\beta = 2$ is the following.

\begin{theorem}[$\fact{5}$, complex]\label{thm:F5complex}
For any $n\times n$ unitary matrix $U\in\un{n}$, there exist
\begin{itemize}
    \item An $n\times n$ orthogonal matrix $O\in\ortho{n}$,
    \item Two unitary matrices $U_p\in\un{p}$ and $U_q\in\un{q}$,
    \item $q$ unique (up to order) angles $\theta_1, \dots, \theta_q\in[0,\frac{\pi}{2})$,
\end{itemize}
such that the following factorization of $U$ holds:
\begin{equation}\label{eq:F5complex}
    U = O 
    \begin{bmatrix}\begin{array}{cc:c}
        C & & iS \\
         & I_{p-q} & \\ \hdashline
         iS & & C
    \end{array}\end{bmatrix}
    \begin{bmatrix}
    \begin{array}{c:c}
    U_p & \\ \hdashline & U_q\end{array}
    \end{bmatrix}.
\end{equation}
where $C, S$ are cosine, sine matrices defined as in Table \ref{tab:auxmatrix}. 
\end{theorem}

Taking the last $q$ columns of \eqref{eq:F5complex} we obtain the following partial format factorization for tall skinny unitary matrix $U$ in the Stiefel manifold $V_q(\mathbb{C}^n)$. 
\begin{corollary}
Let $n\ge 2q$. For any $n\times q$ (tall skinny) unitary matrix $U$, there exist 
\begin{itemize}
    \item An $n\times 2q$ (tall skinny) orthogonal matrix $O$,
    \item A $q\times q$ unitary matrix $V\in\un{q}$,
    \item $q$ angles $\theta_1, \dots, \theta_q\in[0, \frac{\pi}{2})$,
\end{itemize}
such that the following factorization holds:
\begin{gather}
\nonumber{\overbrace{\hspace{0.2cm}}^{q}\hspace{0.7cm}\overbrace{\hspace{1cm}}^{2q}\hspace{0.4cm}\overbrace{\hspace{0.5cm}}^{q}}\hspace{0.7cm}
\\[-5pt]
\hspace{0.2cm}n\left\{
\begin{bmatrix}
\, \\ 
\,U\, \\
\,
\end{bmatrix}\right.
= \begin{bmatrix}
\hspace{0.8cm}\, \\
O \\
\,
\end{bmatrix}
\cdot
\begin{bmatrix}
C \\ iS
\end{bmatrix}
\left.\cdot \begin{bmatrix}
V
\end{bmatrix}\right\}q \ ,
\end{gather}
where $C, S\in\mathbb{R}^{q\times q}$ are the cosine/sine diagonal matrices.
\end{corollary}

For $q$ such that $2q\ge n$, one can pad $1$'s on the diagonal of $C$ (equivalently, selecting the first $p$ columns from \eqref{eq:F5complex}) to obtain a similar result. The quaternionic version of $\fact{5}$ is the following:

\begin{theorem}[$\fact{5}$, quaternion]\label{thm:F5quaternion}
Fix $\eta \in\{j, k\}$. For any $n\times n$ quaternionic unitary matrix $Q\in\un{n, \mathbb{H}}$, there exist
\begin{itemize}
    \item An $n\times n$ complex unitary matrix $U\in\un{n}$,
    \item Two quaternionic unitary matrices, $Q_p\in\un{p, \mathbb{H}}$ and $Q_q\in\un{q, \mathbb{H}}$,
    \item $q$ unique (up to order) angles $\theta_1, \dots, \theta_q\in[0, \frac{\pi}{2})$,
\end{itemize}
such that the following factorization holds: 
\begin{equation}
    Q = U 
    \begin{bmatrix}\begin{array}{cc:c}
        C & & \eta S \\
         & I_{p-q} & \\ \hdashline
         \eta S & & C
    \end{array}\end{bmatrix}
    \begin{bmatrix}
    \begin{array}{c:c}
    Q_p & \\ \hdashline & Q_q\end{array}
    \end{bmatrix}.
\end{equation}
\end{theorem}

Obviously, by the isomorphism between $i, j, k$, one can also select $\eta = i$, with $U$ being a unitary matrix with $j$ or $k$ as the complex unit. The partial format (Stiefel manifold version) of quaternion $\fact{5}$ is the following.
\begin{corollary}
For any $n\times q$ ($n\ge 2q$) tall skinny quaternionic unitary matrix $Q$ and fixed $\eta = j \text{ or } k$, there exist
\begin{itemize}
    \item An $n\times 2q$ tall skinny complex unitary matrix $U$, ($U^HU = I_{2q}$)
    \item A $q\times q$ quaternionic unitary matrix $V\in\un{q, \mathbb{H}}$, 
    \item $q$ angles $\theta_1, \dots, \theta_q\in[0, \frac{\pi}{2})$,
\end{itemize}
such that the following factorization holds: 
\begin{gather}
\nonumber{\overbrace{\hspace{0.5cm}}^{q}\hspace{0.7cm}\overbrace{\hspace{1cm}}^{2q}\hspace{0.4cm}\overbrace{\hspace{0.5cm}}^{q}}\hspace{0.7cm}
\\[-5pt]
\hspace{0.2cm}n\left\{
\begin{bmatrix}
\, \\ 
\,Q\, \\
\,
\end{bmatrix}\right.
= \begin{bmatrix}
\hspace{0.8cm}\, \\
U \\
\,
\end{bmatrix}
\cdot
\begin{bmatrix}
C \\ \eta S
\end{bmatrix}
\left.\cdot \begin{bmatrix}
V
\end{bmatrix}\right\}q \ .
\end{gather}
\end{corollary}

\subsection{${\fact{6}}$ : ${(\Un_\beta(2n), \Un_{2\beta}(n), \Un_\beta(2p)\times\Un_\beta(2q))}$}

\begin{gather*}
\boxed{
{\renewcommand{\arraystretch}{1.2}
\begin{bmatrix}
\rule{0cm}{0.1cm}\\ 2n\times 2n \\ \hspace{0.3cm}\text{$\beta$-unitary} \hspace{0.3cm}\,\\ \rule{0cm}{0.1cm}
\end{bmatrix}}
= 
\overbrace{\renewcommand{\arraystretch}{1.2}
\begin{bmatrix}
\rule{0cm}{0.1cm}\\ n\times n \\ \hspace{0.25cm}\text{$2\beta$-unitary} \hspace{0.25cm}\,\\ \rule{0cm}{0.1cm}
\end{bmatrix}}^{\text{realified/complexified}}
\cdot
\begin{bmatrix}
\makebox(30,30){$\begin{smallmatrix}\diagdown\hspace{0.8cm} \\ c_l \,\, 0 \\ 0\,\,c_l \\ \hspace{0.8cm}\diagdown\end{smallmatrix}$} 
& \hspace{-0.2cm}\makebox(30,30){$\begin{smallmatrix}\diagdown\hspace{0.8cm} \\ 0 \,\, s_l \\ -s_l \,\, 0 \hspace{0.3cm} \\ \hspace{0.8cm}\diagdown\end{smallmatrix}$} \\
\makebox(30,30){$\begin{smallmatrix}\diagdown\hspace{0.8cm} \\  0 \,\, s_l \\ -s_l \,\, 0 \hspace{0.3cm} \\ \hspace{0.8cm}\diagdown\end{smallmatrix}$} 
& \hspace{-0.2cm}\makebox(30,30){$\begin{smallmatrix}\diagdown\hspace{0.8cm} \\ c_l \,\, 0 \\ 0\,\,c_l \\ \hspace{0.8cm}\diagdown\end{smallmatrix}$}
\end{bmatrix}
\cdot
\begin{bmatrix}
\colorbox[HTML]{D6D6D6}{\makebox(35,35){\makecell{$2p\times 2p$\\\small{$\beta$-unitary}}}} & \\ 
& \hspace{-0.35cm}\colorbox[HTML]{D6D6D6}{\makebox(18,18){\makecell{\tiny{$2q\times 2q$}\\\tiny{$\beta$-unit.}}}}
\end{bmatrix}
}
\end{gather*}

The last factorization in this section involves realified complex (complexified quaternionic) unitary matrices. Recall two matrices
\begin{equation}
    I_2 = \twotwo{1}{0}{0}{1},\hspace{0.5cm}J_1= \twotworr{0}{1}{-1}{0},
\end{equation}
and let $n = p + q$, $p\ge q$. The factorization $\fact{6}$, $\beta=1$ is the following.
\begin{theorem}[$\fact{6}$, real]\label{thm:F6real}
For any $2n\times 2n$ orthogonal matrix $O\in\ortho{2n}$, there exist
\begin{itemize}
    \item An $n\times n$ unitary matrix $U\in\un{n}$,
    \item Two orthogonal matrices $O_{2p}\in\ortho{2p}, O_{2q}\in\ortho{2q}$,
    \item $q$ unique (up to order) angles $\theta_1, \dots, \theta_q\in[0, \frac{\pi}{2})$,
\end{itemize}
such that the following factorization holds:
\begin{equation}\label{eq:F6real}
    O = \begin{bmatrix}
    \begin{array}{r:c}
    \Re(U) & \Im(U)  \\ \hdashline
    -\Im(U) & \Re(U) 
    \end{array}
    \end{bmatrix}
    \begin{bmatrix}
    \begin{array}{cc:c|c}
        I_{p-q} & & & \\
        & C' & & S'\\ \hdashline
        & & I_{p-q} & \\ \hline
        & S' & & C'\\
    \end{array}
    \end{bmatrix}
    \begin{bmatrix}\begin{array}{c|c}
    O_{2p}     &  \\ \hline
         & O_{2q}
    \end{array}
    \end{bmatrix}
\end{equation}
where $C', S'\in\mathbb{R}^{2q\times 2q}$ are defined by the Kronecker product,
\begin{equation}\label{eq:F6realCS}
    C' := C\otimes I_2,\hspace{0.5cm} S' := S\otimes J_1.
\end{equation}
\end{theorem}

The first factor is indeed the realified $U$. Another choice of a maximal subalgebra $\mathfrak{a}$ can be made so that $C'=I_2\otimes C$ and $S'=J_1\otimes S$ (at the same positions of $C', S'$ in \eqref{eq:F6real}). Taking the last $2q$ columns of the left hand side of \eqref{eq:F6real}, we obtain the following partial format of Theorem \ref{thm:F6real}. 

\begin{corollary}
For any $2n\times 2q$ tall skinny orthogonal matrix $O$ there exist, ($n\ge 2q$)
\begin{itemize}
    \item An $n\times 2q$ complex tall skinny unitary matrix $U$ such that $U^HU = I_{2q}$,
    \item A $2q\times 2q$ orthogonal matrix $O_{2q}\in\ortho{2q}$,
    \item $q$ angles $\theta_1, \dots, \theta_q\in[0, \frac{\pi}{2})$,
\end{itemize}
such that the following factorization holds:
\begin{equation}
    O = \twotworr{\Re(U)}{\Im(U)}{-\Im(U)}{\Re(U)}\twoone{S'}{C'}O_{2q} = \realify{U}\twoone{S'}{C'} O_{2q},
\end{equation}
where the matrices $C', S'\in\mathbb{R}^{2q\times 2q}$ are defined as in \eqref{eq:F6realCS}.
\end{corollary}

The $\kak$ decomposition of compact type A$\RN{2}$-$\RN{3}$ is the complex version of real $\fact{6}$ on $2n\times 2n$ complex unitary matrices. 
\begin{theorem}[$\fact{6}$, complex]\label{thm:F6complex}
For any $2n\times 2n$ unitary matrix $U\in\un{2n}$, there exist
\begin{itemize}
    \item An $n\times n$ quaternionic unitary matrix $Q\in\un{n, \mathbb{H}}$,
    \item Two unitary matrices $U_{2p}\in\un{2p}, U_{2q}\in\un{2q}$,
    \item $q$ unique (up to order) angles $\theta_1, \dots, \theta_q\in[0, \frac{\pi}{2})$,
\end{itemize}
such that the following factorization holds:
\begin{equation}
    U = \begin{bmatrix}
    \renewcommand\arraystretch{1.3}
    \begin{array}{c:c}
    X & Y  \\ \hdashline
    -\overline{Y} & \overline{X} 
    \end{array}
    \end{bmatrix}
    \begin{bmatrix}
    \begin{array}{cc:c|c}
        I_{p-q} & & & \\
        & C' & & S'\\ \hdashline
        & & I_{p-q} & \\ \hline
        & S' & & C'\\
    \end{array}
    \end{bmatrix}
    \begin{bmatrix}\begin{array}{c|c}
    U_{2p}     &  \\ \hline
         & U_{2q}
    \end{array}
    \end{bmatrix},
\end{equation}
where $C', S'\in\mathbb{R}^{2q\times 2q}$ are defined as \eqref{eq:F6realCS} and $\complexify{Q} = \stwotwo{X}{Y}{-\overline{Y}}{\overline{X}}$ (complexified $Q$).
\end{theorem}

Finally the tall skinny version of complex $\fact{6}$ is as follows. 
\begin{corollary}
For any $2n\times 2q$ complex tall skinny unitary matrix $U$ there exist, ($n \ge 2q$)
\begin{itemize}
    \item A $n\times 2q$ tall skinny quaternionic unitary matrix $Q$ such that $Q^DQ = I_{2q}$,
    \item A $2q\times 2q$ complex unitary matrix $U_{2q}\in\un{2q}$,
    \item $q$ angles $\theta_1, \dots, \theta_q\in[0, \frac{\pi}{2})$,
\end{itemize}
such that the following factorization holds:
\begin{equation}
    U = \complexify{Q}\twoone{S'}{C'}U_{2q}
\end{equation}
with $C', S'\in\mathbb{R}^{2q\times 2q}$ defined as \eqref{eq:F6realCS}
\end{corollary}

\section{Matrix factorizations of invertible matrices}\label{sec:glfact}
In this section we discuss matrix factorizations of invertible matrices arising from the generalized Cartan decompositions of $\Gl_\beta(n)$. We obtain seven matrix factorizations: $\fact{7}$ to $\fact{13}$. The generalized Cartan triple \eqref{def:gencartantriple}, $(G, K_\sigma, K_\tau)$, of each $\fact{m}$ is as follows with $G=\Gl_\beta(n)$. Each category as listed may have a $\beta=1,2,4$ version or may only have two of the three values for $\beta$ providing a grand total of 17 possibilities.

\begin{alignat*}{2}
    \fact{7} &\text{ : }\big(\Gl_\beta(n), \Un_\beta(n), \Un_\beta(n)\big)\,\,(\text{KAK = SVD})\hspace{3.5cm} && \beta = 1, 2, 4 \\
    \fact{8} &\text{ : }\big(\Gl_\beta(n), \Un_{\beta}(n), \Gl_\beta(p)\times \Gl_\beta(q)\big) &&\beta = 1, 2, 4\\
    \fact{9} &\text{ : }\big(\Gl_\beta(n), \Un_{\beta}(n), \Un_\beta(p, q)\big) &&\beta = 1, 2, 4\\
    \fact{10} &\text{ : }\big(\Gl_\beta(2n), \Un_{\beta}(2n), \Symp_\beta(2n)\big) &&\beta = 1, 2\\
    \fact{11} &\text{ : }\big(\Gl_\beta(2n), \Un_{\beta}(2n), \Gl_{2\beta}(n)\big) &&\beta = 1, 2\\
    \fact{12} &\text{ : }\big(\Gl_\beta(n), \Un_{\beta}(n), \Gl_{\beta/2}(n)\big) &&\beta = 2, 4\\
    \fact{13} &\text{ : }\big(\Gl_\beta(n), \Un_{\beta}(n), \Ortho_{\beta}(n)\big) &&\beta = 2, 4
\end{alignat*}

\subsection{${\fact{7}}$ : ${(\Gl_\beta(n), \Un_{\beta}(n), \Un_{\beta}(n))}$, the SVD}

\begin{gather*}
\boxed{
\begin{bmatrix}
\rule{0cm}{0.1cm}\\ n\times n \\ \text{{ Invertible }} \\ \rule{0cm}{0.1cm}
\end{bmatrix}
= 
\begin{bmatrix}
\rule{0cm}{0.1cm}\\ n\times n \\ \hspace{0.1cm} \text{{ Unitary }} \hspace{0.1cm}\, \\ \rule{0cm}{0.1cm}
\end{bmatrix}
\cdot
{\renewcommand{\arraystretch}{1.3}
\begin{bmatrix}
\LARGE{\text{$\diagdown$}} \rule{1.2cm}{0cm} \\ \sigma_l \\  \rule{1.1cm}{0cm}\LARGE{\text{$\diagdown$}}
\end{bmatrix}}
\cdot
\begin{bmatrix}
\rule{0cm}{0.1cm}\\ n\times n \\ \hspace{0.1cm} \text{{ Unitary }} \hspace{0.1cm}\, \\ \rule{0cm}{0.1cm}
\end{bmatrix}
}
\end{gather*}

We first discuss the KAK decomposition of the general linear groups $\gl{n, \mathbb{R}}$, $\gl{n, \mathbb{C}}$ and $\gl{n, \mathbb{H}}$. All three $\beta = 1, 2, 4$ cases have the same double coset structure, the set of all positive diagonal matrices. These KAK decompositions are the famous \textit{singular value decomposition} (SVD) of square invertible matrices.\footnote{This is the most common example of the KAK decomposition used in the Lie theory courses and textbooks.} Note that the rectangular SVDs are obtained from the folding of the KAK decompositions of the indefinite unitary groups $G=\Un_\beta(p, q)$. See Section \ref{sec:F18} or \cite{edelman2020generalized,kleinsteuber2005jacobi} for details.

\begin{theorem}\label{thm:SVD}
The KAK decomposition of the (Riemannian) noncompact symmetric spaces $\gl{n, \mathbb{R}}/\ortho{n}$, $\gl{n, \mathbb{C}}/\un{n}$ and $\gl{n, \mathbb{H}}/\un{n, \mathbb{H}}$ are the real, complex and quaternionic SVD, respectively. 
\end{theorem}

\begin{remark}\label{rem:glsvdnaming}
An important relationship which will be used throughout the section is the equivalence of the (left) folding of the SVD to the eigenvalue decomposition. As we discuss in Section \ref{sec:folding}, the following well-known left folding of the SVD becomes the symmetric (positive definite) eigenvalue decomposition. 
\begin{equation}\label{eq:svdfolding}
    \makecell{A^TA = V\Sigma^2V^T\\(\text{Eigendecomposition})}
    \hspace{0.3cm}{\Large{\xLeftarrow[\text{Left}]{\text{Fold}}}}\hspace{0.3cm} 
    \makecell{A = U\Sigma V^T\\ (\text{SVD})} \ ,
\end{equation}
The $\kak$ decomposition of $\Gl_\beta(n)$ always has the first factor $U\in\Un_\beta(n)=K_\sigma$, which obtains the left folding as follows. ($X=A^\dagger A$)
\begin{equation}\label{eq:glsvdleftfoldingrelation}
    \hspace{1.4cm}\makecell{X = V^\dagger \Lambda^2 V\\(\text{Eigen-like})}
    \hspace{0.5cm}{\Large{\xLeftarrow[\text{Left}]{\text{Fold}}}}\hspace{0.5cm} 
    \makecell{A = U \Lambda V \\ (\kak)}.
\end{equation}

\par The matrix $\Lambda^2$ in \eqref{eq:glsvdleftfoldingrelation} contains the eigenvalue-like information of the congruence diagonalization. In fact, the values of $\Lambda$ are sometimes named as a variant of the eigenvalues (which will appear in Section \ref{sec:F9}, \ref{sec:F10}), although the values themselves are nothing related to the usual eigenvalues. Being consistent with the names in \eqref{eq:svdfolding} and \eqref{eq:glsvdleftfoldingrelation}, we call some factorizations of $\Gl_\beta(n)$ as a variant of the SVD. (Equivalently, one can put the factor $\Un_\beta(n)$ on the right side of the $\kak$ decomposition and perform the right folding.)
\end{remark}

\subsection{${\fact{8}}$ : ${(\Gl_\beta(n), \Un_{\beta}(n), \Gl_\beta(p)\times \Gl_\beta(q))}$}\label{sec:F8}

\begin{gather*}
\boxed{
{\renewcommand{\arraystretch}{1.35}
\begin{bmatrix}
\rule{0cm}{0.1cm}\\ n\times n \\ \hspace{0.4cm} \text{{Invertible}} \hspace{0.4cm}\, \\ \rule{0cm}{0.1cm}
\end{bmatrix}}
= 
{\renewcommand{\arraystretch}{1.35}
\begin{bmatrix}
\rule{0cm}{0.1cm}\\ n\times n \\ \hspace{0.4cm} \text{{ Unitary }} \hspace{0.4cm}\, \\ \rule{0cm}{0.1cm}
\end{bmatrix}}
\cdot
\begin{bmatrix}
\begin{array}{c:c}
\makebox(33, 38){$\begin{smallmatrix}\diagdown\hspace{1.1cm}\\ \text{ch}_l \hspace{0.7cm}\\ \diagdown \hspace{0.1cm} \\ \hspace{0.7cm}\rule{0cm}{0.4cm}\text{\large{$I_{p-q}$}}\end{smallmatrix}$}& 
\makebox(17, 38){$\begin{smallmatrix}\diagdown\hspace{0.6cm}\\ \text{sh}_l\\ \hspace{0.6cm}\diagdown  \\ \rule{0cm}{0.5cm}\end{smallmatrix}$}\\ \hdashline
\makebox(38,17){$\begin{smallmatrix}\diagdown\hspace{1.2cm}\\ \text{sh}_l\hspace{0.7cm}\\ \diagdown\hspace{0.1cm}\end{smallmatrix}$}
& \makebox(17, 20){$\begin{smallmatrix}\diagdown\hspace{0.7cm}\\ \text{ch}_l\\ \hspace{0.5cm}\diagdown\end{smallmatrix}$}
\end{array}
\end{bmatrix}
\cdot
\begin{bmatrix}
\colorbox[HTML]{D6D6D6}{\makebox(35,35){\makecell{$p\times p$\\\small{Invertible}}}} & \\ 
& \hspace{-0.35cm}\colorbox[HTML]{D6D6D6}{\makebox(18,18){\makecell{\small{$q\times q$}\\\tiny{Invert.}}}}
\end{bmatrix}
}
\end{gather*}

The factorization $\fact{8}$ is a matrix factorization of real, complex and quaternionic invertible matrices. Let $n = p+q$, $p\ge q$.
\begin{theorem}[$\fact{8}$, real]\label{thm:F8real}
For any real $n\times n$ invertible matrix $G$, there exist
\begin{itemize}
    \item An $n\times n$ orthogonal matrix $O\in\ortho{n}$,
    \item Two real invertible matrices $G_p\in\gl{p, \mathbb{R}}$ and $G_q\in\gl{q,\mathbb{R}}$,
    \item $q$ unique (up to order and signs) real numbers $\theta_1, \dots, \theta_q$,
\end{itemize}
such that the following factorization of $G$ holds:
\begin{equation}\label{eq:F8real}
    G = O \begin{bmatrix}\begin{array}{cc:c}
        Ch & & Sh \\
         & I_{p-q} & \\ \hdashline
         Sh & & Ch
    \end{array}\end{bmatrix}
    \begin{bmatrix}
    \begin{array}{c:c}
    G_p & \\ \hdashline & G_q\end{array}
    \end{bmatrix},
\end{equation}
where $Ch, Sh\in\mathbb{R}^{q\times q}$ are diagonal matrices with $\cosh, \sinh$ values of $\theta_l$'s.
\end{theorem}
The $\beta=2$ (resp. $\beta=4$) $\fact{8}$ factorization is obtained from the generalized Cartan decomposition of $\gl{n, \mathbb{C}}$ (resp. $\gl{n, \mathbb{H}}$). 

\begin{theorem}[$\fact{8}$, complex/quaternion]\label{thm:F8complexquaternion}
For any complex (resp. quaternionic) $n\times n$ invertible matrix $G$, there exist
\begin{itemize}
    \item An $n\times n$ complex (resp. quaternionic) unitary matrix $U$,
    \item Two complex (resp. quaternionic) invertible matrices $G_p\in\gl{p, \mathbb{C}}$ and $G_q\in\gl{q, \mathbb{C}}$, (resp., from $\gl{p, \mathbb{H}}$ and $\gl{q, \mathbb{H}}$)
    \item $q$ unique (up to order and signs) real numbers $\theta_1, \dots, \theta_q$,
\end{itemize}
such that the following factorization of $G$ holds:
\begin{equation}\label{eq:F8complexquaternion}
    G = U\begin{bmatrix}\begin{array}{cc:c}
        Ch & & Sh \\
         & I_{p-q} & \\ \hdashline
         Sh & & Ch
    \end{array}\end{bmatrix}
    \begin{bmatrix}
    \begin{array}{c:c}
    G_p & \\ \hdashline & G_q\end{array}
    \end{bmatrix}.
\end{equation}
\end{theorem}

\begin{remark}\label{rem:F8triangularformat}
The matrix with hyperbolic cosines and sines in $\fact{8}$ factorizations can be replaced by a simpler matrix. Applying appropriate Givens rotations (on the left) and column scaling by a diagonal matrix (on the right), we obtain a simple modification of \eqref{eq:F8real} (also a similarly form for \eqref{eq:F8complexquaternion}). 
\begin{equation}
    G = O'{\renewcommand{\arraystretch}{1.2}
    \begin{bmatrix}\begin{array}{cc:c}
        I_q & & \Sigma \\
         & I_{p-q} & \\ \hdashline
          & & I_q
    \end{array}\end{bmatrix}
    \begin{bmatrix}
    \begin{array}{c:c}
    G_p' & \\ \hdashline & G_q'\end{array}
    \end{bmatrix}},
\end{equation}
where $\Sigma$ is a $q\times q$ diagonal matrix placed at the top right corner. We use the prime notation to emphasize that the matrices $G_p', G_q', O'$ are modified from \eqref{eq:F8real}. 
\end{remark}

\par It is natural to consider the left folding of $\fact{5}$. By the Cholesky factorization we know that the set of all matrices $G^TG$, $G\in\gl{n, \mathbb{R}}$ is equal to the set of all symmetric positive definite matrices. Similarly we have that $\{G^HG: G\in\gl{n, \mathbb{C}}\}$ is the set of all Hermitian positive definite matrices. (Also holds for quaternionic cases which we will omit here.)

\begin{corollary}
Any $n\times n$ symmetric (resp. complex Hermitian) positive definite matrix $A$ can be tridiagonalized by a congruence transformation. In other words, there exist $G_p\in\gl{p, \mathbb{R}}$, $G_q\in\gl{q, \mathbb{R}}$ (resp. $G_p\in\gl{p, \mathbb{C}}$, $G_q\in\gl{q, \mathbb{C}}$) and $q$ real numbers $\theta_1, \dots, \theta_q$ such that the following holds:
\begin{equation}
    A = \begin{bmatrix}
    \begin{array}{c:c}
        G_p &  \\ \hdashline
         & G_q
    \end{array}
    \end{bmatrix}
    \begin{bmatrix}
    \begin{array}{cc:c}
        Ch & & Sh \\
         & I_{p-q} & \\ \hdashline
         Sh & & Ch
    \end{array}
    \end{bmatrix}
    \begin{bmatrix}
    \begin{array}{c:c}
        G_p &  \\ \hdashline
         & G_q
    \end{array}
    \end{bmatrix}^H.
\end{equation}
\end{corollary}

\subsection{${\fact{9}}$ : ${(\Gl_\beta(n), \Un_{\beta}(n), \Un_\beta(p, q))}$, the hyperbolic SVD}\label{sec:F9}

\begin{gather*}
\boxed{
\begin{bmatrix}
\rule{0cm}{0.1cm}\\ n\times n \\ \text{{ Invertible }} \\ \rule{0cm}{0.1cm}
\end{bmatrix}
= 
\begin{bmatrix}
\rule{0cm}{0.1cm}\\ n\times n \\ \hspace{0.1cm} \text{ Unitary } \hspace{0.1cm}\, \\ \rule{0cm}{0.1cm}
\end{bmatrix}
\cdot
{\renewcommand{\arraystretch}{1.3}
\begin{bmatrix}
\LARGE{\text{$\diagdown$}} \rule{1.2cm}{0cm} \\ \sigma_l \\  \rule{1.1cm}{0cm}\LARGE{\text{$\diagdown$}}
\end{bmatrix}}
\cdot
{\renewcommand{\arraystretch}{1.3}
\begin{bmatrix}
n\times n \\ \hspace{0.1cm} \text{ Indefinite } \hspace{0.1cm}\, \\ \text{unitary}
\end{bmatrix}}
}
\end{gather*}

Suppose $n = p+q$, $p\ge q$ and let $A$ be the group of $n\times n$ positive diagonal matrices. The factorization $\fact{9}$ of real, complex and quaternionic invertible matrices arises from the following three generalized Cartan decompositions:
\begin{align*}
    \gl{n, \mathbb{R}} &= \ortho{n} \cdot A \cdot\ortho{p, q},\\
    \gl{n, \mathbb{C}} &= \un{n} \cdot A \cdot\un{p, q},\\
    \gl{n, \mathbb{H}} &= \un{n, \mathbb{H}}\cdot A \cdot \un{p, q, \mathbb{H}}.
\end{align*}
Using $\beta$ symbols we can denote these decompositions at once by 
\begin{equation}
    \Gl_\beta(n) = \Un_\beta(n) \cdot A \cdot \Un_\beta(p, q).
\end{equation}
We note that the real and complex cases appeared as a matrix factorization in 1989 \cite{onn1989hyperbolic} and was named the \textit{hyperbolic SVD} (HSVD). 

\begin{theorem}[HSVD, real]\label{thm:F9real}
For any $n\times n$ real invertible matrix $G$, there exist
\begin{itemize}
    \item An $n\times n$ orthogonal matrix $O\in\ortho{n}$,
    \item An indefinite orthogonal matrix $V\in\ortho{p, q}$,
    \item A unique (up to diagonal permutation) $n\times n$ positive diagonal matrix $\Sigma$ , 
\end{itemize}
such that the following factorization holds:
\begin{equation}
    G = O\Sigma V.
\end{equation}
The diagonal values $\sigma_1, \dots, \sigma_n$ of $\Sigma$ are called the hyperbolic singular values.
\end{theorem}

\par Recall the discussion regarding the left folding of the SVD in Remark \ref{rem:glsvdnaming}. The hyperbolic SVD is the hyperbolic analogue of the SVD in a sense that its left folding is the \textit{hyperbolic eigenproblem}. The hyperbolic eigenproblem for a symmetric positive definite matrix $K$ is defined as finding a pair $(\lambda, x)$ satisfying ($x\neq 0$)
\begin{equation}\label{eq:hyperboliceigproblem}
    Kx = \lambda I_{p, q} x,
\end{equation}
where $\lambda$ is called the \textit{hyperbolic eigenvalue} and the vector $x$ is called the \textit{hyperbolic eigenvector}. It can be shown that the matrix $X$, such that $X^TKX$ is a diagonal matrix, is exactly the matrix consisting of hyperbolic eigenvectors. In addition, the diagonal matrix $\Lambda = X^TKX$ is the diagonal matrix with hyperbolic eigenvalues. The left folding of the HSVD $A = O\Sigma V$ is
\begin{equation}
    A^TA = V^T\Sigma^2 V,
\end{equation}
which is equivalent to the hyperbolic eigenproblem above with $K=A^TA$, $V = X^{-1}$, $\Lambda = \Sigma^2$. Thus, we justify the name ``HSVD" for Theorem \ref{thm:F9real}. 
\begin{corollary}
Fix $n = p+q$. Any $n\times n$ real symmetric positive definite matrix $S$ can be diagonalized by a congruence transformation of some indefinite orthogonal matrix $V\in\ortho{p, q}$. 
\end{corollary}

\par On the other hand, the right folding of the HSVD is the following important factorization. For any invertible matrix $A$, the matrix 
\begin{equation}\label{eq:indefinitesymeigproblem}
    A^TI_{p, q}A,
\end{equation}
is a symmetric matrix containing $p$ positive and $q$ negative eigenvalues. This eigenproblem is the \textit{indefinite symmetric eigenproblem}. The right folding of the HSVD is the eigendecomposition of \eqref{eq:indefinitesymeigproblem}. 

\par Combining the positive definite, negative definite and indefinite cases, we obtain a set of tools for analyzing the symmetric eigenproblems. The Cholesky factorization can be applied for any symmetric matrix $G$ with a small modification, (the LDL decomposition with $D=I_{p, q}$)
\begin{equation}
    G = LI_{p, q}L^T,
\end{equation}
where $p, q$ are numbers of positive and negative eigenvalues of $G$. The pair $(p, q)$ is often called the \textit{signature} or \textit{Sylvester's inertia}. Taking the HSVD of $L$ (if $q=0$ it is the usual SVD) the eigendecomposition of $G$ is obtained. 

\begin{corollary}
Any $n\times n$ invertible real symmetric matrix $S$ can be decomposed as the following:
\begin{equation}
    S = O\Lambda O^T,
\end{equation}
where $O\in\ortho{n}$ and $\Lambda$ is a real diagonal matrix.
\end{corollary}

\begin{figure}[h]
    \title{\Large  Geometry of the Hyperbolic SVD: $G=O\Sigma V$\vspace{0.1cm}} 
    \centering
    \frame{\includegraphics[width=4.9in]{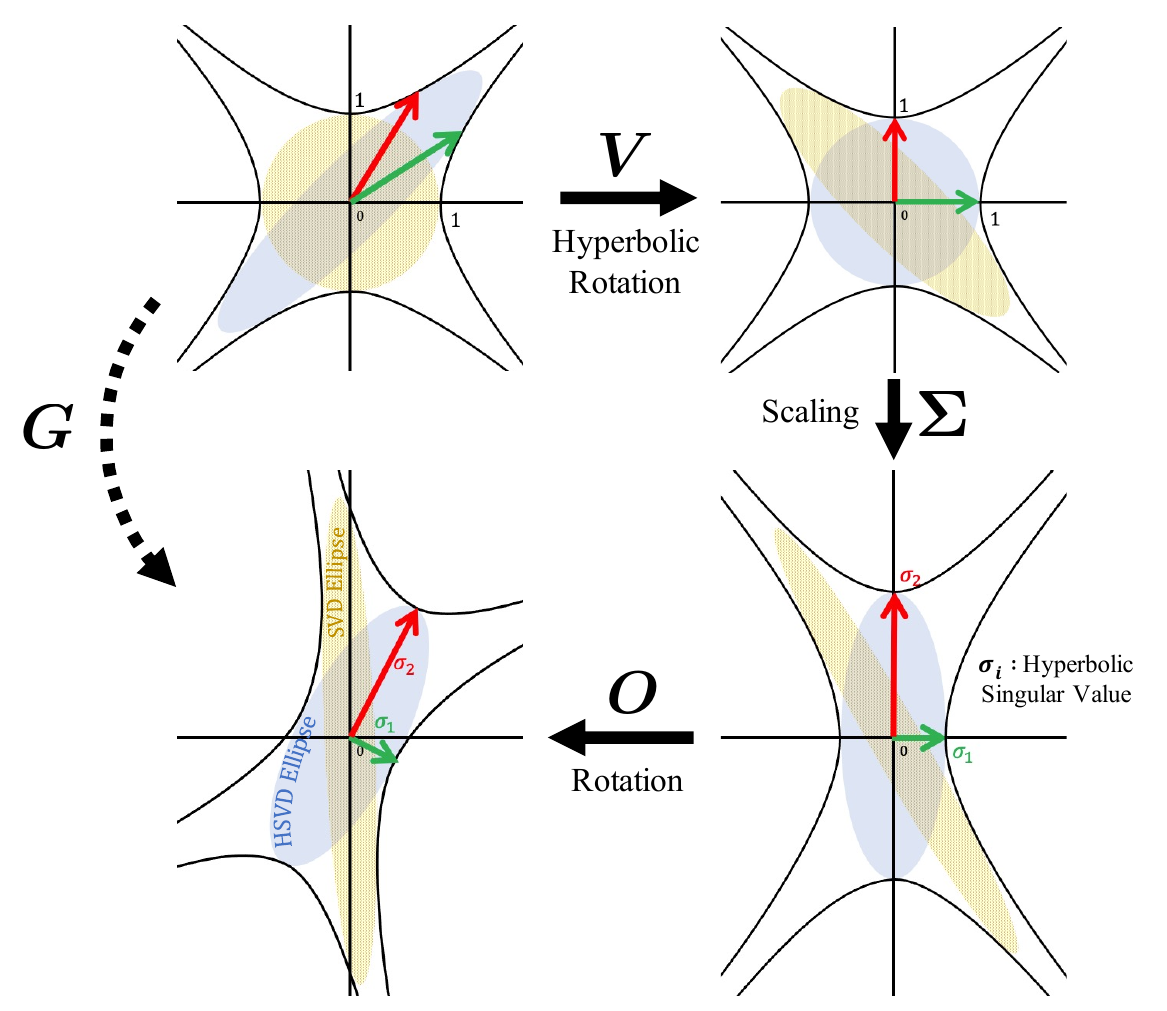}}
    \caption{The yellow shapes are the standard picture of the ordinary SVD ellipse, while the blue shape is the hyperbolic SVD ellipse: 1) Hyperbolic Rotation ($V$): preserves hyperbolas, yellow unit circle $\rightarrow$ yellow ellipse, blue ellipse $\rightarrow$ blue circle, red and green hyperbolic orthogonal vectors $\rightarrow$ another set of hyperbolic orthogonal vectors. 2) Scaling ($\Sigma$) reveals the hyperbolic singular values 3) Orthogonal ($O$) rotates (or reflects).}
    \label{fig:hsvd}
\end{figure}

\par Figure \ref{fig:hsvd} illustrates the geometry of the HSVD as a linear mapping on a real matrix $G$ in analogy to the famous three step illustration of the SVD. (Along with the SVD of $G$ which is illustrated by the yellow ellipses.) In Figure \ref{fig:hsvd}, $G=O\Sigma V$ is the HSVD of $G$. The red and green vectors in the top left represent two hyperbolic orthogonal vectors. ($v_1^T I_{p, q}v_2 = 0$.) Two vectors evolve as the matrices $V$, $\Sigma$ and $O$ consecutively applied as linear operators. The blue ellipse in the top left is the ellipse circumscribing two vectors, equivalently the inverse image of the unit disk (top right) under $V$. The bottom left picture encodes both singular values and hyperbolic singular values, as the length of axes of blue and yellow ellipse, respectively.

\par Complex and quaternionic HSVDs and their foldings are the following.

\begin{theorem}[HSVD, complex/quaternion]\label{thm:F9comp}
For any $n\times n$ complex (resp. quaternionic) invertible matrix $G$, there exist 
\begin{itemize}
    \item An $n\times n$ complex (resp. quaternionic) unitary matrix $U\in\un{n}$ (resp. $U\in\un{n, \mathbb{H}}$),
    \item An indefinite complex (resp. quaternionic) unitary matrix $V\in\un{p, q}$ (resp $V\in\un{p, q,\mathbb{H}}$),
    \item A unique (up to diagonal permutation) $n\times n$ positive diagonal matrix $\Sigma$,
\end{itemize}
such that the factorization $G = U\Sigma V$ holds.
\end{theorem}

\begin{corollary}
Any $n\times n$ complex (resp. quaternionic) Hermitian matrix $A$ is $\dagger$-congruent to a real positive diagonal matrix $D$ by a complex (resp. quaternionic) indefinite unitary matrix. In other words, there exist $V\in\un{p, q}$ (resp. $V\in\un{p, q, \mathbb{H}}$) such that
\begin{equation}
    A = VDV^\dagger.
\end{equation}
\end{corollary}

\subsection{${\fact{10}}$ : ${(\Gl_\beta(2n), \Un_{\beta}(2n), \Symp_\beta(2n))}$, the symplectic SVD}\label{sec:F10}
\begin{gather*}
\boxed{
\begin{bmatrix}
\rule{0cm}{0.1cm}\\ 2n\times 2n \\ \text{{ Invertible }} \\ \rule{0cm}{0.1cm}
\end{bmatrix}
= 
\begin{bmatrix}
\rule{0cm}{0.1cm}\\ 2n\times 2n \\ \hspace{0.15cm} \text{ Unitary } \hspace{0.15cm}\, \\ \rule{0cm}{0.1cm}
\end{bmatrix}
\cdot
{\renewcommand{\arraystretch}{0.65}
\begin{bmatrix}
    \diagdown \rule{1.5cm}{0cm} \\
    \sigma_l  \rule{0.9cm}{0cm}\\
     \diagdown \rule{0.3cm}{0cm}\\
     \hspace{0.4cm} \diagdown  \\
    \hspace{1cm} \sigma_l \\
    \hspace{1.6cm} \diagdown \\
\end{bmatrix}}
\cdot
\begin{bmatrix}
\rule{0cm}{0.1cm}\\ 2n\times 2n \\ \text{{ Symplectic }} \\ \rule{0cm}{0.1cm}
\end{bmatrix}
}
\end{gather*}

The factorization $\fact{10}$ of real and complex invertible matrices involves symplectic matrices from $\Symp_\beta(2n)$. Recall that a $2n\times 2n$ real or complex symplectic matrix $A$ satisfies the property $A^TJ_nA = J_n$ and the transpose $A^T$ is always the (non-conjugate) regular transpose regardless of $\mathbb{F} = \mathbb{R}, \mathbb{C}$. Refer to Section \ref{sec:symplectic} for more details on the symplectic group.

\par The following factorizations are similar to the HSVD with a replacement of the indefinite orthogonal/unitary matrix by a symplectic matrix. Again, recall the discussion about the nomenclature of these factorizations in Remark \ref{rem:glsvdnaming}. In fact, the left folding of these factorizations are related to the so-called \textit{symplectic eigenvalues}. Correspondingly, we call this the \textit{symplectic SVD}. 

\begin{theorem}[symplectic SVD, real]\label{thm:F10real}
For any $2n\times 2n$ real invertible matrix $G$, there exist 
\begin{itemize}
    \item A $2n\times 2n$ orthogonal matrix $O\in\ortho{2n}$, 
    \item A $2n\times 2n$ real symplectic matrix $S\in\symp{2n, \mathbb{R}}$,
    \item A unique $n\times n$ positive diagonal matrix $\Sigma$ (up to diagonal permutation),
\end{itemize}
such that the following factorization holds:
\begin{equation}
    G = O \twotwo{\Sigma}{}{}{\Sigma} S.
\end{equation}
\end{theorem}

The symplectic SVD diagonalizes an invertible matrix by the combination of the (left) orthogonal transformation and the (right) symplectic transformation (or the other way around). This is first introduced by Xu in 2003 \cite{xu2003svd} and Xu calls this factorization the \textit{SVD-like decomposition}. (See \cite{xu2005numerical} for an algorithm.) However as we argue in Remark \ref{rem:glsvdnaming}, we feel that the factorizations of $\gl{n, \mathbb{F}}$ in this section are more appropriate to be called a variant of the SVD. 

\par The following left folding of Theorem \ref{thm:F10real} is another eigen-like problem. 
\begin{corollary}\label{thm:symplecticcong}
Any real $2n\times 2n$ symmetric positive definite matrix $A$ is congruent to a positive diagonal matrix $\Lambda$ of the form $\begin{bsmallmatrix}\Sigma & \\ & \Sigma\end{bsmallmatrix}$ by a real symplectic matrix $S\in\symp{2n, \mathbb{R}}$. In other words, we have
\begin{equation}
    SAS^T = \twotwo{\Sigma}{}{}{\Sigma}.
\end{equation}
\end{corollary}
Corollary \ref{thm:symplecticcong} is called \textit{Williamson's theorem}. It is a special case of Corollary of Theorem 2 in Williamson's original paper in 1936 \cite{williamson1936algebraic}, which is later named and summarized in 1978, in Appendix 6 of \cite{arnol2013mathematical}. It has been studied in several numerical linear algebra papers with applications on the Hamiltonian dynamics, quantum mechanics, quantum information, etc. See \cite{bhatia2015symplectic,bovdi2018symplectic,simon1999congruences} for detailed discussions. The establishment of the name ``Williamson's theorem" is very well explained in Section 1 of \cite{ikramov2018symplectic} or in the bibliography of \cite{simon1999congruences}. 

\par The positive real values appearing on the diagonal of $\Sigma$ are called \textit{symplectic eigenvalues} \cite{bhatia2015symplectic,ikramov2018symplectic}. (Note that a similar terminology, \textit{symplectic eigenproblem} may refer to another concept, the eigenvalue structure of symplectic matrices.) Since this is another special eigenvalues, we again justify the name ``symplectic SVD," following the discussion in Remark \ref{rem:glsvdnaming}. 

\par On the other hand, another corollary follows from the right folding of the symplectic SVD.
\begin{corollary}\label{cor:F10rightfolding}
Any real $2n\times 2n$ invertible skew-symmetric matrix $A$ has the following factorization:
\begin{equation}
    A = O\twotworr{0}{\Lambda}{-\Lambda}{0}O^T
\end{equation}
where $O\in\ortho{2n}$ is an orthogonal matrix, $\Lambda$ is an $n\times n$ positive diagonal matrix. 
\end{corollary}
This result is well studied in \cite{xu2003svd}. The set of all invertible skew-symmetric matrices is the collection of $GJ_nG^T$ for invertible $G$, due to the $BJB^T$ factorization \cite{benner2000cholesky}. The right folding of the symplectic SVD $G = O\stwotwo{\Sigma}{}{}{\Sigma} S$ gives $GJG^T = O\stwotwo{\Sigma}{}{}{\Sigma} J\stwotwo{\Sigma}{}{}{\Sigma} O^T$, equivalent to the above result by setting $\Lambda = \Sigma^2$. 

\par A complex version of $\fact{10}$ and Corollary \ref{thm:symplecticcong} follows. Corollary \ref{cor:F10rightfolding} is also extended to the complex skew-symmetric cases which we will omit.
\begin{theorem}[symplectic SVD, complex]\label{thm:F10complex}
For any $2n\times 2n$ complex invertible matrix $G$, there exist 
\begin{itemize}
    \item A $2n\times 2n$ unitary matrix $U\in\un{2n}$,
    \item A $2n\times 2n$ complex symplectic matrix $S\in\symp{2n, \mathbb{C}}$,
    \item A unique $n\times n$ positive diagonal matrix $\Sigma$ (up to diagonal permutation),
\end{itemize}
such that the following factorization holds:
\begin{equation}
    G = U \twotwo{\Sigma}{}{}{\Sigma} S.
\end{equation}
\end{theorem}

\begin{corollary}
Any $2n\times 2n$ complex Hermitian positive definite matrix $A$ is $H$-congruent to a positive diagonal matrix $D$ with repeating diagonal values, by a complex symplectic matrix $S\in\symp{2n, \mathbb{C}}$. 
\end{corollary}

\subsection{${\fact{11}}$ : ${(\Gl_\beta(2n), \Un_{\beta}(2n), \Gl_{2\beta}(n))}$}\label{sec:F11}

\begin{gather*}
\boxed{
{\renewcommand{\arraystretch}{1.4}
\begin{bmatrix}
\rule{0cm}{0.1cm}\\ 2n\times 2n \\ \hspace{0.3cm} \text{$\beta$-invertible} \hspace{0.3cm}\, \\ \rule{0cm}{0.1cm}
\end{bmatrix}}
= 
{\renewcommand{\arraystretch}{1.4}
\begin{bmatrix}
\rule{0cm}{0.1cm}\\ 2n\times 2n \\ \hspace{0.4cm} \text{$\beta$-unitary} \hspace{0.4cm}\, \\ \rule{0cm}{0.1cm}
\end{bmatrix}}
\cdot
{\small\begin{bmatrix}
    \makebox(30,30){\makecell{$\diagdown$ \hspace{0.7cm}\,\\$\sigma_l$\\ \hspace{0.7cm}$\diagdown$}} & 
    \hspace{-0.2cm}\makebox(30,30){} \\
    \makebox(30,30){} & 
    \hspace{-0.2cm}\makebox(30,30){\makecell{$\diagdown$ \hspace{0.75cm}\,\\$\frac{1}{\sigma_l}$\\ \hspace{0.7cm}$\diagdown$}} 
\end{bmatrix}}
\cdot
\overbrace{\small\begin{bmatrix}
\colorbox{gray!25}{\makebox(27,27){$\,$}}
& \hspace{-0.35cm}\colorbox{gray!50}{\makebox(27,27){$\,$}}\\ 
\colorbox{gray!50}{\makebox(27,27){$\,$}} 
& \hspace{-0.35cm}\colorbox{gray!25}{\makebox(27,27){$\,$}}
\end{bmatrix}}^{\text{realified/complexified}}
\hspace{-2.7cm}{\renewcommand{\arraystretch}{1.25}
\begin{matrix}
 n\times n \\ \hspace{0.3cm}\text{$2\beta$-invertible} \hspace{0.3cm}\,
\end{matrix}}
}
\end{gather*}

The factorization $\fact{11}$ of a size $2n\times 2n$ $\beta$-invertible matrix has a factor of $\Gl_{2\beta}(n)$, with the realify or the complexify map applied. For real invertible matrices, we have the following factorization $\fact{11}$.
\begin{theorem}[$\fact{11}$, real]\label{thm:F11real}
For any $2n\times 2n$ real invertible matrix $G$, there exist 
\begin{itemize}
    \item A $2n\times 2n$ orthogonal matrix $O\in\ortho{2n}$,
    \item Real matrices $X, Y\in\mathbb{R}^{n\times n}$ such that $X + iY$ is complex invertible,
    \item A unique $n\times n$ positive diagonal matrix $\Sigma$ (up to diagonal permutation),
\end{itemize}
such that the following factorization holds:
\begin{equation}
    G = O \twotwo{\Sigma}{}{}{\Sigma^{-1}}\twotworr{X}{Y}{-Y}{X}.
\end{equation}
The last factor $V = \begin{bsmallmatrix}\,\,X & Y \\ -Y & X \end{bsmallmatrix}$ is a realified complex invertible matrix $\realify{X+iY}$, characterized by $-J_nVJ_n = V$.
\end{theorem}

Again with the left folding $G\mapsto G^TG$ and the Cholesky factorization, we obtain an interesting result on the congruence diagonalization of a symmetric positive definite matrix. 

\begin{corollary}\label{cor:F11real}
Any $2n\times 2n$ symmetric positive definite matrix is congruent to a positive diagonal matrix by the realified complex matrix. Moreover, the resulting diagonal matrix has a block structure $\diag(\Sigma, \Sigma^{-1})$. 
\end{corollary}

\par The factorization $\fact{11}$ for complex invertible matrices is as follows. 
\begin{theorem}[$\fact{11}$, complex]\label{thm:F11complex}
For any $2n\times 2n$ complex invertible matrix $G$, there exist 
\begin{itemize}
    \item A $2n\times 2n$ unitary matrix $U\in\un{2n}$, 
    \item A complexified $n\times n$ quaternionic invertible matrix $V\in\complexify{\gl{n, \mathbb{H}}}$
    \item A unique $n\times n$ positive diagonal matrix $\Sigma$ (up to diagonal permutation),
\end{itemize} 
such that the following factorization holds:
\begin{equation}
    G = U \twotwo{\Sigma}{}{}{\Sigma^{-1}}V.
\end{equation}
\end{theorem}

Note that in the classical symbol, $V$ is an element of $\Un^*(2n)$. Again by the Cholesky factorization we obtain the folded factorization of $\fact{11}$, an $H$-congruence diagonalization a Hermitian positive definite matrix. 
\begin{corollary}
Any $2n\times 2n$ Hermitian positive definite matrix $A$ is $H$-congruent to a positive diagonal matrix by a complexified quaternionic matrix $V\in\complexify{\gl{n, \mathbb{H}}}=\Un^*(2n)$. Moreover, the resulting diagonal matrix can be chosen to have a block structure of $\diag(\Sigma, \Sigma^{-1})$
\end{corollary}

\subsection{${\fact{12}}$ : ${(\Gl_\beta(n), \Un_{\beta}(n), \Gl_{\beta/2}(n))}$}

\begin{gather*}
\boxed{
\begin{bmatrix}
\rule{0cm}{0.1cm}\\ n\times n \\ \text{$\beta$-invertible} \\\rule{0cm}{0.1cm}
\end{bmatrix}
= 
\begin{bmatrix}
\rule{0cm}{0.1cm}\\ n\times n \\ \hspace{0.15cm}\text{$\beta$-unitary}\hspace{0.15cm}\, \\\rule{0cm}{0.1cm}
\end{bmatrix}
\cdot
{\renewcommand{\arraystretch}{0.8}
\begin{bmatrix}
\text{\Large{$\diagdown$}} \hspace{1.1cm}\,\vspace{0.1cm} \\
\hspace{0.3cm}\text{\tiny{$\text{ch}_l$ $i\text{sh}_l$}} \\
\text{\tiny{$-i\text{sh}_l$ $\text{ch}_l$}}\\
\hspace{1.3cm}\,\vspace{0.1cm} \text{\Large{$\diagdown$}} \hspace{0.2cm}\\
\end{bmatrix}}
\cdot
\begin{bmatrix}
\rule{0cm}{0.1cm}\\ n\times n \\ \text{$\frac{\beta}{2}$-invertible} \\\rule{0cm}{0.1cm}
\end{bmatrix}
}
\end{gather*}

The factorization $\fact{12}$ of complex and quaternionic invertible matrices is introduced. The factors include the real part of a complex matrix, or the complex part (the real part and a single imaginary part) of a quaternionic matrix. We start with the factorization of a complex invertible matrix. 

\begin{theorem}[$\fact{12}$, complex]\label{thm:F12real}
Let $m=\lfloor\frac{n}{2}\rfloor$. For any $n\times n$ complex invertible matrix $G$, there exist 
\begin{itemize}
    \item An $n\times n$ unitary matrix $U\in\un{n}$, 
    \item An $n\times n$ real invertible matrix $V\in\gl{n, \mathbb{R}}$,
    \item $m$ unique (up to order and signs) real numbers $\theta_1, \dots, \theta_m$, 
\end{itemize}
such that the following factorization holds:
\begin{equation}
    G = UB_n^{i, \theta} V,
\end{equation}
where $B_n^{i, \theta}$ is a block diagonal matrix with $2\times 2$ blocks defined in Table \ref{tab:auxmatrix}. 
\end{theorem}

The $2\times 2$ blocks $\stwotwo{\hspace{0.2cm}\cosh\alpha}{i\sinh\alpha}{-i\sinh\alpha}{\,\,\cosh\alpha}$ on the diagonal of $B_n^{i, \alpha}$ can be thought as the usual rotation matrix with purely imaginary angles  $\stwotwo{\,\,\cos(i\alpha)}{\sin(i\alpha)}{-\sin(i\alpha)}{\cos(i\alpha)}$, and \cite{lansey2009visualizing} provides some helpful intuitions with illustrations. 

\par Another way of describing complex $\fact{12}$ is the following. Any complex invertible matrix can be $2\times 2$ block diagonalized by the combination of the (left) unitary transformation and the (right) real linear transformation. The left folding of Theorem \ref{thm:F12real} implies a simultaneous congruence transform of the real and imaginary parts of a Hermitian positive definite matrix. 
\begin{corollary}\label{cor:F12real}
Any $n\times n$ Hermitian positive definite matrix is congruent to a block diagonal matrix with $2\times 2$ diagonal blocks $B^i_1, \dots, B^i_m$ defined in Table \ref{tab:auxmatrix}, by a real $n\times n$ invertible matrix. 
\end{corollary}

The significance of Corollary \ref{cor:F12real} is that it decouples the real and the imaginary part of a complex Hermitian positive definite matrix $A$ with a simultaneous congruence diagonalization: 
\begin{equation}\label{eq:F12decoupling}
    \Re(A)= V^T\begin{bsmallmatrix}
    c_1 & 0 & & & \\
    0 & c_1 & & & \\
    & & \ddots & &\\
    & & & c_n & 0 \\
    & & & 0 & c_n\\
    \end{bsmallmatrix}V, \hspace{0.5cm} 
    \Im(A) =  V^T\begin{bsmallmatrix}
    0 & s_1 & & & \\
    -s_1 & 0 & & & \\
    & & \ddots & &\\
    & & & 0 & s_n \\
    & & & -s_n & 0\\
    \end{bsmallmatrix}V,
\end{equation}
where $V$ is an invertible real matrix. In equation \eqref{eq:F12decoupling}, the real part $\Re(A)$ is real symmetric and the imaginary part $\Im(B)$ is real skew-symmetric. 

\par The quaternionic version of the factorization $\fact{12}$ is the following.
\begin{theorem}[$\fact{12}$, quaternion]\label{thm:F12quaternion}
Let $m = \lfloor{\frac{n}{2}}\rfloor$ and fix $\eta\in\{j, k\}$. For any $n\times n$ quaternionic invertible matrix $G$, there exist
\begin{itemize}
    \item An $n\times n$ quaternionic unitary matrix $U\in\un{n, \mathbb{H}}$,
    \item An $n\times n$ complex invertible matrix $V\in\gl{n, \mathbb{C}}$,
    \item $m$ unique (up to order and signs) real numbers $\theta_1, \dots, \theta_m$,
\end{itemize}  
such that the following factorization of $G$ holds:
\begin{equation}
    G = UB_n^{\eta, \theta}V,
\end{equation}
with block diagonal matrix $B_n^{\eta, \theta}$ defined in Table \ref{tab:auxmatrix}.
\end{theorem}

Finally the congruence block diagonalization of a quaternionic Hermitian positive matrix is obtained by the left folding of Theorem \ref{thm:F12quaternion}. Similar to Corollary \ref{cor:F12real}, it is a decoupled simultaneous diagonalization of the $(1, i)$ basis and $(j, k)$ basis part of a quaternionic Hermitian positive definite matrix. 

\begin{corollary}
Fix $\eta\in\{j, k\}$. Any $n\times n$ quaternionic Hermitian positive definite matrix is congruent to a block diagonal matrix with $2\times 2$ blocks $B_1^\eta, \dots, B_m^\eta$ defined in Table \ref{tab:auxmatrix}, by an $n\times n$ real invertible matrix. 
\end{corollary}

\subsection{${\fact{13}}$ : ${(\Gl_\beta(n), \Un_{\beta}(n), \Ortho_{\beta}(n))}$}\label{sec:F13}

\begin{gather*}
\boxed{
\begin{bmatrix}
\rule{0cm}{0.1cm}\\ n\times n \\ \text{$\beta$-invertible} \\ \rule{0cm}{0.1cm}
\end{bmatrix}
= 
\begin{bmatrix}
\rule{0cm}{0.1cm}\\ n\times n \\ \hspace{0.1cm} \text{ $\beta$-unitary } \hspace{0.1cm}\, \\ \rule{0cm}{0.1cm}
\end{bmatrix}
\cdot
{\renewcommand{\arraystretch}{1.3}
\begin{bmatrix}
\LARGE{\text{$\diagdown$}} \rule{1.2cm}{0cm} \\ \sigma_l \\  \rule{1.1cm}{0cm}\LARGE{\text{$\diagdown$}}
\end{bmatrix}}
\cdot
\begin{bmatrix}
\rule{0cm}{0.1cm}\\ n\times n \\ \text{$\beta$-orthogonal} \\ \rule{0cm}{0.1cm}
\end{bmatrix}
}
\end{gather*}

The factorization $\fact{13}$ of complex and quaternionic invertible matrices has the complex and quaternionic orthogonal factors. The $\beta$-orthogonal groups are discussed in Section \ref{sec:orthogonal}. As mentioned, the quaternionic orthogonal matrices do not satisfy the property $O^TO = I$, since the set of such quaternion matrices do not form a group. Instead, the quaternionic orthogonal matrices satisfy $O^{D_j}O = I$. 

\par In this section, only $\beta=2, 4$ cases are presented for $\fact{13}$. However if we take a closer look, the case of $\beta = 1$ is identical to the invertible square SVD (real $\fact{7}$), since $\Un_\beta(n) = \Ortho_\beta(n)$ for $\beta = 1$. Let us start with the case of $\beta = 2$. 

\begin{theorem}[$\fact{13}$, complex]\label{thm:F13complex}
For any $n\times n$ complex invertible matrix $G$, there exist 
\begin{itemize}
    \item An $n\times n$ unitary matrix $U\in\un{n}$,
    \item An $n\times n$ complex orthogonal matrix $O\in\ortho{n, \mathbb{C}}$, 
    \item A unique $n\times n$ positive diagonal matrix $\Sigma$ (up to diagonal permutation),
\end{itemize} 
such that the following factorization holds:
\begin{equation}\label{eq:F13complex}
    G = U\Sigma O.
\end{equation}
\end{theorem}

Already introduced in \eqref{eq:orthosvdintro}, Theorem \ref{thm:F13complex} is another variant of the SVD with $\beta$-orthogonal groups $\Ortho_\beta(n)$. As discussed in Section \ref{sec:folding}, the right folding of Theorem \ref{thm:F13complex} is the famous \textit{Takagi factorization}. Before presenting the Takagi factorization, we need a Lemma.
\begin{lemma}\label{lem:compsymsqrt}
The set $\{G^TG:G\in\gl{n, \mathbb{C}}\}$ is equal to the set of all invertible complex symmetric matrices. In other words, for any invertible complex symmetric matrix $A$, there exists a complex invertible matrix $G$ such that $A = G^TG$ holds. 
\end{lemma}
\begin{proof}
This is easily proved by the Takagi factorization (Theorem \ref{thm:Takagicomp}). However, since we will use this Lemma to prove Theorem \ref{thm:Takagicomp}, we proceed differently.  
\par Select a polynomial interpolation $p$ of the matrix square root function (with analytic branches on the eigenvalues of $A$) on $A$, which satisfies $p(A)^2 = A$  (for example, Lagrange-Hermite interpolating polynomial). This agrees with the usual definition of the square root of a complex matrix, see \cite[Chapter 6]{horn1991topics}. Since $A$ is complex symmetric, $p(A)$ is also complex symmetric. Thus, letting $G = p(A)$ we have $G^TG = G^2 = A$. 
\end{proof}

The Takagi factorization follows from the right folding of \eqref{eq:F13complex}. In \cite{bunse1988singular}, Theorem \ref{thm:Takagicomp} is also called the symmetric singular value decomposition.
\begin{theorem}[Takagi factorization]\label{thm:Takagicomp}
For any $n\times n$ square invertible complex symmetric matrix $A$, there exist a unitary matrix $U\in\un{n}$ and a real positive diagonal matrix $\Lambda$ such that the following factorization holds:
\begin{equation}\label{eq:Takagicomp}
    A = U\Lambda U^T.
\end{equation}
Note that the transpose is the regular transpose, not the conjugate transpose. 
\end{theorem}

\par The factorization $\fact{13}$ in the quaternionic case is the following.
\begin{theorem}[$\fact{13}$, quaternion]\label{thm:F13quaternion}
Fix $\eta \in\{i, j, k\}$. For any $n\times n$ quaternionic invertible matrix $G$, there exist
\begin{itemize}
    \item An $n\times n$ quaternionic unitary matrix $U\in\un{n, \mathbb{H}}$,
    \item An $n\times n$ quaternionic orthogonal matrix $O\in\ortho{n, \mathbb{H}}$, i.e., $O^{D_\eta}O = I_n$,
    \item A unique $n\times n$ positive diagonal matrix $\Sigma$ (up to diagonal permutation),
\end{itemize}
such that the following factorization holds:
\begin{equation}
    G = U\Sigma O.
\end{equation}
\end{theorem}

Observing the extension of a complex orthogonal matrix to a quaternionic orthogonal matrix in Section \ref{sec:orthogonal}, we notice that the corresponding quaternionic extension of a complex symmetric matrix is a quaternionic $\eta$-Hermitian matrix defined in \eqref{eq:etaDdefinition} as $M^{D_\eta} = M$. Indeed, the quaternionic extension of Theorem \ref{thm:Takagicomp} is the right folding of Theorem \ref{thm:F13quaternion} and it is pointed out by Horn and Zhang in \cite{horn2012generalization}. Lemma \ref{lem:compsymsqrt} can also be extended to the quaternionic version as follows. (Fix $\eta=j$.)
\begin{lemma}
The set $\{G^{D_j}G: G\in\gl{n, \mathbb{H}}\}$ is equal to the set of all invertible quaternionic $j$-Hermitian matrices. In other words, for any invertible quaternionic $j$-Hermitian matrix $A$, there exist a quaternionic invertible matrix $G$ such that $A = G^{D_j}G$ holds.
\end{lemma}
\begin{proof}
A complex matrix $\complexify{A}$ has the structure $\stwotwo{X}{Y}{-\overline{Y}}{\overline{X}}$ where $X$ is complex symmetric and $Y$ is complex skew-Hermitian. For any $k$ the matrix $(\complexify{A})^k$ has the same block structure. As in Lemma \ref{lem:compsymsqrt} the square root of $\complexify{A}$ can be expressed as a polynomial of $A$, thereby concluding that (a branch of) $(\complexify{A})^{\frac{1}{2}}$ has the same structure $\stwotwo{X}{Y}{-\overline{Y}}{\overline{X}}$ as above. The inverse complexify map sends this matrix to a quaternionic matrix $G$ and $G^{D_j}G = G^2 = A$ holds. 
\end{proof}

\begin{corollary}[Takagi factorization, quaternionic extension]
Fix $\eta \in\{i, j, k\}$. Any $\eta$-Hermitian matrix $A$ can be decomposed into 
\begin{equation}
    A = V\Lambda V^{D_\eta},
\end{equation}
where $\Lambda$ is a nonnegative diagonal matrix and $V\in\Ortho_\eta(n, \mathbb{H})$ is a quaternionic orthogonal matrix satisfying $V^{D_\eta}V = I_n$. 
\end{corollary}



\section{Matrix factorizations of symplectic matrices}\label{sec:sympfact}
In Section \ref{sec:sympfact} we discuss matrix factorizations of symplectic matrices, arising from the generalized Cartan decompositions of $\symp{2n, \mathbb{R}}$ and $\symp{2n, \mathbb{C}}$. There are four categories, $\fact{14}$ to $\fact{17}$, including the KAK decomposition $\fact{14}$. As always, using the $\beta$-symbol for symplectic groups, $\Symp_\beta(2n)$ with $\beta = 1, 2$, the generalized Cartan triples $(G, K_\sigma, K_\tau)$ are the following.
\begin{alignat*}{2}
    \fact{14}&\text{ : }\big(\Symp_\beta(2n), \Un_{2\beta}(n), \Un_{2\beta}(n)\big) \,\,\text{(KAK)} \hspace{4.5cm} && \beta = 1, 2 \\
    \fact{15}&\text{ : }\big(\Symp_\beta(2n), \Un_{2\beta}(n), \Gl_\beta(n)\big) &&\beta=1, 2\\
    \fact{16}&\text{ : }\big(\Symp_\beta(2n), \Un_{2\beta}(n), \Symp_\beta(2p)\times\Symp_\beta(2q)\big) &&\beta=1, 2\\
    \fact{17}&\text{ : }\big(\Symp_\beta(2n), \Un_{2\beta}(n), \Symp_{\beta/2}(2n)\big) &&\beta=2
\end{alignat*}

In this section, the subgroup $K_\sigma = \Un_{2\beta}(n)$ is better understood as the $\beta$-unitary symplectic groups rather than the $2\beta$-unitary groups. (Since they are always realified or complexified.) See Remark \ref{rem:betadoubling} for the isomorphism.

\subsection{${\fact{14}}$ : ${(\Symp_\beta(2n), \Un_{2\beta}(n), \Un_{2\beta}(n))}$, the symplectic structured SVD}
\begin{gather*}
\boxed{
\begin{bmatrix}
\rule{0cm}{0.4cm}\\ 2n\times 2n \\ \text{{ Symplectic }} \\ \rule{0cm}{0.4cm}
\end{bmatrix}
= 
{\renewcommand{\arraystretch}{1.5}
\begin{bmatrix}
2n\times 2n \\\text{Unitary} \\ \text{ Symplectic }
\end{bmatrix}}
\cdot
{\renewcommand{\arraystretch}{0.65}
\begin{bmatrix}
    \diagdown \rule{1.5cm}{0cm} \\
    \sigma_l  \rule{0.9cm}{0cm}\\
     \diagdown \rule{0.3cm}{0cm}\\
     \hspace{0.4cm} \diagdown  \\
    \hspace{1cm} {\frac{1}{\sigma_l}} \\
    \hspace{1.8cm} \diagdown \\
\end{bmatrix}}
\cdot
{\renewcommand{\arraystretch}{1.5}
\begin{bmatrix}
2n\times 2n \\\text{Unitary} \\ \text{ Symplectic }
\end{bmatrix}}
}
\end{gather*}

The KAK decompositions of the real and the complex symplectic groups are introduced. They can be realized as special SVDs with a structure, $S = U\Sigma V$ where all $U, \Sigma, V$ are symplectic. We call this the \textit{symplectic structured SVD}. 
\begin{theorem}[symplectic structured SVD, real]\label{thm:F14real}
For any real symplectic matrix $S\in\symp{2n, \mathbb{R}}$, there exists an SVD of $S$, 
\begin{equation}\label{eq:F14real}
    S = U \twotwo{\Sigma}{}{}{\Sigma^{-1}} V,
\end{equation}
where $U, V\in\osp{2n}$ are $2n\times 2n$ orthogonal symplectic matrices and $\Sigma$ is an $n\times n$ positive diagonal matrix.
\end{theorem}
The complex $\fact{14}$ follows as the KAK decomposition of $\symp{2n, \mathbb{C}}/\usp{2n}$.
\begin{theorem}[symplectic structured SVD, complex]\label{thm:F14complex}
For any complex symplectic matrix $S\in\symp{2n, \mathbb{C}}$, there exists an SVD of $S$, 
\begin{equation}
    S = U \twotwo{\Sigma}{}{}{\Sigma^{-1}} V
\end{equation}
wehre $U, V\in\usp{2n}$ are $2n\times 2n$ unitary symplectic matrices and $\Sigma$ is an $n\times n$ positive diagonal matrix.
\end{theorem}

\begin{remark}[three symplectic structured SVDs]\label{rem:threesympsvds}
The KAK decompositions of the symplectic matrices, Theorems \ref{thm:F14real}, \ref{thm:F14complex} and Corollary \ref{cor:conjsympSVD} (which will be introduced in Section \ref{sec:F18}), are the three ``symplectic structured SVDs" of $\symp{2n, \mathbb{R}}$, $\symp{2n, \mathbb{C}}$, and $\Symp^*(2n, \mathbb{C})$, respectively. These three SVDs have an interesting intertwining history of appearances in different languages and different fields as the following. 

\par The \textit{real symplectic structured SVD}, Theorem \ref{thm:F14real} (real $\fact{14}$), appears first in 1965 in a physics paper \cite{balian1965forme} written in French. It showed up again in 2003 \cite[Theorem 3]{xu2003svd} without knowing the presence of \cite{balian1965forme}. In the following Remark, we explain how Theorem \ref{thm:F14real} obtained the name ``Bloch-Messiah decomposition" in physics.

\par The \textit{complex symplectic structured SVD}, Theorem \ref{thm:F14complex} (complex $\fact{14}$), appears first in Fa{\ss}bender and Ikramov \cite[Theorem 5]{fassbender2005several} and referred by Mackey et al. \cite{mackey2005structured}. The folding of Theorem \ref{thm:F14complex} appears in \cite[Proposition 3]{fassbender2005several}.

\par The \textit{conjugate symplectic structured SVD} is similarly defined as the SVD of a conjugate symplectic matrix with $U, \Sigma, V$ being conjugate symplectic. It is Corollary \ref{cor:conjsympSVD}, which will appear in Section \ref{sec:F18} as an isomorphic form of complex $\fact{18}$. (The complex $\fact{18}$ first appeared in 1968 \cite{wigner1968generalization}.) The conjugate symplectic structured SVD appears also in \cite[Theorem 2]{xu2003svd} which is cited again by \cite{mackey2005structured}. 

\par The symplectic structured SVD is sometimes referred as the ``symplectic SVD" \cite{mackey2005structured}. However as we discussed in Remark \ref{rem:glsvdnaming}, the name ``symplectic SVD" is more consistent with the factorizations in Section \ref{sec:F10}, which we use an alternative name ``symplectic structured SVD" for these factorizations.
\end{remark}

\begin{remark}[Bloch-Messiah decomposition and real $\fact{14}$]\label{rem:blochmessiah}
In physics literature, Theorem \ref{thm:F14real} has the name \textit{Bloch-Messiah decomposition}. In Bloch and Messiah's original paper in 1962 \cite{bloch1962canonical}, the factorization of the Bogoliubov transformation for fermions are presented:
\begin{equation}\label{eq:blochmessiahfermion}
    \twotworr{A}{\overline{B}}{B}{\overline{A}} = \twotwo{U}{}{}{\overline{U}}\twotwo{C}{S}{-S}{C}\twotwo{V}{}{}{\overline{V}},
\end{equation}
where $U, V$ are $n\times n$ unitary matrices and $M_f = \stwotwo{A}{\overline{B}}{B}{\overline{A}} \in \un{2n}$ is implied from the anticommutator conserving property (the Bogoliubov transformation $M_f$ for fermions). In fact, \eqref{eq:blochmessiahfermion} is an isomorphic form of complexified Theorem \ref{thm:F1quaternion}, quaternion $\fact{1}$, the KAK decomposition of the compact symmetric space C$\RN{1}$. Notice that the noncompact dual of the quaternion $\fact{1}$ is Theorem \ref{thm:F14real}, real $\fact{14}$.

\par Later in 1965 the result has been extended to bosons \cite{balian1965forme}. The commutator conserving property (the Bogoliubov transformation $M_b$ for bosons) implies $M_b=\stwotwo{A}{\overline{B}}{B}{\overline{A}}\in\un{n, n}$. The set of such $M_b$ (which is a subset of $\un{n, n}$) is isomorphic to the group $\symp{2n, \mathbb{R}}$.\footnote{What is the map sending the set of $M_b$ to $\symp{2n, \mathbb{R}}$? It is exactly the (restriction of) isomorphism sending $\un{n, n}$ to $\Symp^*(2n, \mathbb{C})$, discussed in Remark \ref{rem:conjsympandunn}. Therefore the map sends the set $\{\text{All }M_b\text{'s}\}\subset\un{n, n}$ to $\symp{2n, \mathbb{R}}\subset\Symp^*(2n, \mathbb{C})$.} Thus, a factorization similar to \eqref{eq:blochmessiahfermion} is transformed to the current Bloch-Messiah decomposition, Theorem \ref{thm:F14real}, as presented in (4.23) of \cite{balian1965forme}. In language of our paper this extension to bosons is equivalent to the extension of the KAK decomposition to the noncompact dual.
\end{remark}

\subsection{${\fact{15}}$ : ${(\Symp_\beta(2n), \Un_{2\beta}(n), \Gl_\beta(n))}$}

\begin{gather*}
\boxed{
{\renewcommand{\arraystretch}{1.4}
\begin{bmatrix}
\rule{0cm}{0.1cm}\\ 2n\times 2n \\ \hspace{0.4cm} \text{Symplectic} \hspace{0.4cm}\, \\ \rule{0cm}{0.1cm}
\end{bmatrix}}
= 
{\renewcommand{\arraystretch}{1.8}
\begin{bmatrix}
2n\times 2n \\ \hspace{0.5cm}\text{Unitary} \hspace{0.5cm}\, \\ \text{Symplectic}
\end{bmatrix}}
\cdot
\small{\begin{bmatrix}
    \makebox(30,30){\makecell{$\diagdown$ \hspace{0.7cm}\,\\$\cosh\theta_l$\\ \hspace{0.7cm}$\diagdown$}} & 
    \hspace{-0.2cm}\makebox(30,30){\makecell{$\diagdown$ \hspace{0.7cm}\,\\$\sinh\theta_l$\\ \hspace{0.7cm}$\diagdown$}} \\
    \makebox(30,30){\makecell{$\diagdown$ \hspace{0.7cm}\,\\$\sinh\theta_l$\\ \hspace{0.7cm}$\diagdown$}} & 
    \hspace{-0.2cm}\makebox(30,30){\makecell{$\diagdown$ \hspace{0.7cm}\,\\$\cosh\theta_l$\\ \hspace{0.7cm}$\diagdown$}} 
\end{bmatrix}}
\cdot
\begin{bmatrix}
\colorbox[HTML]{D6D6D6}{\makebox(27,27){\makecell{\small{$n\times n$}\\ \scriptsize{\text{Invertible}}\\ $G$}}}
& \\ 
 & \hspace{-0.3cm} \colorbox[HTML]{D6D6D6}{\makebox(27,27){\makecell{\small{$n\times n$}\\ \scriptsize{\text{Invertible}}\\ $G^{-T}$}}}
\end{bmatrix}
}
\end{gather*}

The factorization $\fact{15}$ of real and complex symplectic matrices is the $\kak$ decomposition of the symplectic group. First we start with the real symplectic matrices. 
\begin{theorem}[$\fact{15}$, real]\label{thm:F15real}
For any $2n\times 2n$ real symplectic matrix $S\in\symp{2n, \mathbb{R}}$, there exist
\begin{itemize}
    \item A $2n\times 2n$ orthogonal symplectic matrix $O\in\osp{2n}$,
    \item An $n\times n$ real invertible matrix $G\in\gl{n, \mathbb{R}}$
    \item $n$ unique (up to order and signs) real numbers $\theta_1, \dots, \theta_n$,
\end{itemize} 
such that the following factorization holds:
\begin{equation}
    S = O \twotwo{Ch}{Sh}{Sh}{Ch} \twotwo{G}{}{}{G^{-T}}.
\end{equation}
where $Ch, Sh$ are $n\times n$ diagonal matrices with $\cosh, \sinh$ values of $\theta_1, \dots, \theta_n$. 
\end{theorem}

\par Observe that the matrix $\diag({G}, {G^{-T}})$ is a block diagonal symplectic matrix. The immediate corollary is again obtained by the left folding, $S\mapsto S^TS$. The symplectic Cholesky factorization \cite{gama2006symplectic} guarantees the left folding of Theorem \ref{thm:F15real} to cover the whole set of symmetric positive definite symplectic matrices. 

\begin{corollary}\label{cor:F15real}
For any symmetric positive definite symplectic matrix $A$, we have the following factorization:
\begin{equation}
    A = \twotwo{G^T}{}{}{G^{-1}}\twotwo{Ch}{Sh}{Sh}{Ch}\twotwo{G}{}{}{G^{-T}},
\end{equation}
where $G$ is some $n\times n$ real invertible matrix and $Ch, Sh$ are $\cosh, \sinh$ diagonal matrices as usual.
\end{corollary}

\par Complex analogues of Theorem \ref{thm:F15real} and Corollary \ref{cor:F15real} are the following.

\begin{theorem}[$\fact{15}$, complex]\label{thm:F15complex}
For any $2n\times 2n$ complex symplectic matrix $S\in\symp{2n, \mathbb{C}}$, there exist
\begin{itemize}
    \item A $2n\times 2n$ unitary symplectic matrix $U\in\usp{2n}$,
    \item An $n\times n$ complex invertible matrix $G\in\gl{n, \mathbb{C}}$
    \item $n$ unique (up to order and signs) real numbers $\theta_1, \dots, \theta_n$,
\end{itemize} 
such that the following factorization holds:
\begin{equation}
    S = U \twotwo{Ch}{Sh}{Sh}{Ch} \twotwo{G}{}{}{G^{-T}}.
\end{equation}
Note that the transpose is not the conjugate transpose. 
\end{theorem}
\begin{corollary}
For any Hermitian positive definite symplectic matrix $A$, we have the following factorization:
\begin{equation}
    A = \twotwo{G^T}{}{}{G^{-1}}\twotwo{Ch}{Sh}{Sh}{Ch}\twotwo{\overline{G}}{}{}{G^{-H}},
\end{equation}
where $G$ is some $n\times n$ complex invertible matrix and $Ch, Sh$ are $\cosh, \sinh$ diagonal matrices as usual.
\end{corollary}

\subsection{${\fact{16}}$ : ${(\Symp_\beta(2n), \Un_{2\beta}(n), \Symp_\beta(2p)\times\Symp_\beta(2q))}$}\label{sec:F16}

\begin{gather*}
\boxed{
\begin{bmatrix}
\rule{0cm}{0.1cm}\\ 2n\times 2n \\ \text{{ Symplectic }} \\ \rule{0cm}{0.1cm}
\end{bmatrix}
\!=\! 
{\renewcommand{\arraystretch}{1.3}
\begin{bmatrix}
2n\times 2n \\\text{Unitary} \\ \text{ Symplectic }
\end{bmatrix}}
\!\cdot\!
\begin{bmatrix}
    \rule{0cm}{0.8cm}\makebox(18.5,18.5){$\begin{smallmatrix}\diagdown\hspace{0.6cm}\\ \text{ch}_l \\ \hspace{0.6cm}\diagdown    \end{smallmatrix}$}
    & \hspace{-0.15cm}\makebox(18.5,18.5){$\begin{smallmatrix}\diagdown\hspace{0.6cm}\\ \text{sh}_l \\ \hspace{0.6cm}\diagdown    \end{smallmatrix}$}
    & \hspace{-0.15cm}\makebox(18.5,18.5){$\,$} & \hspace{-0.15cm}\makebox(18.5,18.5){$\,$}
    \\
    \makebox(18.5,18.5){$\begin{smallmatrix}\diagdown\hspace{0.6cm}\\ \text{sh}_l \\ \hspace{0.6cm}\diagdown    \end{smallmatrix}$}
    & \hspace{-0.15cm}\makebox(18.5,18.5){$\begin{smallmatrix}\diagdown\hspace{0.6cm}\\ \text{ch}_l \\ \hspace{0.6cm}\diagdown    \end{smallmatrix}$}
    & \hspace{-0.15cm}\makebox(18.5,18.5){$\,$} & \hspace{-0.15cm}\makebox(18.5,18.5){$\,$}
    \\ 
    \makebox(18.5,18.5){$\,$} & \hspace{-0.15cm}\makebox(18.5,18.5){$\,$}
    & \hspace{-0.15cm}\makebox(18.5,18.5){$\begin{smallmatrix}\diagdown\hspace{0.6cm}\\ \text{ch}_l \\ \hspace{0.6cm}\diagdown    \end{smallmatrix}$}
    & \hspace{-0.15cm}\makebox(18.5,18.5){$\begin{smallmatrix}\diagdown\hspace{0.6cm}\\ -\text{sh}_l \\ \hspace{0.6cm}\diagdown    \end{smallmatrix}$}
    \\
    \makebox(18.5,18.5){$\,$} & \hspace{-0.15cm}\makebox(18.5,18.5){$\,$}
    & \hspace{-0.15cm}\makebox(18.5,18.5){$\begin{smallmatrix}\diagdown\hspace{0.6cm}\\ -\text{sh}_l \\ \hspace{0.6cm}\diagdown    \end{smallmatrix}$}
    & \hspace{-0.15cm}\makebox(18.5,18.5){$\begin{smallmatrix}\diagdown\hspace{0.6cm}\\ \text{ch}_l \\ \hspace{0.6cm}\diagdown    \end{smallmatrix}$}
\end{bmatrix}
\!\cdot\!
\underbrace{\renewcommand{\arraystretch}{1.3}
\begin{bmatrix}
\colorbox[HTML]{D6D6D6}{\makebox(34,34){\makecell{\small{$2p\times 2p$}\\ \scriptsize{\text{Symplectic}}}}}
& \\ 
 & \hspace{-0.3cm} \colorbox[HTML]{D6D6D6}{\makebox(34,34){\makecell{\small{$2q\times 2q$}\\ \scriptsize{\text{Symplectic}}}}}
\end{bmatrix}}_{\text{with permutation}}
}
\end{gather*}

Suppose $n = p + q$, $p \ge q$. In $\fact{16}$ we decompose $2n\times 2n$ symplectic matrices into smaller symplectic matrices of sizes $2p\times 2p$ and $2q\times 2q$. Let us first define a permutation matrix $P_{p, q}$.

\begin{gather}
    \nonumber\hspace{1cm}
    \small{\overbrace{}^p\overbrace{}^q\overbrace{}^{p}\overbrace{}^{q}}
    \\[-4pt]
    \vspace{-1cm}
    P_{p, q}:= 
    \begin{bmatrix}
    I_p &0 &0 &0\\
    0& 0& I_p &0 \\
    0& I_q &0 &0\\
    0&0 &0 & I_q
    \end{bmatrix}
    \hspace{-0.2cm}
    \begin{array}{l}
    \rdelim\}{1}{0mm}[\scriptsize{$p$}] \\
    \rdelim\}{1}{0mm}[\scriptsize{$p$}] \\ 
    \rdelim\}{1}{0mm}[\scriptsize{$q$}] \\
    \rdelim\}{1}{0mm}[\scriptsize{$q$}] 
    \end{array}\label{def:Ppq}
\end{gather}
The map $A\mapsto P_{p, q}^TAP_{p, q}$ does the following.
\begin{gather}
    \nonumber
    {\overbrace{}^p\overbrace{}^p\overbrace{}^{q}\overbrace{}^{q}}
    \hspace{1.9cm}
    {\overbrace{}^p\overbrace{}^q\overbrace{}^{p}\overbrace{}^{q}}\hspace{0.2cm}
    \\[-4pt]
    \begin{bmatrix}
    X_1 & X_2 & 0 & 0\\
    X_3 & X_4 & 0 & 0 \\
    0 & 0 & Y_1 & Y_2 \\
    0 & 0 & Y_3 & Y_4
    \end{bmatrix} 
    \hspace{-0.2cm}
    \begin{array}{l}
    \rdelim\}{1}{0mm}[\scriptsize{$p$}] \\
    \rdelim\}{1}{0mm}[\scriptsize{$p$}] \\ 
    \rdelim\}{1}{0mm}[\scriptsize{$q$}] \\
    \rdelim\}{1}{0mm}[\scriptsize{$q$}] 
    \end{array}
    \hspace{0.3cm}\longmapsto \begin{bmatrix}
     X_1 & 0 & X_2 & 0 \\
    0 & Y_1 & 0 & Y_2 \\
    X_3 & 0 & X_4 & 0 \\
    0 & Y_3 & 0 & Y_4 
    \end{bmatrix}
    \hspace{-0.2cm}
    \begin{array}{l}
    \rdelim\}{1}{0mm}[\scriptsize{$p$}] \\
    \rdelim\}{1}{0mm}[\scriptsize{$q$}] \\ 
    \rdelim\}{1}{0mm}[\scriptsize{$p$}] \\
    \rdelim\}{1}{0mm}[\scriptsize{$q$}] 
    \end{array} \label{eq:F16permmap}
\end{gather}
where the sizes of $A_j, B_j, C_j, D_j$ are $p\times p$, $p\times q$, $q\times p$, $q\times q$, respectively. With the permutation matrix $P_{p, q}$, we can state the factorization $\fact{16}$ as follows.
\begin{theorem}[$\fact{16}$, real/complex]\label{thm:F16real}
Let $P = P_{p,q}$ and $f:A\mapsto P^TAP$. For any $2n\times 2n$ real (resp. complex) symplectic matrix $S\in\symp{2n, \mathbb{R}}$ (resp. $S\in\symp{2n, \mathbb{C}}$), there exist 
\begin{itemize}
    \item A $2n\times 2n$ orthogonal (resp. unitary) symplectic matrix $U$,
    \item Two real (resp. complex) symplectic matrices $S_p\in\symp{2p, \mathbb{R}}$ and $S_q\in\symp{2q, \mathbb{R}}$ (resp. $S_p\in\symp{2p, \mathbb{C}}$ and $S_q\in\symp{2q, \mathbb{C}}$),
    \item $q$ unique (up to order and signs) real numbers $\theta_1, \dots, \theta_q$,
\end{itemize}
such that the following factorization holds:
\begin{equation}
    S = U \twotwo{H}{}{}{H^{-1}}f\bigg(\twotwo{S_p}{}{}{S_q}\bigg) = U
    \begin{bsmallmatrix}Ch& & Sh & & &\\ & I_{p-q} & & & & \\ Sh & & Ch & & &\\ & & & \,\,Ch & & -Sh \\ & & & & I_{p-q} &\\ & & &-Sh & &\,\,\,Ch \end{bsmallmatrix}
    \begin{bsmallmatrix}
    X_1 & 0 & X_2 & 0\\
    0 & Y_1 & 0 & Y_2 \\
    X_3 & 0 & X_4 & 0 \\
    0 & Y_3 & 0 & Y_4 
    \end{bsmallmatrix},
\end{equation}
where $S_p = \stwotwo{X_1}{X_2}{X_3}{X_4}$ and $S_q = \stwotwo{Y_1}{Y_2}{Y_3}{Y_4}$. The blocks $X_j$'s are $p\times p$ and $Y_j$'s are $q\times q$. 
\end{theorem}

\subsection{${\fact{17}}$ : ${(\Symp_\beta(2n), \Un_{2\beta}(n), \Symp_{\beta/2}(2n))}$}

\begin{gather*}
\boxed{
{\renewcommand{\arraystretch}{1.3}
\begin{bmatrix}
2n\times 2n \\ \rule{0.3cm}{0cm}\text{Complex}\rule{0.3cm}{0cm} \\ \text{ Symplectic }
\end{bmatrix}}
= 
{\renewcommand{\arraystretch}{1.3}
\begin{bmatrix}
2n\times 2n \\ \text{Unitary} \\ \text{ Symplectic }
\end{bmatrix}}
\cdot
{\small{\begin{bmatrix}
    \makebox(30,30){\makecell{$\diagdown$ \hspace{0.7cm}\,\\$\cosh\theta_l$\\ \hspace{0.7cm}$\diagdown$}} & 
    \hspace{-0.2cm}\makebox(30,30){\makecell{$\diagdown$ \hspace{0.7cm}\,\\$i\sinh\theta_l$\\ \hspace{0.7cm}$\diagdown$}} \\
    \makebox(30,30){\makecell{$\diagdown$ \hspace{0.7cm}\,\\$-i\sinh\theta_l$\\ \hspace{0.7cm}$\diagdown$}} & 
    \hspace{-0.2cm}\makebox(30,30){\makecell{$\diagdown$ \hspace{0.7cm}\,\\\,\,$\cosh\theta_l$\\ \hspace{0.7cm}$\diagdown$}} 
\end{bmatrix}}}
\cdot
{\renewcommand{\arraystretch}{1.3}
\begin{bmatrix}
2n\times 2n \\ \text{Real} \\ \text{ Symplectic }
\end{bmatrix}}
}
\end{gather*}

%
%
%

The factorization $\fact{17}$ is a matrix factorization of complex symplectic matrices. As $\symp{2n, \mathbb{R}}\subset \symp{2n, \mathbb{C}}$, the real symplectic factor of a complex symplectic matrix is factored out in $\fact{17}$. The factorization $\fact{17}$ is as follows.
\begin{theorem}[$\fact{17}$, complex]\label{thm:F17complex}
For any $2n\times 2n$ complex symplectic matrix $S\in\symp{2n, \mathbb{C}}$, there exist
\begin{itemize}
    \item A unitary symplectic matrix $U\in\usp{2n}$,
    \item A real symplectic matrix $Y\in\symp{2n, \mathbb{R}}$
    \item $n$ unique (up to order and signs) real numbers $\theta_1, \dots, \theta_n$,
\end{itemize} 
such that the following factorization holds:
\begin{equation}
    S = U\twotworr{Ch}{iSh}{-iSh}{Ch}Y.
\end{equation}
\end{theorem}

This is in fact Theorem \ref{thm:F12real} ($\fact{12}$ real) with the additional symplectic structure. (In $\fact{12}$, an alternative selection of the maximal abelian subgroup $A= \{\text{All } \stwotwo{\,\,\,Ch}{iSh}{-iSh}{Ch}\}$ is equivalent to the subgroup $A$ in Theorem \ref{thm:F17complex}.) Similar to Corollary \ref{cor:F12real}, the left folding of Theorem \ref{thm:F17complex} is the decoupling of real and imaginary parts of a Hermitian positive definite symplectic matrix.

\begin{corollary}
Any $2n\times 2n$ Hermitian positive definite symplectic matrix $A$ is congruent to a tridiagonal matrix of the form $\stwotwo{\,\,\,Ch}{iSh}{-iSh}{Ch}$ where $Ch, Sh$ are $n\times n$ $\cosh, \sinh$ diagonal matrices. The congruence transformation is done by a $2n\times 2n$ real symplectic matrix. 
\end{corollary}

A decoupling of the real and the imaginary part for a given Hermitian positive definite symplectic matrix $A$ ($A^H = A$ and $A^TJA = J$) is,
\begin{equation}
    \Re(A) = S^T\twotwo{Ch}{}{}{Ch}S, \hspace{0.5cm} \Im(A) = S^T\twotwo{\,}{Sh}{-Sh}{}S,
\end{equation}
where $S$ is some $2n\times 2n$ real symplectic matrix.

\section{Matrix factorizations of indefinite orthogonal/unitary matrices}\label{sec:jorthogonalfact}
In this section we discuss the matrix factorizations of indefinite orthogonal and unitary matrices, $\fact{18}$ to $\fact{22}$. Let $n = p + q$ with $p\ge q$. Again we categorize the $\kak$ decompositions into five categories, including the KAK decomposition ($\fact{18}$, the HCSD). The generalized Cartan triple $(G, K_\sigma, K_\tau)$ of the five categories are the following:
\begin{alignat*}{2}
    \fact{18}&\text{ : }\big(\Un_\beta(p, q), \Un_{\beta}(p)\times  \Un_{\beta}(q), \Un_{\beta}(p)\times  \Un_{\beta}(q) \big) \,\,(\text{KAK = HCSD}) \hspace{1cm} &&\beta=1, 2, 4 \\
    \fact{19}&\text{ : } \big(\Un_\beta(p, q), \Un_{\beta}(p)\times  \Un_{\beta}(q), \Un_\beta(p_1, q_1)\times \Un_\beta(p_2, q_2)\big) &&\beta = 1, 2, 4\\
    \fact{20}&\text{ : } \big(\Un_\beta(2p, 2q), \Un_{\beta}(2p)\times  \Un_{\beta}(2q), \Un_{2\beta}(p, q)\big) &&\beta = 1, 2\\
    \fact{21}&\text{ : } \big(\Un_\beta(n, n), \Un_\beta(n)\times \Un_\beta(n), \Ortho_{2\beta}(n)\big) &&\beta = 1, 2\\
    \fact{22}&\text{ : } \big(\Un_\beta(p, q), \Un_\beta(p)\times \Un_\beta(q), \Un_{\beta/2}(p, q)\big) &&\beta = 2, 4\\
\end{alignat*}

\subsection{${\fact{18}}$ : ${(\Un_\beta(p, q), \Un_{\beta}(p)\times  \Un_{\beta}(q), \Un_{\beta}(p)\times  \Un_{\beta}(q))}$, the hyperbolic CSD}\label{sec:F18}

\begin{gather*}
\boxed{
{\renewcommand{\arraystretch}{1.6}
\begin{bmatrix}
n\times n \\ \hspace{0.2cm}\text{{ Indefinite }} \hspace{0.2cm}\,\\ \text{unitary}
\end{bmatrix}}
= 
\begin{bmatrix}
\colorbox[HTML]{D6D6D6}{\makebox(35,35){\makecell{$p\times p$\\\small{Unitary}}}} & \\ 
& \hspace{-0.35cm}\colorbox[HTML]{D6D6D6}{\makebox(18,18){\makecell{\small{$q\times q$}\\\tiny{Unitary}}}}
\end{bmatrix}
\cdot
\begin{bmatrix}
\begin{array}{c:c}
\makebox(33, 38){$\begin{smallmatrix}\diagdown\hspace{1.1cm}\\ \text{ch}_l \hspace{0.7cm}\\ \diagdown \hspace{0.1cm} \\ \hspace{0.7cm}\rule{0cm}{0.4cm}\text{\large{$I_{p-q}$}}\end{smallmatrix}$}& 
\makebox(17, 38){$\begin{smallmatrix}\diagdown\hspace{0.6cm}\\ \text{sh}_l\\ \hspace{0.6cm}\diagdown  \\ \rule{0cm}{0.5cm}\end{smallmatrix}$}\\ \hdashline
\makebox(38,17){$\begin{smallmatrix}\diagdown\hspace{1.2cm}\\ \text{sh}_l\hspace{0.7cm}\\ \diagdown\hspace{0.1cm}\end{smallmatrix}$}
& \makebox(17, 20){$\begin{smallmatrix}\diagdown\hspace{0.7cm}\\ \text{ch}_l\\ \hspace{0.5cm}\diagdown\end{smallmatrix}$}
\end{array}
\end{bmatrix}
\cdot
\begin{bmatrix}
\colorbox[HTML]{D6D6D6}{\makebox(35,35){\makecell{$p\times p$\\\small{Unitary}}}} & \\ 
& \hspace{-0.35cm}\colorbox[HTML]{D6D6D6}{\makebox(18,18){\makecell{\small{$q\times q$}\\\tiny{Unitary}}}}
\end{bmatrix}
}
\end{gather*}

\par The KAK decomposition of the indefinite orthogonal, indefinite unitary and quaternionic indefinite unitary group is the \textit{hyperbolic CS decomposition} (HCSD). In Wigner's 1968 work \cite[Eq.16]{wigner1968generalization} the real and the complex HCSD appeared as a counterpart of the CSD. Interestingly Wigner did not notice the CSD and the HCSD as an examples of the KAK decomposition, which was implicit in Helgason's work he cited. Rather he mentioned how it looked similar to the Iwasawa decomposition\footnote{The Iwasawa decomposition is another important Lie group decomposition which generalizes the QR factorization. For a nilpotent subgroup $N\subset G$ we have the decomposition $G = K A N$, where $K, A$ are identical to the KAK decomposition. For example if $G = \gl{n, \mathbb{R}}$, a standard choice of $N$ is the group of all upper triangular matrices with diagonal entries fixed to one. Thus, the decomposition $G = K \cdot (A\cdot N) $ is the QR factorization. One can add the positive roots of $\mathfrak{g}$ respect to $\mathfrak{a}$ to obtain a nilpotent Lie algebra $\mathfrak{n}$. As always, $N$ is a Lie subgroup of $G$ such that $\text{Lie}(N) = \mathfrak{n}$. As we are working with various $\kak$ decompositions in this paper, the Iwasawa decomposition of semisimple Lie groups could also be understood as a family of matrix factorizations.} appeared in Helgason's work. See Section \ref{sec:lesson} for details. 

\par The HCSD (Theorem \ref{thm:F18realcomplexquaternion}) was also derived as an example of the KAK decomposition in Representation theory literature, e.g., Vilenkin's book \cite{vilenkin2013representation}. In numerical linear algebra context it was introduced (with the name ``hyperbolic CS decomposition") by Grimme, Sorensen and Van Dooren \cite[Lemma 6]{grimme1996model}, with algorithmic considerations coming thereafter by Stewart and Van Dooren \cite{stewart2005factorization}. 

\par As we mentioned in Section \ref{sec:jorthogonal}, the HCSD is the hyperbolic analogue of the CSD. (See Remark \ref{rem:csdhcsdcompare} for the dual relationship.) Let $n = p + q$ and assume $p\ge q$ as always. We state the HCSD (all three $\beta$'s) as follows.

\begin{theorem}[$\fact{18}$, HCSD]\label{thm:F18realcomplexquaternion}
Fix $\beta=1,2,4$. For any $\beta$-indefinite unitary matrix $G\in\Un_\beta(p,q)$, we have the following factorization:
\begin{equation}\label{eq:F18}
    G = \begin{bmatrix}
    \begin{array}{c:c}
    U_p & \\ \hdashline & U_q\end{array}
    \end{bmatrix}
    \begin{bmatrix}\begin{array}{cc:c}
        Ch & & Sh \\
         & I_{p-q} & \\ \hdashline
         Sh & & Ch
    \end{array}\end{bmatrix}
    \begin{bmatrix}
    \begin{array}{c:c}
    U'_p & \\ \hdashline & U'_q\end{array}
    \end{bmatrix},\end{equation}
where $U_p, {U_p}'\in\Un_\beta(p)$ and $U_q, {U_q}'\in\Un_\beta(q)$. The matrices $Ch$, $Sh$ are $q\times q$ diagonal matrices with $\cosh, \sinh$ values (of $q$ unique angles, up to order and signs) on their diagonals, respectively. 
\end{theorem}

The folding (for the KAK decomposition, the left and the right folding are equivalent) of the HCSD can be used to obtain the general nonsquare SVD, as is hinted in an equation of Wigner's work \cite[Eq.29]{wigner1968generalization} but one has to recognize that the top right block, say, can be an arbitrary $p\times q$ matrix.

The collection of all $G^TG$ for $G\in\ortho{p, q}$ is the following: 
\begin{equation}
    \{G^TG\,|\,G\in\ortho{p, q}\} = \bigg\{\twotwo{\sqrt{I+XX^T}}{X}{X^T}{\sqrt{I+X^TX}}\bigg|X\in\mathbb{R}^{p\times q}\bigg\}.
\end{equation}
This can also be realized as the representatives of the noncompact symmetric space type BD$\RN{1}$. The folded factorization follows. 

\begin{corollary}[nonsquare SVD]\label{cor:F18}
For any matrix in $X\in\mathbb{R}^{p\times q}$, we have the following factorization:
\begin{equation}
    {\renewcommand{\arraystretch}{1.2}
    \begin{bmatrix}
    \begin{array}{c:c}
    \sqrt{I+XX^T} & X \\ \hdashline X^T  & \sqrt{I+X^TX} \end{array}
    \end{bmatrix}}
    = \begin{bmatrix}
    \begin{array}{c:c}
    U & \\ \hdashline & V\end{array}
    \end{bmatrix}
    \begin{bmatrix}\begin{array}{cc:c}
        Ch & & Sh \\
         & I_{p-q} & \\ \hdashline
         Sh & & Ch
    \end{array}\end{bmatrix}
    \begin{bmatrix}
    \begin{array}{c:c}
    U & \\ \hdashline & V\end{array}
    \end{bmatrix}^T.
\end{equation}
The decomposition of the upper right $p\times q$ block, $X =U\cdot Sh\cdot V^T$, represents the SVD of any $p\times q$ real matrix. 
\end{corollary}

It is a well-known fact that the eigendecomposition of the matrix $\stwotwo{0}{X}{X^\dagger}{0}$ is an equivalent form of the singular value decomposition.\footnote{This is equivalent to a Lie algebra decomposition, i.e., Lemma 6.3 ({\romannumeral 3}) of \cite[p.247]{Helgason1978}. In computational perspective the Lie algebra decomposition is pointed out by Kleinsteuber \cite{kleinsteuber2005jacobi} as the SVD.} What seems to not be so well-known is that taking the exponential also is an equivalent form of the singular value decomposition. Corollary \ref{cor:F18} is an explicit exponentiation which can be viewed in multipe ways:
\begin{itemize}
    \item The SVD can be read directly off the top right and bottom left block entries, for any matrix. (With $Sh$ playing the role of $\Sigma$.)
    \item The diagonal block entries (after squaring and subtracting the identity $I$) together give the familiar $XX^\dagger$ and $X^\dagger X$  versions of the SVD often taught in elementary linear algebra textbooks.
\end{itemize}

\begin{remark}[dual relationship between the CSD and the HCSD]\label{rem:csdhcsdcompare}
The noncompactness of $\ortho{p, q}$ and the dual relationship between $\ortho{n}$, $\ortho{p, q}$ can be better understood by comparing the CSD and the HCSD. For a fixed set of matrices $O_p, O'_p\in\ortho{p}$, $O_q, O'_q\in\ortho{q}$ and a fixed real vector $\theta\in\mathbb{R}^q$, we can define the orthogonal matrix $O_n$ and the indefinite orthogonal matrix $O_{p, q}$ with the CSD and the HCSD, respectively,
\begin{align*}
    O_n &= \twotwo{O_p}{}{}{O_q}\threediag{C}{S}{I_{p-q}}{-S}{C}\twotwo{{O_p}'}{}{}{{O_q}'},\\
    O_{p, q} &= \twotwo{O_p}{}{}{O_q}\threediag{Ch}{Sh}{I_{p-q}}{Sh}{Ch}\twotwo{{O_p}'}{}{}{{O_q}'}
\end{align*}
with the cosine/sine diagonal matrices $C, S$ and the hyperbolic cosine/sine diagonal matrice $Ch, Sh$. As we make an increasing sequence of $\theta$ we obtain diverging sequence of matrices $O_{p, q}$ while the compact dual $O_n$ converges due to the compactness of the trigonometric functions. 
\end{remark}

\begin{remark}[real perplectic group and $\ortho{\lceil \frac{n}{2}\rceil, \lfloor\frac{n}{2}\rfloor}$]\label{rem:realperplecticsvd}
Recall the matrix $E_n$ with 1's on its antidiagonal, defined in \ref{eq:Edefinition}. Let $m_1 = \lceil \frac{n}{2}\rceil$, $m_2 = \lfloor\frac{n}{2}\rfloor$. Then the matrix $V=\frac{1}{\sqrt{2}}(I_{m_1, m_2}+E_n)$ is the eigenvector matrix (if $n$ is odd set the $(m_1, m_1)$ entry of $V$ to 1 for normalization) of $E_n$, with the eigendecomposition $E_n = V  I_{m_1, m_2} V^T$. The map $A\mapsto VAV^T$ is the isomorphism that sends $\ortho{\lceil \frac{n}{2}\rceil, \lfloor\frac{n}{2}\rfloor}$ to the real perplectic group. Applying the map to the real $\fact{18}$ factorization we obtain the SVD of a perplectic matrix with all factors preserving the perplectic structure. (Similar isomorphism of factorizations have been pointed out by Kleinsteuber in his doctoral thesis \cite{kleinsteuber2005jacobi}. He discusses the isomorphism of the folded Lie algebra decompositions between real perplectic and indefinite orthogonal matrices, see Example 4.2 of \cite{kleinsteuber2005jacobi}.)
\end{remark}

\begin{corollary}[perplectic structured SVD, real]
Let $n$ be even so that $m= m_1 = m_2$. For an $n\times n$ real perplectic matrix $P$, there exist
\begin{itemize}
    \item Two $n\times n$ perplectic orthogonal matrices $O_1, O_2$
    \item An $n\times n$ positive diagonal matrix $\Sigma = \diag(\sigma_1, \cdots, \sigma_m, \frac{1}{\sigma_m}, \cdots, \frac{1}{\sigma_1})$,
\end{itemize}
such that $P = O_1 \Sigma O_2$ is the SVD of $P$ with all three factors $O_1, \Sigma, O_2$ being perplectic. (For an odd $n$, we put $1$ between $\sigma_m$ and $\frac{1}{\sigma_m}$ in $\Sigma$.)
\end{corollary}

\begin{remark}[conjugate symplectic group and $\un{n, n}$]\label{rem:conjsympandunn}
Recall the conjugate symplectic group $\Symp^*(2n, \mathbb{C})$ in Remark \ref{rem:conjsymp}. Following Remark \ref{rem:realperplecticsvd}, we can similarly compute the isomorphism between $\Symp^*(2n, \mathbb{C})$ and $\un{n, n}$. Let $V = \frac{1}{\sqrt{2}}\begin{bsmallmatrix} I_n & -iI_n\\ -iIn & I_n \end{bsmallmatrix}$. Then $V$ is unitary and the eigendecomposition $iJ_n= VI_{n, n}V^H$ holds. The isomorphism $A\mapsto VAV^H$ sends $\un{n, n}$ to $\Symp^*(2n, \mathbb{C})$. Applying this isomorphism, an isomorphic form of the complex HCSD (complex $\fact{18}$) is obtained. It is the \textit{conjugate symplectic structured SVD}. (The symplectic structured SVDs are discussed in Remark \ref{rem:threesympsvds}.)
\end{remark}
\begin{corollary}[conjugate symplectic structured SVD]\label{cor:conjsympSVD}
For any $2n\times 2n$ conjugate symplectic matrix $S$, there exist,
\begin{itemize}
    \item Two $2n\times 2n$ conjugate symplectic and unitary matrix $U, V$,
    \item A positive diagonal matrix $\Sigma = \diag(\sigma_1, \cdots, \sigma_n, \frac{1}{\sigma_1}, \cdots, \frac{1}{\sigma_n})$,
\end{itemize}
such that $S=U\Sigma V $ is the SVD of $S$ with all three factors $U, \Sigma, V$ remains conjugate symplectic. 
\end{corollary}

\subsection{${\fact{19}}$ : ${(\Un_\beta(p, q), \Un_{\beta}(p)\times  \Un_{\beta}(q), \Un_\beta(p_1, q_1)\times \Un_\beta(p_2, q_2))}$}

\begin{gather*}
\boxed{
{\renewcommand{\arraystretch}{1.3}
\begin{bmatrix}
n\times n \\ \hspace{0.2cm}\text{ Indefinite } \hspace{0.2cm}\, \\ \text{unitary} \\ (p, q)
\end{bmatrix}}
= 
\begin{bmatrix}
\colorbox[HTML]{D6D6D6}{\makebox(35,35){\makecell{$p\times p$\\\small{Unitary}}}} & \\ 
& \hspace{-0.35cm}\colorbox[HTML]{D6D6D6}{\makebox(18,18){\makecell{\small{$q\times q$}\\\tiny{Unitary}}}}
\end{bmatrix}
\cdot
\underbrace{\begin{bmatrix}
\begin{array}{c:c}
\makebox(40, 44){$\begin{smallmatrix}
\diagdown\hspace{0.55cm}\diagdown\hspace{0.6cm}\\ 
\cosh\theta_l \hspace{0.15cm} \sinh\theta_l \hspace{0.1cm}\\ 
\hspace{0.45cm}\diagdown \hspace{0.55cm}\diagdown \\ 
\diagdown\hspace{0.55cm}\diagdown\hspace{0.6cm}\\ 
\sinh\theta_l \hspace{0.15cm} \cosh\theta_l \hspace{0.1cm}\\ 
\hspace{0.45cm}\diagdown \hspace{0.55cm}\diagdown \\ 
\end{smallmatrix}$} 
& \makebox(17, 38){}\\ \hdashline
\makebox(38,17){}
& \makebox(17, 20){$\begin{smallmatrix}\diagdown\hspace{0.3cm}\diagdown\hspace{0.2cm}\\ \hspace{0.1cm}\text{ch}\gamma_l\,\text{sh}\gamma_l\\ \diagdown\hspace{0.3cm}\diagdown\end{smallmatrix}$}
\end{array}
\end{bmatrix}}_{\text{with permutation}}
\cdot
\underbrace{\begin{bmatrix}
\colorbox[HTML]{D6D6D6}{\makebox(28,28){\makecell{\tiny{Indefnite} \\ \tiny{unitary}\\\footnotesize{$(p_1, q_1)$}}}} & \\ 
& \hspace{-0.35cm}\colorbox[HTML]{D6D6D6}{\makebox(28,28){\makecell{\tiny{Indefinite} \\ \tiny{unitary}\\\footnotesize{$(p_2, q_2)$}}}}
\end{bmatrix}}_{\text{with permutation}}
}
\end{gather*}

The factorization $\fact{19}$ of indefinite orthogonal/unitary matrices is introduced. Let $p_1 + p_2 = p$ and $q_1 + q_2 = q$ be the partitions of $p, q$. With these four integers, we define two permutation matrices $P_1$ and $P_2$.
\begin{equation}
    P_1 := \begin{bmatrix}
    I_{p_1} & 0 & 0 & 0\\
    0 & 0 & I_{q_1} & 0 \\
    0 & I_{p_2} & 0 & 0\\
    0 & 0 & 0 & I_{q_2}
    \end{bmatrix}, \hspace{0.5cm}
    P_2 := \begin{bmatrix}
    I_{p_1} & 0 & 0 & 0\\
    0 & 0 & 0 & I_{q_2} \\
    0 & I_{p_2} & 0 & 0\\
    0 & 0 & I_{q_1} & 0
    \end{bmatrix}, 
\end{equation}
Similar to the map \eqref{eq:F16permmap}, the map $A\mapsto P_1^TAP_1$ does the following permutation to a block diagonal matrix. 
\begin{gather}
    \nonumber
    {\overbrace{}^{p_1}\overbrace{}^{q_1}\overbrace{}^{p_2}\overbrace{}^{q_2}}
    \hspace{1.7cm}
    {\overbrace{}^{p_1}\overbrace{}^{p_2}\overbrace{}^{q_1}\overbrace{}^{q_2}}\hspace{0.2cm}
    \\[-4pt]\label{eq:F19perm1}
    \begin{bmatrix}
    X_1 & X_2 & 0 & 0\\
    X_3 & X_4 & 0 & 0 \\
    0 & 0 & Y_1 & Y_2 \\
    0 & 0 & Y_3 & Y_4
    \end{bmatrix} 
    \hspace{-0.2cm}
    \begin{array}{l}
    \rdelim\}{1}{0mm}[\scriptsize{$p_1$}] \\
    \rdelim\}{1}{0mm}[\scriptsize{$q_1$}] \\ 
    \rdelim\}{1}{0mm}[\scriptsize{$p_2$}] \\
    \rdelim\}{1}{0mm}[\scriptsize{$q_2$}] 
    \end{array}
    \hspace{0.3cm}\longmapsto \begin{bmatrix}
    X_1 & 0 & X_2 & 0 \\
    0 & Y_1 & 0 & Y_2 \\
    X_3 & 0 & X_4 & 0 \\
    0 & Y_3 & 0 & Y_4 
    \end{bmatrix}
    \hspace{-0.2cm}
    \begin{array}{l}
    \rdelim\}{1}{0mm}[\scriptsize{$p_1$}] \\
    \rdelim\}{1}{0mm}[\scriptsize{$p_2$}] \\ 
    \rdelim\}{1}{0mm}[\scriptsize{$q_1$}] \\
    \rdelim\}{1}{0mm}[\scriptsize{$q_2$}] 
    \end{array}
\end{gather}
In particular, a matrix $A\in\ortho{p_1, q_1}\times \ortho{p_2, q_2}$ permuted to an element in $\ortho{p, q}$ with \eqref{eq:F19perm1}, as we can see in the block structure. Also the map $A\mapsto P_2^TAP_2$ does the following.
\begin{gather}\label{eq:F19perm2}
    \nonumber
    {\overbrace{}^{p_1}\overbrace{}^{q_2}\overbrace{}^{p_2}\overbrace{}^{q_1}}
    \hspace{1.7cm}
    {\overbrace{}^{p_1}\overbrace{}^{p_2}\overbrace{}^{q_1}\overbrace{}^{q_2}}\hspace{0.2cm}
    \\[-4pt]
    \begin{bmatrix}
    X_1 & X_2 & 0 & 0\\
    X_3 & X_4 & 0 & 0 \\
    0 & 0 & Y_1 & Y_2 \\
    0 & 0 & Y_3 & Y_4
    \end{bmatrix} 
    \hspace{-0.2cm}
    \begin{array}{l}
    \rdelim\}{1}{0mm}[\scriptsize{$p_1$}] \\
    \rdelim\}{1}{0mm}[\scriptsize{$q_2$}] \\ 
    \rdelim\}{1}{0mm}[\scriptsize{$p_2$}] \\
    \rdelim\}{1}{0mm}[\scriptsize{$q_1$}] 
    \end{array}
    \hspace{0.3cm}\longmapsto \begin{bmatrix}
    X_1 & 0 & 0 & X_2 \\
    0 & Y_1 & Y_2 & 0 \\
    0 & Y_3 & Y_4 & 0 \\
    X_3 & 0 & 0 & X_4 
    \end{bmatrix}
    \hspace{-0.2cm}
    \begin{array}{l}
    \rdelim\}{1}{0mm}[\scriptsize{$p_1$}] \\
    \rdelim\}{1}{0mm}[\scriptsize{$p_2$}] \\ 
    \rdelim\}{1}{0mm}[\scriptsize{$q_1$}] \\
    \rdelim\}{1}{0mm}[\scriptsize{$q_2$}] 
    \end{array}
\end{gather}
The factorization $\fact{19}$ of indefinite orthogonal matrices is as follows.

\begin{theorem}[$\fact{19}$, real/complex/quaternion]\label{thm:F19real}
For a fixed $\beta=1,2,4$, let $m_1 = \min(p_1, q_2)$ and $m_2 = \min(p_2, q_1)$. For any $\beta$-indefinite unitary matrix $U\in\Un_\beta(p,q)$, there exist
\begin{itemize}
    \item Two $\beta$-unitary matrices $U_p\in\Un_\beta(p)$, $U_q\in\Un_\beta(q)$,
    \item Two $\beta$-indefinite unitary matrices $V_1\in\Un_\beta(p_1, q_1)$, $V_2\in\Un_\beta(p_2, q_2)$,
    \item $m_1$ unique (up to order and signs) real numbers $\theta_1, \dots, \theta_{m_1}$,
    \item $m_2$ unique (up to order and signs) real numbers $\gamma_1, \dots, \gamma_{m_2}$, 
\end{itemize}
such that the following factorization holds:
\begin{equation}
    U = \twotwo{U_p}{}{}{U_q}
    P_2^T\twotwo{H^\theta}{}{}{H^\gamma}P_2 V,
\end{equation}
where $H^\theta = H_{\max(p_1, q_2), m_1}^\theta$ and $H^\gamma =H_{\max(p_2, q_1), m_2}^\gamma$, and $V$ is the permuted $\diag(V_1, V_2)$ by the transformation \eqref{eq:F19perm1},
\begin{equation}
    V := P_1^T\twotwo{V_1}{}{}{V_2}P_1. 
\end{equation}
Again, refer to Table \ref{tab:translation} for $\beta$-symbols of Lie groups. 
\end{theorem}

\par The block structure of the matrix $P_2^T\stwotwo{H^\theta}{}{}{H^\gamma}P_2$ depend on the partitions $(p_1, p_2), (q_1, q_2)$. For example, if $p_1\ge q_2 = m_1$ and $p_2\ge q_1 = m_2$ we have 
\begin{equation}
    P_2^T\twotwo{H^\theta}{}{}{H^\gamma}P_2 = 
    \begin{bsmallmatrix}
    Ch_{q_2} & & & Sh_{q_2} \\
    & I_{p_1 - q_2} & & & \\
    & & H_{p_2, q_1}^\gamma \\
    Sh_{q_2} & & & Ch_{q_2}
    \end{bsmallmatrix}.
\end{equation}

\par Essentially, the factorization $\fact{19}$ is a breakdown of the signature (more precisely the signature of the corresponding sesquilinear form) $(p, q)$ of the indefinite metric into smaller indefinite metrics with signatures $(p_1, q_1)$ and $(p_2, q_2)$. The complex and quaternionic versions of $\fact{19}$ is similar, as follows.

\subsection{${\fact{20}}$ : ${(\Un_\beta(2p, 2q), \Un_{\beta}(2p)\times  \Un_{\beta}(2q), \Un_{2\beta}(p, q))}$}

\begin{gather*}
\boxed{
{\renewcommand{\arraystretch}{1.25}
\begin{bmatrix}
2n\times 2n \\ \hspace{0.2cm}\text{$\beta$-indefinite} \hspace{0.2cm}\, \\ \text{unitary} \\ (2p, 2q)
\end{bmatrix}}
= 
\begin{bmatrix}
\colorbox[HTML]{D6D6D6}{\makebox(35,35){\makecell{$2p\times 2p$\\\small{Unitary}}}} & \\ 
& \hspace{-0.35cm}\colorbox[HTML]{D6D6D6}{\makebox(18,18){\makecell{\tiny{$2q\times 2q$}\\\tiny{Unit.}}}}
\end{bmatrix}
\cdot
\begin{bmatrix}
    \makebox(15,15){$I_{p-q}$} & \hspace{-0.15cm}\makebox(18.5,18.5){$\,$} & \hspace{-0.15cm}\makebox(13,13){$\,$} & \hspace{-0.15cm}\makebox(18.5,18.5){$\,$}
    \\
    & \hspace{-0.15cm}\makebox(18.5,18.5){$\begin{smallmatrix}\diagdown\hspace{0.6cm}\\ \text{ch}_l \\ \hspace{0.6cm}\diagdown    \end{smallmatrix}$} & & \hspace{-0.15cm}\makebox(18.5,18.5){$\begin{smallmatrix}\diagdown\hspace{0.6cm}\\ \text{sh}_l \\ \hspace{0.6cm}\diagdown    \end{smallmatrix}$}
    \\ 
    & & \hspace{-0.15cm}\makebox(13,13){$I_{p-q}$} & 
    \\
    & \hspace{-0.15cm}\makebox(18.5,18.5){$\begin{smallmatrix}\diagdown\hspace{0.6cm}\\ \text{sh}_l \\ \hspace{0.6cm}\diagdown    \end{smallmatrix}$} & & \hspace{-0.15cm}\makebox(18.5,18.5){$\begin{smallmatrix}\diagdown\hspace{0.6cm}\\ \text{ch}_l \\ \hspace{0.6cm}\diagdown    \end{smallmatrix}$}
\end{bmatrix}
\cdot
\underbrace{\overbrace{
\renewcommand{\arraystretch}{1.25}
\begin{bmatrix}
n\times n \\ \hspace{0.15cm}\text{$2\beta$-indefinite} \hspace{0.15cm}\, \\ \text{unitary} \\ (p, q)\end{bmatrix}
}^{\text{realified/complexified}}}_{\text{with permutation}}
}
\end{gather*}

The factorization $\fact{20}$ of indefinite orthogonal/unitary matrices is introduced. The realify and complexify maps are used throughout the section. We again use the permutation matrix $P_{p,q}$ defined in \eqref{def:Ppq}, Section \ref{sec:F16}. For $\fact{20}$, the following block permutation map $A\mapsto P_{p,q} A P_{p,q}^T$ is used:
\begin{gather}
    \nonumber
    {\overbrace{\hspace{0.7cm}}^p\overbrace{\hspace{0.7cm}}^q\overbrace{\hspace{0.7cm}}^{p}\overbrace{\hspace{0.7cm}}^{q}}
    \hspace{1.7cm}
    {\overbrace{\hspace{0.7cm}}^p\overbrace{\hspace{0.7cm}}^p\overbrace{\hspace{0.7cm}}^{q}\overbrace{\hspace{0.7cm}}^{q}}\hspace{0.2cm}
    \\[-4pt]
    \begin{bmatrix}
    X_1 & X_2 & W_1 & W_2\\
    X_3 & X_4 & W_3 & W_4 \\
    Z_1 & Z_2 & Y_1 & Y_2 \\
    Z_3 & Z_4 & Y_3 & Y_4
    \end{bmatrix} 
    \hspace{-0.2cm}
    \begin{array}{l}
    \rdelim\}{1}{0mm}[\scriptsize{$p$}] \\
    \rdelim\}{1}{0mm}[\scriptsize{$q$}] \\ 
    \rdelim\}{1}{0mm}[\scriptsize{$p$}] \\
    \rdelim\}{1}{0mm}[\scriptsize{$q$}] 
    \end{array}
    \hspace{0.3cm}\longmapsto \begin{bmatrix}
     X_1 & W_1 & X_2 & W_2 \\
    Z_1 & Y_1 & Z_2 & Y_2 \\
    X_3 & W_3 & X_4 & W_4 \\
    Z_3 & Y_3 & Z_4 & Y_4 
    \end{bmatrix}
    \hspace{-0.2cm}
    \begin{array}{l}
    \rdelim\}{1}{0mm}[\scriptsize{$p$}] \\
    \rdelim\}{1}{0mm}[\scriptsize{$p$}] \\ 
    \rdelim\}{1}{0mm}[\scriptsize{$q$}] \\
    \rdelim\}{1}{0mm}[\scriptsize{$q$}] 
    \end{array} \label{eq:F20perm}
\end{gather}

The factorization $\fact{20}$ of indefinite orthogonal matrices is as follows.
\begin{theorem}[$\fact{20}$, real]\label{thm:F20real}
Let $f:A\mapsto P_{p,q} AP_{p,q}^T$ be the map \eqref{eq:F20perm}. For any $2n\times 2n$ indefinite orthogonal matrix $O\in\ortho{2p, 2q}$, there exist
\begin{itemize}
    \item Two orthogonal matrices $O_{2p}\in\ortho{2p}$, $O_{2q}\in\ortho{2q}$,
    \item An $n\times n$ indefinite unitary matrix $V\in\un{p,q}$, 
    \item $q$ unique (up to order and signs) real numbers $\theta_1, \dots, \theta_q$,
\end{itemize}
such that the following factorization holds:
\begin{equation}
    O = \twotwo{O_{2p}}{}{}{O_{2q}} \twotwo{I_{p-q}}{}{}{H_{p+q, 2q}^\delta} f(\realify{V}),
\end{equation}
where $\delta = (\theta_1, \dots, \theta_q, -\theta_1, \dots, -\theta_q) \in\mathbb{R}^{2q}$. 
\end{theorem}

Note that $\realify{V}$ in Theorem \ref{thm:F20real} is a $2n\times 2n$ matrix that satisfies
\begin{equation*}
    \realify{V}^T\Phi \realify{V} = \Phi, \hspace{0.5cm} \text{where }\Phi = \diag(I_p, -I_q, I_p, -I_q).
\end{equation*}
The tridiagonal matrix $\stwotwo{I_{p-q}}{}{}{H_{p+q, 2q}^\delta}$ has the following block structure:
\begin{equation}
     \twotwo{I_{p-q}}{}{}{H_{p+q, 2q}^\delta} = \begin{bsmallmatrix}
     I_{p-q} & & & 0 & & \\ 
     & Ch & & & Sh & \\
     & & Ch & & & -Sh\\
     0& & & I_{p-q} & & \\
     & Sh & & & Ch & \\
     & & -Sh & & & Ch 
     \end{bsmallmatrix},
\end{equation}
where $Ch, Sh$ are $\cosh$, $\sinh$ matrices of $\theta_1, \cdots, \theta_q$. The complex $\fact{20}$ follows similarly.
\begin{theorem}[$\fact{20}$, complex]\label{thm:F20complex}
Let $f:A\mapsto P_{p,q} AP_{p,q}^T$ be the map \eqref{eq:F20perm}. For any $2n\times 2n$ complex indefinite unitary matrix $U\in\un{2p, 2q}$, there exist 
\begin{itemize}
    \item Two unitary matrices $U_{2p}\in\un{2p}$, $U_{2q}\in\un{2q}$,
    \item An $n\times n$ quaternionic indefinite unitary matrix $V\in\un{p, q, \mathbb{H}}$,
    \item $q$ unique (up to order and signs) real numbers $\theta_1, \dots, \theta_q$,
\end{itemize}
such that the following factorization holds:
\begin{equation}
    U = \twotwo{U_{2p}}{}{}{U_{2q}} \twotwo{I_{p-q}}{}{}{H_{p+q, 2q}^\delta} f(\complexify{V}),
\end{equation}
where $\delta = (\theta_1, \dots, \theta_q, -\theta_1, \dots, \theta_q)\in\mathbb{R}^{2q}$.
\end{theorem}
Note that a $2n\times 2n$ matrix $\complexify{V}$ in Theorem \ref{thm:F20complex} satisfies the relationship 
\begin{equation*}
    \complexify{V}^H\Phi \complexify{V} = \Phi, \hspace{0.5cm} \text{where }\Phi = \diag(I_p, -I_q, I_p, -I_q).
\end{equation*}

\subsection{${\fact{21}}$ : ${(\Un_\beta(n, n), \Un_\beta(n)\times \Un_\beta(n), \Ortho_{2\beta}(n))}$}

\begin{gather*}
\boxed{
{\renewcommand{\arraystretch}{1.25}
\begin{bmatrix}
2n\times 2n \\ \hspace{0.2cm}\text{$\beta$-indefinite} \hspace{0.2cm}\, \\ \text{unitary} \\ (n, n)
\end{bmatrix}}
= 
\begin{bmatrix}
\colorbox[HTML]{D6D6D6}{\makebox(28,28){\makecell{$n\times n$ \\ \small{Unitary}}}} & \\ 
& \hspace{-0.35cm}\colorbox[HTML]{D6D6D6}{\makebox(28,28){\makecell{$n\times n$ \\ \small{Unitary}}}}
\end{bmatrix}
\cdot
\small{\begin{bmatrix}
    \makebox(30,30){\makecell{$\diagdown$ \hspace{0.7cm}\,\\$\cosh\theta_l$\\ \hspace{0.7cm}$\diagdown$}} & 
    \hspace{-0.2cm}\makebox(30,30){\makecell{$\diagdown$ \hspace{0.7cm}\,\\$\sinh\theta_l$\\ \hspace{0.7cm}$\diagdown$}} \\
    \makebox(30,30){\makecell{$\diagdown$ \hspace{0.7cm}\,\\$\sinh\theta_l$\\ \hspace{0.7cm}$\diagdown$}} & 
    \hspace{-0.2cm}\makebox(30,30){\makecell{$\diagdown$ \hspace{0.7cm}\,\\$\cosh\theta_l$\\ \hspace{0.7cm}$\diagdown$}} 
\end{bmatrix}}
\cdot
\overbrace{\begin{bmatrix}
\colorbox{gray!25}{\makebox(27,27){$\,$}}
& \hspace{-0.35cm}\colorbox{gray!50}{\makebox(27,27){$\,$}}\\ 
\colorbox{gray!50}{\makebox(27,27){$\,$}} 
& \hspace{-0.35cm}\colorbox{gray!25}{\makebox(27,27){$\,$}}
\end{bmatrix}}^{\text{realified/complexified}}
\hspace{-2.7cm}{\renewcommand{\arraystretch}{1.25}
\begin{matrix}
 n\times n \\ \hspace{0.3cm}\text{$2\beta$-orthogonal} \hspace{0.3cm}\,
\end{matrix}}
}
\end{gather*}

\par The factorization $\fact{21}$ is a factorization of $2n\times 2n$ indefinite orthogonal/unitary matrices with the signature $(n, n)$. The realify and complexify maps $\realify{\,\,\cdot\,\,}$ and $\complexify{\,\,\cdot\,\,}$ (see \eqref{eq:realify}, \eqref{eq:complexify}) are used on the $n\times n$ $2\beta$-orthogonal matrices.

\begin{theorem}[$\fact{21}$, real]\label{thm:F21real}
For any $2n\times 2n$ indefinite orthogonal matrix $O\in\ortho{n, n}$, there exist
\begin{itemize}
    \item Two $n\times n$ orthogonal matrices $O_1, O_2\in\ortho{n}$,
    \item An $n\times n$ complex orthogonal matrix $V$, 
    \item $n$ unique (up to order and signs) real numbers $\theta_1, \dots, \theta_n$,
\end{itemize}
such that the following factorization holds:
\begin{equation}
    O = \twotwo{O_1}{}{}{O_2}\twotwo{Ch}{Sh}{Sh}{Ch} \realify{V}.
\end{equation}
\end{theorem}

Since the group $\ortho{n, n}$ is isomporphic to the real perplectic group, we also obtain a matrix factorization of real perplectic matrices. Under the isomorphism that sends $\ortho{n, n}$ to the real perplectic group (see Remark \ref{rem:realperplecticsvd}), the matrix $V\in\ortho{n, \mathbb{C}}$ becomes a realified complex perplectic matrix. Choose a subgroup $A$ (recall that the choice of $\mathfrak{a}$ is not unique) such that the elements of $A$ are the matrices with the diagonal ($\cosh$) and antidiagonal ($\sinh$) values being nonzero. See \eqref{eq:alternativeforms} for such an example. Then, an isomorphic form of $\fact{21}$ follows as a matrix factorization of a real perplectic matrix. 

\begin{corollary}[$\fact{21}$, real perplectic]\label{cor:F21real}
For any $2n\times 2n$ real perplectic matrix $P$, there exist
\begin{itemize}
    \item A $2n\times 2n$ real perplectic orthogonal matrix $O$,
    \item An $n\times n$ complex perplectic matrix $V$,
    \item $n$ unique (up to order and signs) real numbers $\theta_1, \dots, \theta_n$,
\end{itemize}
such that the following factorization holds:
\begin{equation}
    P = O \twotwo{\Sigma}{}{}{\Sigma^{-1}} \realify{V}.
\end{equation}
\end{corollary}

\par The complex $\fact{21}$ is the following. 
\begin{theorem}[$\fact{21}$, complex]\label{thm:F21complex}
For any $2n\times 2n$ indefinite unitary matrix $U\in\un{n, n}$, there exist
\begin{itemize}
    \item Two $n\times n$ unitary matrices $U_1, U_2\in\un{n}$,
    \item An $n\times n$ quaternionic orthogonal matrix $V$ such that $V^{D_i}V = I_n$,
    \item $n$ unique (up to order and signs) real numbers $\theta_1, \dots, \theta_n$,
\end{itemize}
such that the following factorization holds:
\begin{equation}
    U = \twotwo{U_1}{}{}{U_2} \twotwo{Ch}{Sh}{Sh}{Ch} \complexify{V}.
\end{equation}
\end{theorem}

Similar to Corollary \ref{cor:F21real}, the isomorphism between $\un{n, n}$ and $\Symp^*(2n, \mathbb{C})$ obtains another matrix factorization. Under the map which sends $\un{n, n}$ to $\Symp^*(2n, \mathbb{C})$, the matrix $V$ in Theorem \ref{thm:F21complex} becomes a complex orthogonal and conjugate symplectic matrix. We omit the detail.

\subsection{${\fact{22}}$ : ${(\Un_\beta(p, q), \Un_\beta(p)\times \Un_\beta(q), \Un_{\beta/2}(p, q))}$}

\begin{gather*}
\boxed{
{\renewcommand{\arraystretch}{1.25}
\begin{bmatrix}
n\times n \\ \hspace{0.2cm}\text{$\beta$-indefinite} \hspace{0.2cm}\, \\ \text{unitary} \\ (p, q)
\end{bmatrix}}
= 
\begin{bmatrix}
\colorbox[HTML]{D6D6D6}{\makebox(35,35){\makecell{$p\times p$\\\small{Unitary}}}} & \\ 
& \hspace{-0.35cm}\colorbox[HTML]{D6D6D6}{\makebox(18,18){\makecell{\small{$q\times q$}\\\tiny{Unitary}}}}
\end{bmatrix}
\cdot
\begin{bmatrix}
\begin{array}{c:c}
\makebox(33, 38){$\begin{smallmatrix}\diagdown\hspace{1.1cm}\\ \text{ch}_l \hspace{0.7cm}\\ \diagdown \hspace{0.1cm} \\ \hspace{0.7cm}\rule{0cm}{0.4cm}\text{\large{$I_{p-q}$}}\end{smallmatrix}$}
& \makebox(17, 38){$\begin{smallmatrix}\diagdown\hspace{0.6cm}\\ i\text{sh}_l\\ \hspace{0.6cm}\diagdown  \\ \rule{0cm}{0.5cm}\end{smallmatrix}$}\\ \hdashline
\makebox(38,17){$\begin{smallmatrix}\diagdown\hspace{1.2cm}\\ -i\text{sh}_l\hspace{0.8cm}\\ \diagdown\hspace{0.1cm}\end{smallmatrix}$}
& \makebox(17, 20){$\begin{smallmatrix}\diagdown\hspace{0.7cm}\\ \text{ch}_l\\ \hspace{0.5cm}\diagdown\end{smallmatrix}$}
\end{array}
\end{bmatrix}
\cdot
{\renewcommand{\arraystretch}{1.25}
\begin{bmatrix}
n\times n \\ \hspace{0.2cm}\text{$\frac{\beta}{2}$-indefinite} \hspace{0.2cm}\, \\ \text{unitary} \\ (p, q)
\end{bmatrix}}
}
\end{gather*}

The factorization $\fact{22}$ is a factorization of complex/quaternionic indefinite unitary matrices. The relationship $\ortho{p, q}\subset \un{p, q}\subset\un{p, q, \mathbb{H}}$ is crucial. 
\begin{theorem}[$\fact{22}$, complex]\label{thm:F22complex}
For any $n\times n$ complex indefinite unitary matrix $U\in\un{p, q}$, there exist
\begin{itemize}
    \item Two unitary matrices $U_p\in\un{p}$ and $U_q\in\un{q}$,
    \item An $n\times n$ indefinite orthogonal matrix $O\in\ortho{p, q}$,
    \item $q$ unique (up to order and signs) real numbers $\theta_1, \dots, \theta_q$,
\end{itemize}
such that the following factorization holds:
\begin{equation}
    U = \begin{bmatrix}
    \begin{array}{c:c}
        U_p &  \\ \hdashline
         & U_q
    \end{array}
    \end{bmatrix}
    \begin{bmatrix}
    \begin{array}{rc:c}
        Ch& & iSh \\
         & I_{p-q} & \\ \hdashline
         -iSh & & Ch
    \end{array}
    \end{bmatrix}
    O.
\end{equation}
$Ch, Sh$ are $q\times q$ $\cosh, \sinh$ diagonal matrices as usual.
\end{theorem}

The quaternionic version of $\fact{22}$ follows similarly.
\begin{theorem}[$\fact{22}$, quaternion]\label{thm:F22quaternion}
Fix $\eta\in\{j, k\}$. For any $n\times n$ quaternionic indefinite unitary matrix $U\in\un{p, q, \mathbb{H}}$, there exist
\begin{itemize}
    \item Two quaternionic unitary matrices $U_p\in\un{p, \mathbb{H}}$ and $U_q\in\un{q, \mathbb{H}}$,
    \item An $n\times n$ complex indefinite unitary matrix $V\in\un{p, q}$,
    \item $q$ unique (up to order and signs) real numbers $\theta_1, \dots, \theta_q$,
\end{itemize}
such that the following factorization holds:
\begin{equation}
     U = \begin{bmatrix}
    \begin{array}{c:c}
        U_p &  \\ \hdashline
         & U_q
    \end{array}
    \end{bmatrix}
    \begin{bmatrix}
    \begin{array}{rc:c}
        Ch& & \eta Sh \\
         & I_{p-q} & \\ \hdashline
         -\eta Sh & & Ch
    \end{array}
    \end{bmatrix}
    V.
\end{equation}
\end{theorem}

\section{Matrix factorization of complex/quaternionic orthogonal matrices}\label{sec:orthofact}
In this section we present the matrix factorizations $\fact{23}$, $\fact{24}$, $\fact{25}$, arising from the generalized Cartan decomposition of the orthogonal groups $\ortho{n, \mathbb{C}}$ and $\ortho{n, \mathbb{H}}$. The generalized Cartan triple $(G, K_\sigma, K_\tau)$ for each factorization is the following. 
\begin{alignat*}{2}
    \fact{23}&\text{ : }\big(\Ortho_\beta(n), \Un_{\beta/2}(n), \Un_{\beta/2}(n) \big) \,\,(\text{KAK}) \hspace{3.5cm} &&\beta= 2, 4 \\
    \fact{24}&\text{ : } \big(\Ortho_\beta(n), \Un_{\beta/2}(n), \Ortho_\beta(p)\times\Ortho_\beta(q)\big) &&\beta = 2, 4\\
    \fact{25}&\text{ : } \big(\Ortho_\beta(2n), \Ortho_{\beta/2}(2n), \Ortho_{2\beta}(n)\big) &&\beta = 2\\
\end{alignat*}

\subsection{$\fact{23}$ : $(\Ortho_\beta(n), \Un_{\beta/2}(n), \Un_{\beta/2}(n))$}\label{sec:F23}

\begin{gather*}
\boxed{
\begin{bmatrix}
\, \\ n\times n \\ \small{\beta\text{-orthogonal}} \\ \,
\end{bmatrix}
= 
\begin{bmatrix}
\, \\ n\times n \\ \hspace{0.2cm}\small{\frac{\beta}{2}\text{-unitary}}\hspace{0.2cm}\, \\ \,
\end{bmatrix}
\cdot
{\renewcommand{\arraystretch}{0.8}
\begin{bmatrix}
\text{\Large{$\diagdown$}} \hspace{1cm}\,\vspace{0.1cm} \\
\hspace{0.3cm}\text{\tiny{$\text{ch}_l$ $i\text{sh}_l$}} \\
\text{\tiny{$-i\text{sh}_l$ $\text{ch}_l$}}\\
\hspace{1.4cm}\,\vspace{0.1cm} \text{\Large{$\diagdown$}} \hspace{0.1cm}\\
\end{bmatrix}}
\cdot
\begin{bmatrix}
\, \\ n\times n \\ \hspace{0.2cm}\small{\frac{\beta}{2}\text{-unitary}}\hspace{0.2cm}\, \\ \,
\end{bmatrix}
}
\end{gather*}

The factorization $\fact{23}$ is a matrix factorization of complex and quaternionic orthogonal matrices arising from the KAK decomposition of $\ortho{n, \mathbb{C}}$ and $\ortho{n, \mathbb{H}}$. 
\begin{theorem}[$\fact{23}$, complex]\label{thm:F23complex}
Let $m=\lfloor \frac{n}{2} \rfloor$. For any $n\times n$ complex orthogonal matrix $O\in\ortho{n, \mathbb{C}}$, there exist
\begin{itemize}
    \item Two $n\times n$ real orthogonal matrices $O_1, O_2\in\ortho{n}$,
    \item $m$ unique (up to order and signs) real numbers $\theta_1, \dots, \theta_m$,
\end{itemize}
such that the following factorization holds: 
\begin{equation}\label{eq:F23complex}
    O = O_1BO_2,
\end{equation}
where $B=B_n^{i, \theta}$ is defined as an $n\times n$ block diagonal matrix $\diag(B_1^i, \dots, B_m^i)$ ($n$ even) or $\diag(1, B_1^i, \dots, B_m^i)$ ($n$ odd), defined in Table \ref{tab:auxmatrix}. 
\end{theorem}

As we discuss at the end of Section \ref{sec:gencartanex2}, the choice of $\diag(B_1^i, \dots, B_m^i)$ is not unique. Another choice of $A$ can substitute \eqref{eq:F23complex} with the following (when $n$ even). 
\begin{equation}\label{eq:permutedF23complex}
    O = O_1 \twotworr{Ch}{iSh}{-iSh}{Ch} O_2.
\end{equation}

\par An important choice of the subgroup $A$ is the collection of the following $a$ matrices.
\begin{equation}\label{eq:complexperpsvdA}
    \arraycolsep=1pt
    a = \begin{bmatrix}
    \cosh\theta_1 & & & & & i\sinh\theta_1 \\
    & \ddots & & & \reflectbox{$\ddots$} & \\
    & & \cosh\theta_m & i\sinh\theta_m & & \\
    & & -i\sinh\theta_m & \cosh\theta & & \\
    & \reflectbox{$\ddots$} & & & \ddots & \\
    -i\sinh\theta_1 & & & & & \cosh\theta_1
    \end{bmatrix}.
\end{equation}
This choice leads to the complex perplectic SVD.

\begin{remark}[complex perplectic group and $\ortho{n, \mathbb{C}}$]\label{rem:complexperplecticsvd}
Recall the complex perplectic group $\{P\in\mathbb{C}^{n\times n}|P^TE_nP=E_n\}$. Define a matrix $V= \frac{1+i}{2}I_n + \frac{1-i}{2}E_n$. Then, $V$ satistfies $V^HV = I_n$ and $V^TV = V^2 = E_n$. (This is a simple modification of the Takagi factorization of $E_n$.) Imagine the isomorphism $A\mapsto VAV^H$, which sends the complex orthogonal group onto the complex perplectic group since
\begin{equation*}
    O^TO=I_n \hspace{0.5cm}\text{ implies }\hspace{0.5cm} (VOV^H)^TE_n(VOV^H) = E_n.
\end{equation*}
We can apply this isomorphism to obtain the SVD of a complex perplectic matrix, with the perplectic structure preserved. (See Remark \ref{rem:realperplecticsvd} for the real perplectic structured SVD.) The matrix $a$ in \eqref{eq:complexperpsvdA} is transformed to a diagonal matrix by the above isomorphism. Applying the map $A\mapsto VAV^H$ for all matrices in $O = O_1aO_2$ with \eqref{eq:complexperpsvdA}, we obtain the following structured SVD. 
\end{remark}

\begin{corollary}[perplectic structured SVD, complex]
For an $n\times n$ complex perplectic matrix $P$, there exist
\begin{itemize}
    \item Two $n\times n$ perplectic unitary matrices $U_1, U_2$
    \item A diagonal matrix $\Sigma = \diag(\sigma_1, \cdots, \sigma_m, \frac{1}{\sigma_m}, \cdots, \frac{1}{\sigma_1})$,
\end{itemize}
such that $P = U_1 \Sigma U_2$ is the SVD of $P$ with all three factors being perplectic. (For an odd $n$, we put $1$ between $\sigma_m$ and $\frac{1}{\sigma_m}$ in $\Sigma$.)
\end{corollary}

The quaternionic version of $\fact{23}$ is the following. 
\begin{theorem}[$\fact{23}$, quaternion]\label{thm:F23quaternion}
Fix $\eta\in\{j, k\}$. Let $m=\lfloor \frac{n}{2} \rfloor$. For any $n\times n$ quaternionic orthogonal matrix $O\in\ortho{n, \mathbb{H}}$ such that $O^{D_i}O = I_n$ there exist
\begin{itemize}
    \item Two $n\times n$ complex unitary matrices $U_1, U_2\in\un{n}$,
    \item $m$ unique (up to order and signs) real numbers $\theta_1, \dots, \theta_m$,
\end{itemize}
such that the following factorization holds: 
\begin{equation}
    O = U_1B U_2,
\end{equation}
where $B = B_n^{\eta, \theta}$ is an $n\times n$ block diagonal matrix defined in Table \ref{tab:auxmatrix}. 
\end{theorem}

\par Theorem \ref{thm:F23quaternion} has other types of isomorphic forms. Obviously, switching $i, \eta$ is one such isomorphism. In this case, Theorem \ref{thm:F23quaternion} can be applied to $O\in\Ortho_j(n, \mathbb{H})$ such that $O^{D_j}O=I$ and the factors $U_1, U_2$ become unitary matrices with imaginary unit $j$. One can use the imaginary unit $k$ to obtain another isomorphic form of quaternion $F_{23}$. Other isomorphic forms are also obtained by the similar choice of $A$ as in \eqref{eq:permutedF23complex} and \eqref{eq:complexperpsvdA}.

\subsection{$\fact{24}$ : $(\Ortho_\beta(n), \Un_{\beta/2}(n), \Ortho_\beta(p)\times\Ortho_\beta(q))$}

\begin{gather*}
\boxed{
{\renewcommand{\arraystretch}{1.35}
\begin{bmatrix}
\rule{0cm}{0.1cm}\\ n\times n \\ \hspace{0.1cm} \text{$\beta$-orthogonal} \hspace{0.1cm}\, \\ \rule{0cm}{0.1cm}
\end{bmatrix}}
= 
{\renewcommand{\arraystretch}{1.35}
\begin{bmatrix}
\rule{0cm}{0.1cm}\\ n\times n \\ \hspace{0.2cm} \text{{ $\frac{\beta}{2}$-unitary }} \hspace{0.2cm}\, \\ \rule{0cm}{0.1cm}
\end{bmatrix}}
\cdot
\begin{bmatrix}
\begin{array}{c:c}
\makebox(33, 38){$\begin{smallmatrix}\diagdown\hspace{1.1cm}\\ \text{ch}_l \hspace{0.7cm}\\ \diagdown \hspace{0.1cm} \\ \hspace{0.7cm}\rule{0cm}{0.4cm}\text{\large{$I_{p-q}$}}\end{smallmatrix}$}& 
\makebox(17, 38){$\begin{smallmatrix}\diagdown\hspace{0.6cm}\\ i\text{sh}_l\\ \hspace{0.6cm}\diagdown  \\ \rule{0cm}{0.5cm}\end{smallmatrix}$}\\ \hdashline
\makebox(38,17){$\begin{smallmatrix}\diagdown\hspace{1.2cm}\\ -i\text{sh}_l\hspace{0.7cm}\\ \diagdown\hspace{0.1cm}\end{smallmatrix}$}
& \makebox(17, 20){$\begin{smallmatrix}\diagdown\hspace{0.7cm}\\ \text{ch}_l\\ \hspace{0.5cm}\diagdown\end{smallmatrix}$}
\end{array}
\end{bmatrix}
\cdot
\begin{bmatrix}
\colorbox[HTML]{D6D6D6}{\makebox(35,35){\makecell{$p\times p$\\\small{$\beta$-ortho.}}}} & \\ 
& \hspace{-0.35cm}\colorbox[HTML]{D6D6D6}{\makebox(18,18){\makecell{\small{$q\times q$}\\\tiny{$\beta$-ortho.}}}}
\end{bmatrix}
}
\end{gather*}

Let $n = p+q$, $p\ge q$. The factorization $\fact{24}$ is a factorization of complex and quaternionic orthogonal matrices. 
\begin{theorem}[$\fact{24}$, complex]\label{thm:F24complex}
For any $n\times n$ complex orthogonal matrix $O$, there exist,
\begin{itemize}
    \item An $n\times n$ orthogonal matrix $U\in\ortho{n}$, 
    \item Two complex orthogonal matrices $V_p\in\ortho{p, \mathbb{C}}$ and $V_q\in\ortho{q, \mathbb{C}}$,
    \item $q$ unique (up to order and signs) real numbers $\theta_1, \dots, \theta_q$,
\end{itemize}
such that the following factorization holds:
\begin{equation}
    O = U \begin{bmatrix}
    \begin{array}{rc:c}
        Ch& & iSh \\
         & I_{p-q} & \\ \hdashline
         -iSh & & Ch
    \end{array}
    \end{bmatrix}
    \begin{bmatrix}
    \begin{array}{c:c}
        V_p &  \\ \hdashline
         & V_q
    \end{array}
    \end{bmatrix}.
\end{equation}
\end{theorem}

Similarly the quaternionic version follows.
\begin{theorem}[$\fact{24}$, quaternion]\label{thm:F24quaternion}
Fix $\eta\in\{j, k\}$. For any $n\times n$ quaternionic orthogonal matrix $O$ such that $O^{D_i}O=I$, there exist,
\begin{itemize}
    \item An $n\times n$ complex unitary matrix $U\in\un{n}$,
    \item Two quaternionic orthogonal matrices $V_p\in\ortho{p, \mathbb{H}}$, $V_q\in\ortho{q, \mathbb{H}}$,
    \item $q$ unique (up to order and signs) real numbers $\theta_1, \dots, \theta_q$,
\end{itemize}
such that the following factorization holds:
\begin{equation}
    O = U \begin{bmatrix}
    \begin{array}{rc:c}
        Ch& & \eta Sh \\
         & I_{p-q} & \\ \hdashline
         -\eta Sh & & Ch
    \end{array}
    \end{bmatrix}
    \begin{bmatrix}
    \begin{array}{c:c}
        V_p &  \\ \hdashline
         & V_q
    \end{array}
    \end{bmatrix}.
\end{equation}
\end{theorem}

As always the isomorphism between $i, j, k$ is valid for quaternionic $\fact{24}$. 

\subsection{$\fact{25}$ : $(\Ortho_\beta(2n), \Ortho_{\beta/2}(2n), \Ortho_{2\beta}(n))$}
\begin{gather*}
\boxed{
{\renewcommand{\arraystretch}{1.5}
\begin{bmatrix}
2n\times 2n \\\text{Complex} \\ \text{ Orthogonal }
\end{bmatrix}}
= 
{\renewcommand{\arraystretch}{1.5}
\begin{bmatrix}
2n\times 2n \\\text{Real} \\ \text{ Orthogonal }
\end{bmatrix}}
\cdot
{\small{\begin{bmatrix}
    \makebox(30,30){\makecell{$\diagdown$ \hspace{0.7cm}\,\\$\cosh\theta_l$\\ \hspace{0.7cm}$\diagdown$}} & 
    \hspace{-0.2cm}\makebox(30,30){\makecell{$\diagdown$ \hspace{0.7cm}\,\\$i\sinh\theta_l$\\ \hspace{0.7cm}$\diagdown$}} \\
    \makebox(30,30){\makecell{$\diagdown$ \hspace{0.7cm}\,\\$-i\sinh\theta_l$\\ \hspace{0.7cm}$\diagdown$}} & 
    \hspace{-0.2cm}\makebox(30,30){\makecell{$\diagdown$ \hspace{0.7cm}\,\\\,\,$\cosh\theta_l$\\ \hspace{0.7cm}$\diagdown$}} 
\end{bmatrix}}}
\cdot
\overbrace{\small\begin{bmatrix}
\colorbox{gray!25}{\makebox(27,27){$\,$}}
& \hspace{-0.35cm}\colorbox{gray!50}{\makebox(27,27){$\,$}}\\ 
\colorbox{gray!50}{\makebox(27,27){$\,$}} 
& \hspace{-0.35cm} \colorbox{gray!25}{\makebox(27,27){$\,$}}
\end{bmatrix}}^{\text{complexified}}
\hspace{-2.7cm}{\renewcommand{\arraystretch}{1.25}
\begin{matrix}
 n\times n \\ \hspace{0.3cm}\text{Quaternionic}\hspace{0.3cm}\, \\ \text{Orthogonal}
\end{matrix}}
}
\end{gather*}

The factorization $\fact{25}$ is a matrix factorization of complex orthogonal matrices. The fact that the complexified quaternionic orthogonal group $\complexify{\Ortho_j(n, \mathbb{H})}$ is a subset of $\ortho{2n, \mathbb{C}}$ is used. 
\begin{theorem}[$\fact{25}$, complex]\label{thm:F25complex}
For any $2n\times 2n$ complex orthogonal matrix $O$, there exist
\begin{itemize}
    \item A $2n\times 2n$ real orthogonal matrix $U$, 
    \item A $n\times n$ quaternionic orthogonal matrix $V$ such that $V^{D_j}V = I_n$, 
    \item $n$ unique (up to order and signs) real numbers $\theta_1, \dots, \theta_n$,
\end{itemize}
such that the following factorization holds:
\begin{equation}
    O = U\twotworr{Ch}{iSh}{-iSh}{Ch} \complexify{V}.
\end{equation}
\end{theorem}

\addtocontents{toc}{\protect\setcounter{tocdepth}{1}}

\vspace{1.5cm}
\section{2020 Mathematical Hindsight (in 2021):\\ A lesson in mathematical history}\label{sec:lesson}
\begin{figure}[h]
    \centering
    \frame{\includegraphics[width=4.9in]{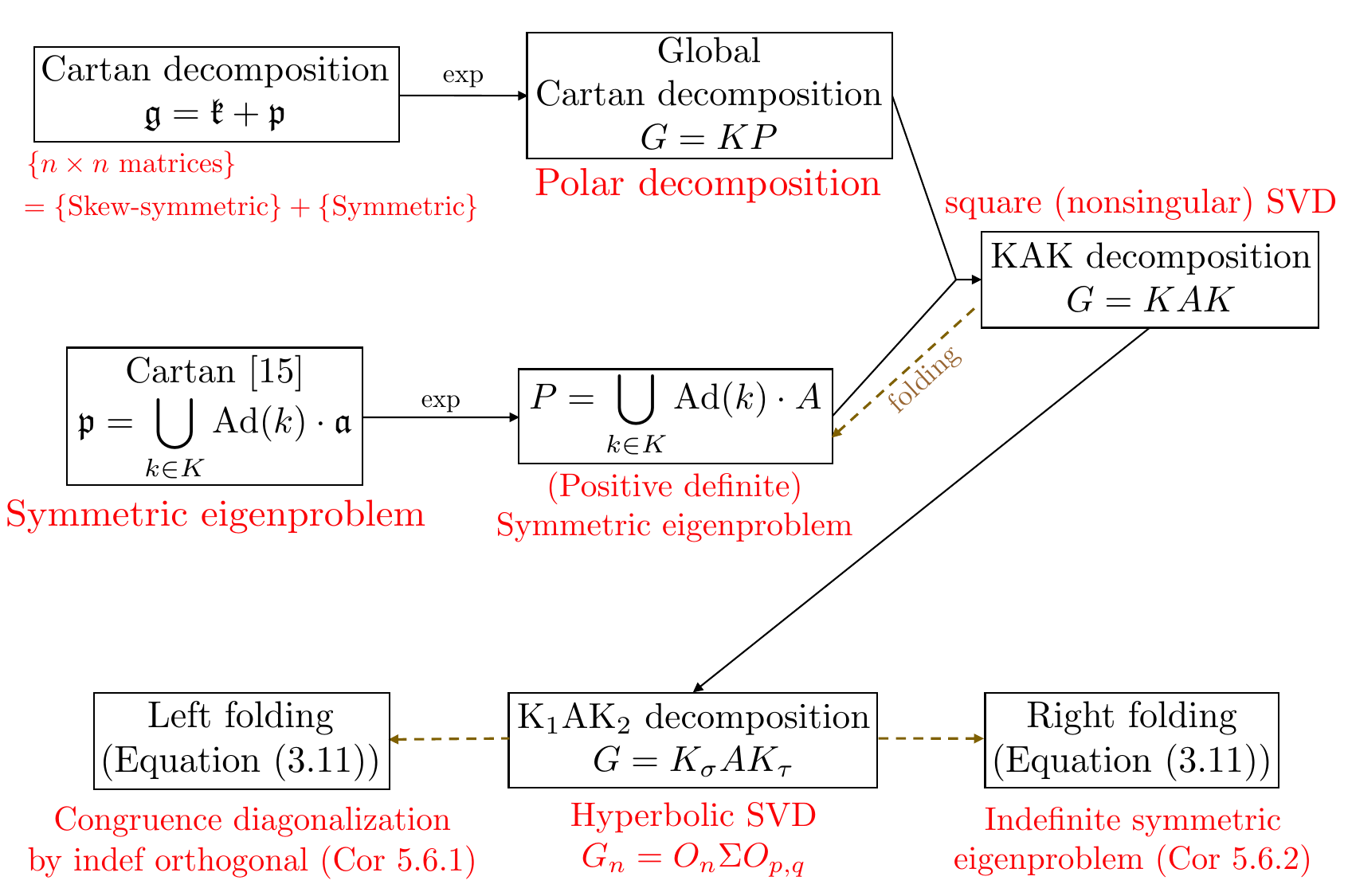}}
    \caption{The diagram of the abstractions appearing in this paper. Each abstraction amounts to a class of matrix factorizations. An example is written in {\color{red}red}. The connections between the abstractions makes this framework interesting. One could  make up any factorization with $n^2$ or $mn$ parameters, and solve the nonlinear systems, but such a factorization would not enjoy the properties that these factorizations have.}
    \label{fig:diagram}
\end{figure}

In this section we will provide a brief timeline of matrix factorizations in so much as they relate to Cartan's 1927 work \cite{cartan1927certaines} published at the age of 58. Timelines of this sort can be difficult, because there are so many different aspects that can be emphasized. We wanted to trace in the context of the KAK decomposition and the $\kak$ decomposition (summarized as in Figure \ref{fig:diagram}) how a beautiful idea like a matrix factorization could be lost in insignificance until its importance is recognized and matured. We uncovered a fascinating story that we very much wished to share.

There is a kind of joke in applied mathematics that if you discover anything important then possibly Gauss had it first. As you will see from the discussion in this timeline, even if a matrix factorization can be recognized, if it is buried in a proof of a technical lemma, if it uses non-googleable terms, and if specific cases are not worked out and highlighted, we seem to be in something of a ``tree falling in a forest" type of situation.

\subsection{1920s, Cartan}
The famous Cartan decomposition $\mathfrak{g} = \mathfrak{k} + \mathfrak{p}$ may be found in his 1929 work \cite{cartan1929groupes}. (One special case is that all square matrices can be written as an skew-symmetric plus a symmetric.)

\par With today's eyes, we recognize Cartan's work from two years earlier in 1927 \cite{cartan1927certaines} as a blueprint for matrix factorizations. Cartan provided a decomposition of the space $\mathfrak{p}$:
\begin{equation}\label{eq:cartanfrakp}
    \mathfrak{p} = \bigcup_{k\in K}\text{Ad}(k)\mathfrak{a} = \bigcup_{k\in K}k\mathfrak{a}k^{-1},
\end{equation}
which appears in French as the ``adjoint representation" on the 17th page of a 126 page paper (Figure \ref{fig:cartankak}). In equations \eqref{eq:cartanfrakp} and \eqref{eq:svdblockform} below we show how the familiar $2\times 2$ block form of the SVD is a special case of Cartan's decomposition.

\begin{figure}[h]
    \centering
    \includegraphics[width=4.8in]{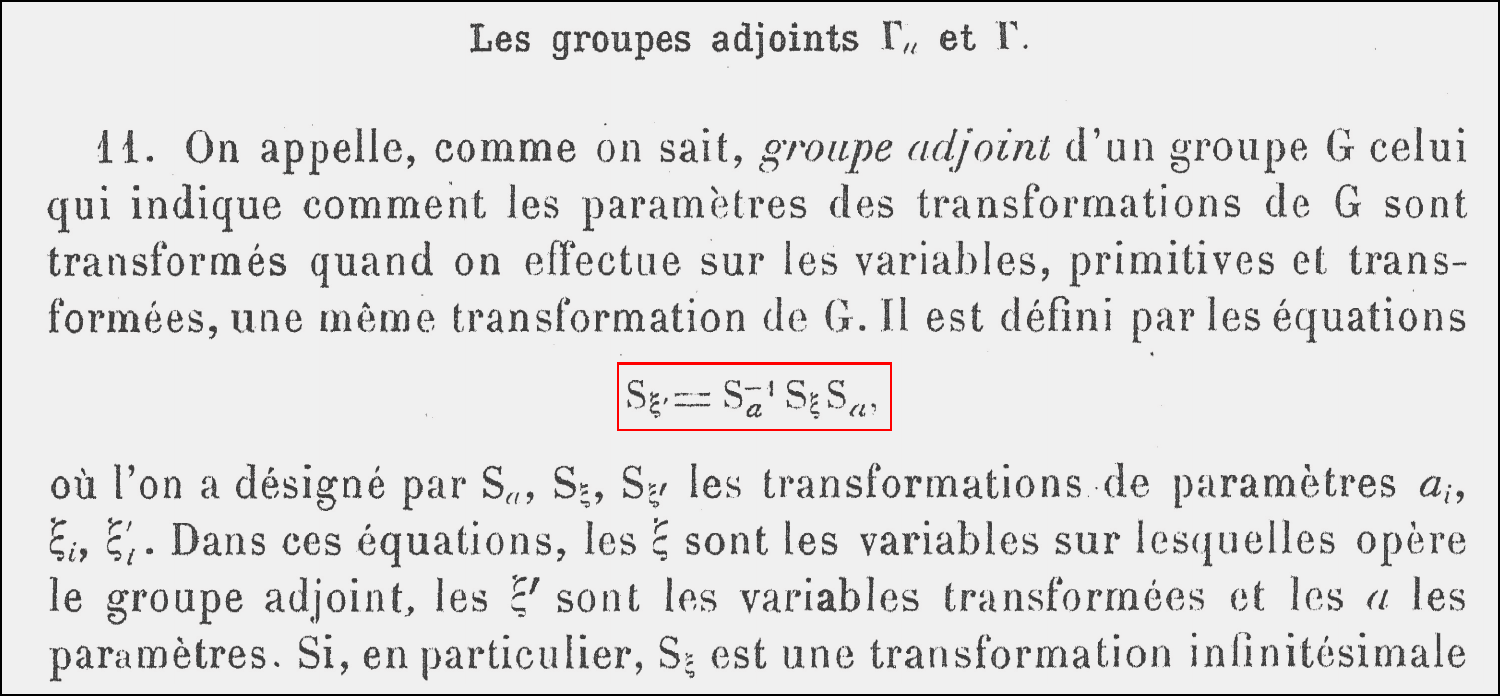}
    \caption{Cartan's adjoint representation is not yet the KAK decomposition, but does include the symmetric eigenproblem and a form of the SVD (equation \eqref{eq:svdblockform}) as special cases. It appears on the 17th page of a 126 page paper \cite{cartan1927certaines} that does not appear to have been translated into English.} 
    \label{fig:cartankak}
\end{figure}

\par Notably, this adjoint representation is not yet exactly the KAK decomposition. It is also not certain if Cartan knew further than \eqref{eq:cartanfrakp}, since the connection between Lie algebras and Lie groups were not complete at the time  according to \cite{helgasonprivatecomm}. At least he did not view the symmetric spaces as quotient spaces in his papers for the same reason. 

\par Crediting Cartan with the KAK decomposition may be an example of \href{https://en.wikipedia.org/wiki/Stigler\%27s_law_of_eponymy}{\color{blue}Stigler's law of eponymy,} though Cartan was certainly very close. (Also see \href{https://en.wikipedia.org/wiki/List_of_misnamed_theorems}{\color{blue}List of misnamed theorems} for famous misnamed theorems.) Another example of Stigler's law in this paper is the Bloch-Messiah decomposition, described in Remark \ref{rem:blochmessiah}. 

\par The equation \eqref{eq:cartanfrakp} does include as special cases the symmetric/Hermitian eigenproblems and also thinly disguised, the nonsquare SVD. To work out the SVD, we point out that a special case of the left hand side of equation \eqref{eq:cartanfrakp} is the tangent space $\mathfrak{p}$ of the symmetric space noncompact BD$\RN{1}$ which is  
\begin{equation}
    \mathfrak{p} = \{\text{All matrices of the block structure }\stwotwo{0}{A}{A^T}{0}, A\in\mathbb{R}^{p\times q}\}.
\end{equation}
The subgroup $K$ is $\ortho{p}\times \ortho{q}$ and the subalgebra $\mathfrak{a}$ is the subset of $\mathfrak{p}$ with the matrices $A$ above being diagonal matrices, say $\Sigma$. The equation \eqref{eq:cartanfrakp} is therefore, in this case, equivalent to $A = O_p \Sigma O_q^T$, a well-known $2\times 2$ block form of the SVD:

\begin{equation}\label{eq:svdblockform}
   \twotwo {0 }{A}{A^T}{0} =
    \twotwo{U}{}{}{V}\twotwo{0}{\Sigma}{\Sigma^T}{0}\twotwo{U}{}{}{V}^T .
\end{equation}

\subsection{1956, Harish-Chandra}
The well-known representation theorist who simply went by the one name Harish-Chandra may be the first to explicitly write down the KAK decomposition on the sixth line of the proof of Lemma 21 of \cite{harish1956representations} on page 590. To emphasize just how buried and unlikely to be found we provide an excerpt:

\begin{figure}[h]
    \centering
    \includegraphics[width=4.8in]{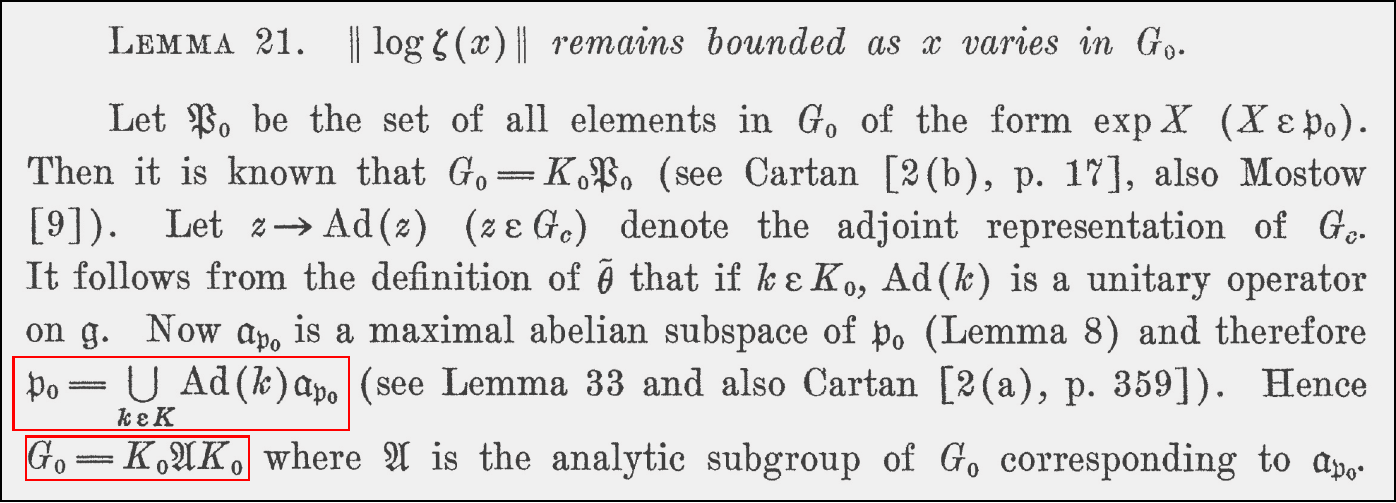}
    \caption{Harish-Chandra buries the KAK decomposition on the 6th line of a proof of what appears to be an unremarkable technical Lemma 21 \cite{harish1956representations}. The square full rank SVD and the hyperbolic CS decomposition are special cases as is the square CS decomposition (compact case).} 
    \label{fig:harishchandrakak}
\end{figure}

\par It seems unlikely that many mathematicians could find such an important result when it appears not as a statement of a theorem, not in a statement of a lemma, but buried deep in the proof of a seemingly technical lemma of little significance. On a modern level, an internet search engine would not be expected to pick up Harish-Chandra's work this way.

\par We feel comfortable stating that Harish-Chandra probably viewed KAK as a trivial consequence of Cartan's decomposition and found its utility as a technical tool but not yet as a deeply important mathematical object of its own right at least not in 1956.

\par The square full rank SVD is a special case of this ``sixth line" stating that $ A = U \Sigma V^T $, when $A$ is nonsingular. The hyperbolic CS decomposition is another special case of this ``sixth line", and we note the general rectangular SVD may be found, for example, in the upper right block entry of the folding of the hyperbolic CS decomposition. (Explicitly appears as Corollary \ref{cor:F18} of this paper.)\footnote{In fact, the upper right block entry of the HCSD, Theorem \ref{thm:F18realcomplexquaternion}, also contains the SVD even before the folding.} The square CS decomposition (Theorem \ref{thm:F4real}, $(p,q) = (r,s)$) is also a special case of this ``sixth line" of the proof of Lemma 21.

\subsection{1962, Helgason}
The KAK decomposition appears once in the 1962 book of Helgason \cite[p.320]{Helgason1962}, covering Harish-Chandra's result (exactly the Lemma 21 of \cite{harish1956representations}). The KAK decomposition continues to be used as a tool in passing seemingly without any further recognition.

\subsection{1968, Wigner}
Wigner wrote a 1968 paper \cite{wigner1968generalization} entitled ``On a generalization of Euler’s angles." It contains explicitly what we now call the CS decomposition (square case), the hyperbolic CS decomposition and even the folding of the HCSD which we have noted (that both the HCSD and the folding) includes the rectangular SVD inside (Theorem \ref{thm:F18realcomplexquaternion}, Corollary \ref{cor:F18}). At the time of writing, Google Scholar indicates this paper was only referenced eight times in the over 50 years since it was published.

\begin{figure}[h]
    \centering
    \includegraphics[width=4.8in]{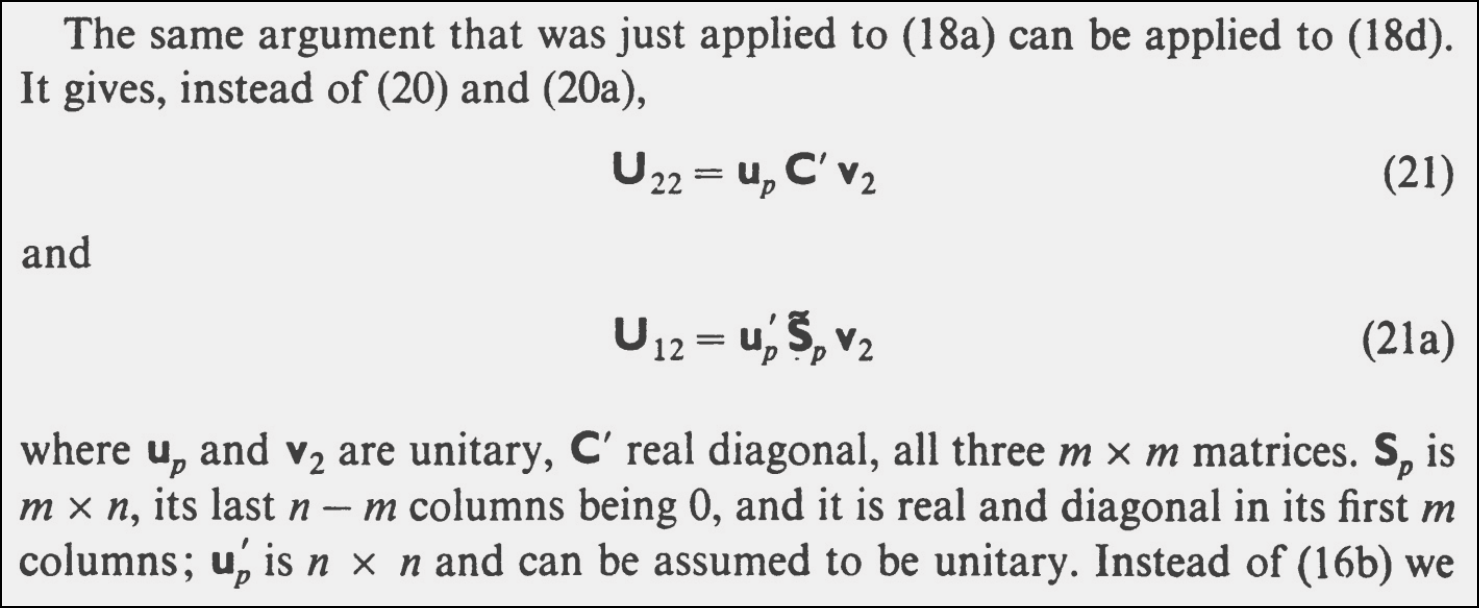}
    \caption{What we today recognize as the  general rectangular SVD is buried in insignificance
    in Wigner's equation (21a) from \cite{wigner1968generalization} which itself is the (1,2) 
     entry of what today we would call a hyperbolic CS decomposition.} 
    \label{fig:wignersvd}
\end{figure}

\par Wigner was absolutely aware of Helgason's 1962 book \cite{Helgason1962} as it is referenced in \cite{wigner1968generalization}.\footnote{In Wigner's case he thought the shape of the CSD and the HCSD reminded him of the Iwasawa decomposition, as mentioned in Section \ref{sec:F18}.} However, he certainly did not recognize his result was already there on page 320 of \cite{Helgason1962}. How could anyone?

\par Not only did Wigner not see his results looking backwards towards \cite{Helgason1962}, very few people recognized Wigner's results going forward. We credit Audrey Terras \cite{terras1988harmonic} for the observation that the SVD is in Wigner's work. If one studies Wigner's paper, the SVD is there, but completely buried in insignificance. (It appears in \cite[Eq.21a]{wigner1968generalization}, which is the upper right block of a hyperbolic CS decomposition. See Figure \ref{fig:wignersvd}.) 

\par In summary, Wigner's results were special cases of results that are in Helgason's book, but Wigner could hardly have been aware given how it appears. Unfortunately, Wigner's results themselves have mostly been overlooked in the over five decades. Why? Because the ideas were not sufficiently mature for Wigner to realize just how important they could be, the title is uninviting, and the motivation was narrow.

\subsection{1968 and 1969, Davis and Kahan}
Davis and Kahan in two papers \cite{davis1969some,davis1970rotation} develop the `cosines' and `sines' that we now recognize as  the key components of the `square' CS decomposition, a KAK decomposition, and the general CSD, a $\kak$ decomposition.

\subsection{1978, Helgason}  
Reinforcing his 1962 work \cite{Helgason1962}, in his 1978 book Helgason provides more details on the KAK decomposition. In particular, he states the KAK decomposition as an explicit Theorem in his book. The development from what Cartan wrote \eqref{eq:cartanfrakp} starts with \cite[p.247, Lemma 6.3]{Helgason1978}, and then the statement of the KAK decomposition as a theorem may then be found in \cite[p.249, Theorem 6.7]{Helgason1978}.

\par This entire sequence is in a chapter entitled `Decomposition of Symmetric Spaces', in a Section titled ``Rank of Symmetric Spaces," which would still evade a modern search engine's attempts to try to find what would be familiar to modern applications.

\subsection{1978, Flensted-Jensen}
The major development historically for this paper is Theorem 4.1 in the 1978 paper by Flensted-Jensen \cite{flensted1978spherical} which introduces what he calls the generalized Cartan decomposition following the historical, if not clearly correct, lead of giving Cartan
credit for writing the KAK decomposition. 

\par Terms such as ``$\kak$ decomposition" and ``double coset decomposition" expand upon Flensted-Jensen's original meaning which focused exclusively on the noncompact cases. The term ``generalized Cartan" now appears to refer to even extensions to compact cases as well \cite{Kobayashi}.

\subsection{1982, Stewart}
Paige and Wei \cite{paige1994history} give the credit to G. Stewart for having taken the CS decomposition to its proper place, as a first class decomposition with a name. The quoted paragraph in the Introduction of this paper and nearby paragraphs in \cite{paige1994history} explain these developments. As described in \cite{paige1994history}, Stewart first used the name ``CS decomposition" in 1982 when presenting \cite{stewart1983method}. As a printed paper, the CSD first appeared in \cite{stewart1982computing}.

\subsection{1988, Terras}
Audrey Terras captures on page 23 of her book on `Harmonic Analysis on Symmetric Spaces' \cite{terras1988harmonic} the various threads connecting the KAK decomposition, generalized Euler angles, and the SVD in different areas of study (Figure \ref{fig:terrassvd}).

\begin{figure}[h]
    \centering
    \includegraphics[width=4.8in]{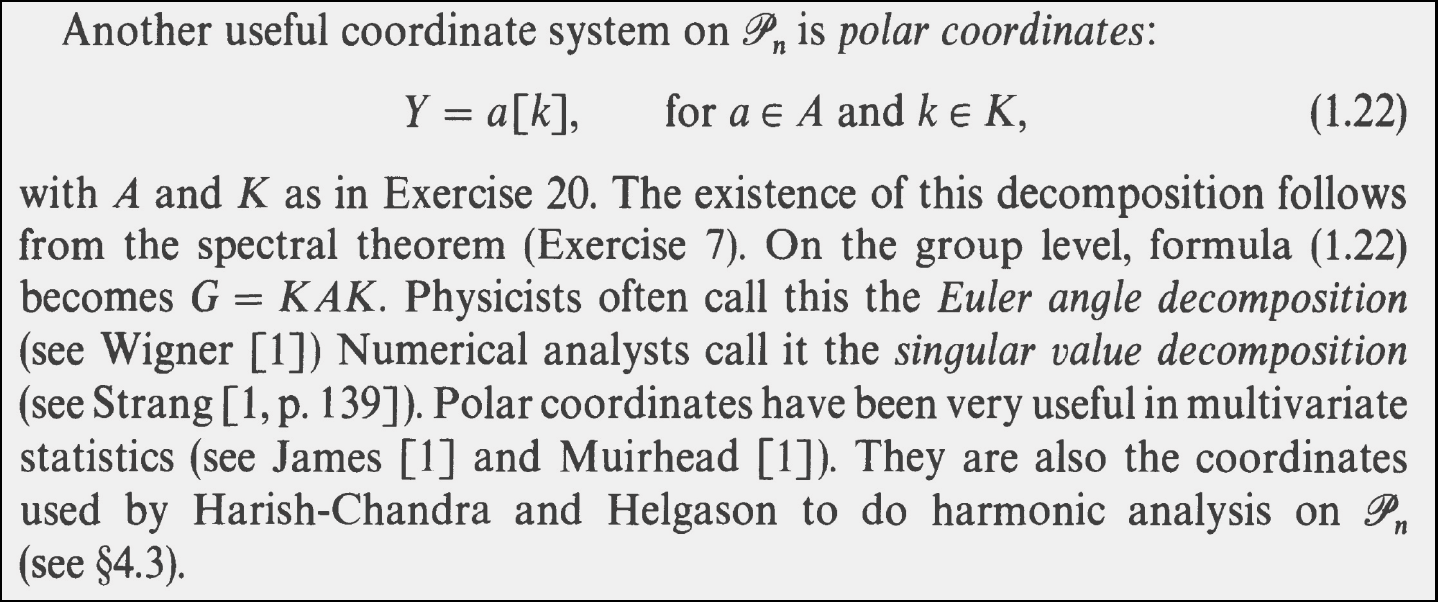}
    \caption{Terras in Exercise 23 of \cite[p.23]{terras1988harmonic} connects the symmetric positive definite eigenproblem, the polar decomposition, and the SVD.} 
    \label{fig:terrassvd}
\end{figure}

\subsection{2000s, Quantum Computing}
In a series of papers in the area of quantum computing such as \cite{bullock2004note,khaneja2001cartan,tucci2005introduction} one finds mention of the Cartan decomposition most especially in the context of the CS decomposition, and in one case \cite{fuehr2018note} the ODO decomposition ($\fact{1}$).

\subsection{Numerical Linear Algebra}
We have already discussed the SVD and the CS decomposition at length, as well as a few other decompositions. We also want to bring attention to the following papers on matrix factorizations that one way or another have perhaps without even knowing it brought what Cartan began as theory to useful developed practice: Angelika Bunse-Gerstner \cite{bunse1992chart}, Heike Fa{\ss}bender  \cite{fassbender2005several}, Nick Higham \cite{higham2003j}, Khakim Ikramov \cite{fassbender2005several,ikramov2012takagi}, D. Steven Mackey, Niloufer Mackey, and Fran\c{c}oise Tisseur \cite{mackey2003structured,mackey2005structured}, Volker Mehrmann \cite{benner2000cholesky,bunse1992chart}, Paul Van Dooren \cite{grimme1996model}, Hongguo Xu \cite{xu2003svd} and so many more. 

\subsection{2021, Wikipedia}
\par The Cartan decomposition currently has its own page in Wikipedia, and we draw attention to the \href{https://en.wikipedia.org/wiki/Cartan_decomposition#Cartan_decomposition_on_the_Lie_group_level}{\color{blue}paragraph} discussing the Cartan decomposition on the Lie Group Level. We note again that these are called refinements, and it seems likely that only $\mathfrak{p} = \bigcup_{k\in K} \text{Ad}(k)\mathfrak{a}$ appears in Cartan's work itself. As evidence of the obscurity of the generalized Cartan decomposition of Flensted-Jensen, there is no Wikipedia article at this time, though we expect it will appear shortly after this paper appears.

\addtocontents{toc}{\protect\setcounter{tocdepth}{0}}
\section{Acknowledgements}
We thank David Vogan, Sigurður Helgason, and Pavel Etingof for so much about the history, notation, and ideas connecting the Cartan decomposition and generalized Cartan decomposition to matrix factorizations. We thank Robert Gilmore for interesting conversations at Drexel University back in 2013. So much divides the applied and the pure, but so much more can unify. This work represents a bridge. We also thank David Sanders and Bernie Wang for comments on early versions of this manuscript.

\par We thank NSF grants OAC-1835443, SII-2029670, and PHY-2021825 for financial support.

\appendix
\addtocontents{toc}{\protect\setcounter{tocdepth}{1}}

\section{Lie algebras of classical Lie groups}\label{app:liealgebras}
In this section, we list the Lie algebras of the classical Lie groups. We also demonstrate how to create an element of those Lie algebras and Lie groups in the programming language Julia. Each code produces matrices \verb|g| and \verb|G| which are the matrices from listed Lie algebra and the corresponding Lie group.
\par The Lie group element \verb|G| can be obtained by applying the matrix exponential function on \verb|g|. In fact, the matrix exponential will create elements in the connected identity component of the Lie group, e.g., if you apply the exponential function to $\verb|g|\in\mathfrak{o}(p, q)$ then the resulting matrix is always in $\text{SO}(p, q)$. However one can always multiply $\pm 1$ or a complex unit in complex cases to recover the whole Lie group.

\subsection{General linear groups $\Gl_\beta(n)$}
The Lie algebras of general linear groups $\Gl_\beta(n)$, $\beta = 1,2,4$ are just the set of $n\times n$ matrices, without the invertibility restriction. 
\begin{gather}
    \mathfrak{gl}(n, \mathbb{R}) = \{n\times n\text{ real matrices}\} \\
    \mathfrak{gl}(n, \mathbb{C}) = \{n\times n\text{ complex matrices}\} \\
    \mathfrak{gl}(n, \mathbb{H}) = \{n\times n\text{ quaternionic matrices}\} 
\end{gather}

In many Lie theory textbooks the symbol $\mathfrak{u}^*(2n)$, the complexified $\mathfrak{gl}(n, \mathbb{H})$, is frequently used. With our $\complexify{\,\,\cdot\,\,}$ notation,
\begin{equation}
    \complexify{\mathfrak{gl}(n, \mathbb{H})} = \mathfrak{u}^*(2n) = \bigg\{\twotworr{a}{b}{-\overline{b}}{\overline{a}}\bigg| a, b\in\mathbb{C}^{n\times n}\bigg\}.
\end{equation}
For $\mathfrak{gl}(n, \mathbb{H})$ and $\mathfrak{u}^*(2n)$, we demonstrate both cases (a quaternion matrix and complexified matrix) of creating matrices in the Lie algebra. 

\begin{table}[h]
    \centering
    \small
    \begin{tabular}{c|l}
    Lie algebra & Julia code creating matrices \verb|g| $\in$ Lie algebra, \verb|G| $\in$ Lie Group\\ \hline
    $\mathfrak{gl}(n, \mathbb{R})$ & \verb|g = randn(n,n); G = exp(g)| \\ \hdashline
    $\mathfrak{gl}(n, \mathbb{C})$ & \verb|g = randn(n,n) + im * randn(n,n); G = exp(g)| \\ \hdashline
    $\mathfrak{gl}(n, \mathbb{H})$  & \verb|using Quaternions;| \\
    & \verb|g = [Quaternion(randn(4)...) for i in 1:n, j in 1:n];| \\ 
    & \verb|G = exp(g)| \\ \hdashline
    $\mathfrak{u}^*(2n)$ & \verb|a = randn(n,n) + im * randn(n,n);| \\
    & \verb|b = randn(n,n) + im * randn(n,n);|\\
    & \verb|g = [a b; -conj(b) conj(a)]; G = exp(g)| 
    \end{tabular}
\end{table}

\subsection{Unitary groups $\Un_\beta(n)$}
The Lie algebras of unitary groups $\beta = 1,2,4$ are well-known skew-symmetric, skew-Hermitian and quaternionic skew-Hermitian matrices, respectively. The Lie algebra of $\un{n, \mathbb{H}}$ is often denoted by $\mathfrak{sp}(n)$.
\begin{gather}
    \mathfrak{o}(n) = \{g\in\mathbb{R}^{n\times n}| g+g^T = 0\} \\
    \mathfrak{u}(n) = \{g\in\mathbb{C}^{n\times n}| g+g^H = 0\} \\
    \mathfrak{u}(n, \mathbb{H}) = \{g\in\mathbb{H}^{n\times n}| g+g^D = 0\} 
\end{gather}

\begin{table}[h]
    \centering
    \small
    \begin{tabular}{c|l}
    Lie algebra & Julia code creating matrices \verb|g| $\in$ Lie algebra, \verb|G| $\in$ Lie Group\\ \hline
    $\mathfrak{o}(n)$ & \verb|g = randn(n,n); g -= g'; G = exp(g)| \\ \hdashline
    $\mathfrak{u}(n)$ & \verb|g = randn(n,n) + im * randn(n,n); g -= g'; G = exp(g)| \\ \hdashline
    $\mathfrak{u}(n, \mathbb{H})$  & \verb|using Quaternions;| \\
    & \verb|g = [Quaternion(randn(4)...) for i in 1:n, j in 1:n];| \\ 
    & \verb|g -= g'; G = exp(g)| 
    \end{tabular}
\end{table}

\subsection{Indefinite unitary groups $\Un_\beta(p,q)$}
The Lie algebras of $\Un_\beta(p,q)$, $\beta=1,2,4$ are denoted by $\mathfrak{o}(p, q)$, $\mathfrak{u}(p, q)$, and $\mathfrak{u}(p, q, \mathbb{H})$, respectively. (In the literature, the symbol $\mathfrak{sp}(p,q)$ might be used.) 
\begin{gather}
    \mathfrak{o}(p, q) = \bigg\{\twotwo{a}{b}{b^T}{c}\bigg|a=-a^T, c=-c^T, a\in\mathbb{R}^{p\times p},  b\in\mathbb{R}^{p\times q}, c\in\mathbb{R}^{q\times q}\bigg\},\\
    \mathfrak{u}(p, q) = \bigg\{\twotwo{a}{b}{b^H}{c}\bigg|a = -a^H, c = -c^H, a\in\mathbb{C}^{p\times p},  b\in\mathbb{C}^{p\times q}, c\in\mathbb{C}^{q\times q}\bigg\},\\
    \mathfrak{u}(p, q, \mathbb{H}) = \bigg\{\twotwo{a}{b}{b^D}{c}\bigg|a = -a^D, c = -c^D, a\in\mathbb{H}^{p\times p},  b\in\mathbb{H}^{p\times q}, c\in\mathbb{H}^{q\times q}\bigg\}.
\end{gather}

\begin{table}[h]
    \centering
    \small
    \begin{tabular}{c|l}
    Lie algebra & Julia code creating matrices \verb|g| $\in$ Lie algebra, \verb|G| $\in$ Lie Group\\ \hline
    $\mathfrak{o}(p, q)$ & \verb|a = randn(p,p); b = randn(p,q); c = randn(q,q);| \\ 
    & \verb|a -= a'; c -= c'; g = [a b; b' c]; G = exp(g)| \\ \hdashline
    $\mathfrak{u}(p, q)$ & \verb|a = randn(p,p) + im * randn(p,p); a -= a';| \\ 
    & \verb|b = randn(p,q) + im * randn(p,q);| \\
    & \verb|c = randn(q,q) + im * randn(q,q); c -= c';| \\
    & \verb|g = [a b; b' c]; G = exp(g)| \\\hdashline
    $\mathfrak{u}(p, q, \mathbb{H})$  & \verb|using Quaternions;| \\
    & \verb|a = [Quaternion(randn(4)...) for i in 1:p, j in 1:p];| \\ 
    & \verb|b = [Quaternion(randn(4)...) for i in 1:p, j in 1:q];| \\ 
    & \verb|c = [Quaternion(randn(4)...) for i in 1:q, j in 1:q];| \\ 
    & \verb|a -= a'; c -= c';| \\
    & \verb|g = [a b; b' c]; G = exp(g)| 
    \end{tabular}
\end{table}

\subsection{Symplectic groups $\Symp_\beta(2n)$}
The Lie algebras of $\symp{2n, \mathbb{R}}$ and $\symp{2n, \mathbb{C}}$ are denoted by the symbols $\mathfrak{sp}(2n, \mathbb{R})$ and $\mathfrak{sp}(2n, \mathbb{C})$. 
\begin{gather}
    \mathfrak{sp}(2n, \mathbb{R}) = \bigg\{\twotwo{a}{b}{c}{-a^T}\bigg|b=b^T, c=c^T, a, b, c\in\mathbb{R}^{n\times n}\bigg\}\\
    \mathfrak{sp}(2n, \mathbb{C}) = \bigg\{\twotwo{a}{b}{c}{-a^T}\bigg|b=b^T, c=c^T, a, b, c\in\mathbb{C}^{n\times n}\bigg\}
\end{gather}
Moreover the conjugate symplectic group (isomorphic to $\un{n, n}$) has the Lie algebra
\begin{equation}
    \mathfrak{sp}^*(2n, \mathbb{C}) = \bigg\{\twotwo{a}{b}{c}{-a^H}\bigg|b=b^H, c=c^H, a, b, c\in\mathbb{C}^{n\times n}\bigg\}. 
\end{equation}

\begin{table}[h]
    \centering
    \small
    \begin{tabular}{c|l}
    Lie algebra & Julia code creating matrices \verb|g| $\in$ Lie algebra, \verb|G| $\in$ Lie Group\\ \hline
    $\mathfrak{sp}(2n, \mathbb{R})$ & \verb|a = randn(n,n); b = randn(n,n); c = randn(n,n);| \\ 
    & \verb|b += b'; c += c'; g = [a b; c -a']; G = exp(g)| \\ \hdashline
    $\mathfrak{sp}(2n, \mathbb{C})$ & \verb|a = randn(n,n) + im * randn(n,n);| \\
    & \verb|b = randn(n,n) + im * randn(n,n); b += transpose(b);| \\
    & \verb|c = randn(n,n) + im * randn(n,n); c += transpose(c);| \\
    & \verb|g = [a b; c -transpose(a)]; G = exp(g)| \\ \hdashline
    $\mathfrak{sp}^*(2n, \mathbb{C})$ & \verb|a = randn(n,n) + im * randn(n,n);| \\
    & \verb|b = randn(n,n) + im * randn(n,n); b += b';| \\
    & \verb|c = randn(n,n) + im * randn(n,n); c += c';| \\
    & \verb|g = [a b; c -a']; G = exp(g)| 
    \end{tabular}
\end{table}

\subsection{Orthogonal groups $\Ortho_\beta(n)$}
The Lie algebra of the complex orthogonal group $\ortho{n, \mathbb{C}}$ is straightforward,
\begin{equation}
    \mathfrak{o}(n, \mathbb{C}) = \{g\in\mathbb{C}^{n\times n}| g + g^T = 0\}.
\end{equation}
The Lie algebra of $\Ortho_\eta(n, \mathbb{H})$ is the set of all \textit{$\eta$-skew-Hermitian matrices}, 
\begin{equation}
    \mathfrak{o}_\eta(n, \mathbb{H}) := \{g\in\mathbb{H}^{n\times n}|g+g^{D_\eta} = 0\}.
\end{equation}

\begin{table}[h]
    \centering
    \small
    \begin{tabular}{c|l}
    Lie algebra & Julia code creating matrices \verb|g| $\in$ Lie algebra, \verb|G| $\in$ Lie Group\\ \hline
    $\mathfrak{o}(n, \mathbb{C})$ & \verb|g = randn(n,n) + im * randn(n,n);| \\
    & \verb|g -= transpose(g); G = exp(g)| \\ \hdashline
    $\mathfrak{o}(n, \mathbb{H})$ & \verb|using Quaternions;| \\
    & \verb|g = [Quaternion(randn(4)...) for i in 1:n, j in 1:n];| \\
    & \verb|g += Quaternion(0,0,1,0) * g' * Quaternion(0,0,1,0);| \\
    & \verb|G = exp(g)|
    \end{tabular}
\end{table}

\section{The list of Lie algebra involution $\tau$}\label{app:involutions}
In Tables \ref{tab:involutionlist1} to \ref{tab:involutionlist5}, we list the involutions used in our computation of 53 matrix factorizations. Note that these involutions are applied to the Lie algebra $\mathfrak{g}$ to obtain $\mathfrak{k}_\tau$ and $\mathfrak{p}_\tau$. The Cartan involutions could be found in rows marked with \colorbox{brown!20}{brown} since the brown represents the KAK decomposition. 

\par For compact cases ($\fact{1}$ to $\fact{6}$), there is no Cartan involution. The two involutions in each cell serve as $\sigma$ and $\tau$ respectively.  

\par We note that this is not the exhaustive list of non-Cartan involutions $\tau$ on each $\mathfrak{g}$. However we have included every non-Cartan involutions such that the induced subgroup $K_\tau$ has the explicit representation of Lie groups appearing in Table \ref{tab:translation}. We refer to Gilmore's textbook, Table 9.7 of \cite{gilmore2012lie} for the complete list of the involutions (real forms, up to isomorphism) associated with the classical Lie groups. 

\begin{table}[h]
{\tabulinesep=1mm
\begin{tabu}{|c|c|c|c|}
\hline

& $\mathfrak{o}(m)$ & $\mathfrak{u}(m)$ & $\mathfrak{u}(m, \mathbb{H})$ \\ \hline

$\fact{1}$ & & $-X^T$ & $-X^{D_i}$ \\ \hline

$\fact{2}$ & $-J_n X J_n$ & $-J_n \overline{X} J_n$ & \\ \hline

$\fact{3}$ & & $-X^T$, $-J_n X J_n$ & \\ \hline

$\fact{4}$ & $I_{p,q}XI_{p,q}$, $I_{r,s}XI_{r,s}$ & $I_{p,q}XI_{p,q}$, $I_{r,s}XI_{r,s}$ & $I_{p,q}XI_{p,q}$, $I_{r,s}XI_{r,s}$ \\ \hline

$\fact{5}$ & & $-X^T$, $I_{p,q}XI_{p,q}$ & $-X^{D_i}$, $I_{p,q}XI_{p,q}$ \\ \hline

$\fact{6}$ & $-J_n X J_n$, $I_{2p, 2q}X I_{2p, 2q}$ & $-J_n \overline{X} J_n$, $I_{2p, 2q}X I_{2p, 2q}$ & \\ \hline

\end{tabu}}
\caption{List of the involutions $\tau$ for Lie algebras of $\Un_\beta(m)$ ($m = n\text{ or }2n$).}
\label{tab:involutionlist1}
\end{table}

\begin{table}[h]
{\tabulinesep=1mm
\begin{tabu}{|c|c|c|c|}
\hline
& $\mathfrak{gl}(m, \mathbb{R})$ & $\mathfrak{gl}(m, \mathbb{C})$ & $\mathfrak{gl}(m, \mathbb{H})$ \\ \hline

\cellcolor{brown!20}$\fact{7}$ & $-X^T$ & $-X^H$ & $-X^D$ \\ \hline

$\fact{8}$ & $I_{p,q}X I_{p,q}$ & $I_{p,q}X I_{p,q}$ & $I_{p,q}X I_{p,q}$ \\ \hline

$\fact{9}$ & $-I_{p,q}X^TI_{p,q}$ & $-I_{p,q}X^HI_{p,q}$ & $-I_{p,q}X^DI_{p,q}$ \\ \hline

$\fact{10}$ & $J_nX^TJ_n$ & $J_nX^TJ_n$ &  \\ \hline

$\fact{11}$ & $-J_nXJ_n$ & $-J_nXJ_n$ & \\ \hline

$\fact{12}$ & & $\overline{X}$ & $-iXi$ \\ \hline

$\fact{13}$ & & $-X^T$ & $-X^{D_j}$ \\ \hline

\end{tabu}}
\caption{List of the involutions $\tau$ for Lie algebras of $\Gl_\beta(m)$ ($m = n\text{ or }2n$).}
\label{tab:involutionlist2}
\end{table}

\begin{table}[h]
{\tabulinesep=1mm
\begin{tabu}{|c|c|c|}
\hline
& $\mathfrak{sp}(2n, \mathbb{R})$ & $\mathfrak{sp}(2n, \mathbb{C})$ \\ \hline

\cellcolor{brown!20}$\fact{14}$ & $-X^T$ & $-X^H$ \\ \hline

$\fact{15}$ & $I_{n,n}XI_{n,n}$ & $I_{n,n}XI_{n,n}$ \\ \hline

$\fact{16}$ & $\stwotwo{I_{p,q}}{}{}{I_{p,q}}X\stwotwo{I_{p,q}}{}{}{I_{p,q}}$ & $\stwotwo{I_{p,q}}{}{}{I_{p,q}}X\stwotwo{I_{p,q}}{}{}{I_{p,q}}$ \\ \hline

$\fact{17}$ & & $\overline{X}$ \\ \hline

\end{tabu}}
\caption{List of the involutions $\tau$ for Lie algebras of $\Symp_\beta(2n)$.}
\label{tab:involutionlist3}
\end{table}

\begin{table}[H]
{\tabulinesep=1mm
\begin{tabu}{|c|c|c|c|}
\hline
& $\mathfrak{o}(p, q)$ & $\mathfrak{u}(p, q)$ & $\mathfrak{u}(p, q, \mathbb{H})$ \\ \hline

\cellcolor{brown!20}$\fact{18}$ & $I_{p,q}XI_{p,q}$ & $I_{p,q}XI_{p,q}$ & $I_{p,q}XI_{p,q}$ \\ \hline


$\fact{19}$ & \multicolumn{3}{c|}{$\stwotwo{I_{p_1, p_2}}{}{}{I_{q_1, q_2}}X\stwotwo{I_{p_1, p_2}}{}{}{I_{q_1, q_2}}$}  \\ \hline

$\fact{20}$ & $-\stwotwo{J_p}{}{}{J_q}X\stwotwo{J_p}{}{}{J_q}$ & $-\stwotwo{J_p}{}{}{J_q}\overline{X}\stwotwo{J_p}{}{}{J_q}$ & \\ \hline

$\fact{21}$ & $-J_nXJ_n$ & $-J_n\overline{X}J_n$ & \\ \hline

$\fact{22}$ & & $\overline{X}$ & $-iXi$ \\ \hline

\end{tabu}}
\caption{List of the involutions $\tau$ for Lie algebras of $\Un_\beta(p, q)$ (or $\Un_\beta(2p, 2q)$).}
\label{tab:involutionlist4}
\end{table}

\begin{table}[H]
{\tabulinesep=1mm
\begin{tabu}{|c|c|c|}
\hline
& $\mathfrak{o}(m, \mathbb{C})$ & $\mathfrak{o}(m, \mathbb{H})$ \\ \hline

\cellcolor{brown!20}$\fact{23}$ & $\overline{X}$ & $-iXi$ \\ \hline

$\fact{24}$ & $I_{p,q}XI_{p,q}$ & $I_{p,q}XI_{p,q}$ \\ \hline

$\fact{25}$ & $-J_n\overline{X}J_n$ & \\ \hline

\end{tabu}}
\caption{List of the involutions $\tau$ for Lie algebras of $\Ortho_\beta(m)$ ($m = n\text{ or }2n$).}
\label{tab:involutionlist5}
\end{table}

\bibliographystyle{siam}
\bibliography{bibliography.bib}

\end{document}